\documentclass[letterpaper]{amsart}

\usepackage[letterpaper]{geometry}
\usepackage{amsmath}
\usepackage{amsthm}
\usepackage{amsfonts}
\usepackage{thmtools}

\usepackage{thm-restate}
\usepackage{hyperref} 
\usepackage[]{ytableau}
\usepackage[mathscr]{euscript}
\usepackage{enumitem}
\usepackage{stmaryrd}
\usepackage[all, arc, curve, frame]{xy}

\makeatletter
\def\namedlabel#1#2{\begingroup
    #2%
    \def\@currentlabel{#2}%
    \phantomsection\label{#1}\endgroup
}
\makeatother

\numberwithin{equation}{section}

\DeclareMathOperator{\Id}{Id}

\DeclareMathOperator{\Tr}{Tr}
\DeclareMathOperator{\Var}{Var}
\DeclareMathOperator{\semicircle}{sc}

\title{Bulk Universality for Generalized Wigner Matrices With Few Moments}
\author{Amol Aggarwal}

\begin{document}

\newtheorem{thm}{Theorem}[section]
\newtheorem{prop}[thm]{Proposition}
\newtheorem{lem}[thm]{Lemma}
\newtheorem{cor}[thm]{Corollary}
\newtheorem{conj}[thm]{Conjecture}
\newtheorem{que}[thm]{Question}
\newtheorem{exa}[thm]{Example}
\theoremstyle{remark}
\newtheorem{rem}[thm]{Remark}
\theoremstyle{definition}
\newtheorem{definition}[thm]{Definition}
\newtheorem{assumption}[thm]{Assumption}
\newtheorem{sampling}[thm]{Sampling}

\begin{abstract}

In this paper we consider $N \times N$ real generalized Wigner matrices whose entries are only assumed to have finite $(2 + \varepsilon)$-th moment for some fixed, but arbitrarily small, $\varepsilon > 0$. We show that the Stieltjes transforms $m_N (z)$ of these matrices satisfy a weak local semicircle law on the nearly smallest possible scale, when $\eta = \Im (z)$ is almost of order $N^{-1}$. As a consequence, we establish bulk universality for local spectral statistics of these matrices at fixed energy levels, both in terms of eigenvalue gap distributions and correlation functions, meaning that these statistics converge to those of the Gaussian Orthogonal Ensemble (GOE) in the large $N$ limit. 

\end{abstract}

\maketitle

\tableofcontents

\section{Introduction}

\label{Introduction}

Since the seminal work of Wigner \cite{CVBMWID} over sixty years ago, the spectral analysis of random matrices has been a topic of intense study. A central phenomenon that has guided significant effort in this field has been that of \emph{universality}, also called the \emph{Wigner-Dyson-Mehta conjecture}. This approximately states (see Conjecture 1.2.1 and Conjecture 1.2.2 of \cite{RM}) that the bulk local spectral statistics of an $N \times N$ real symmetric (or complex Hermitian) Wigner matrix should become independent of the explicit laws of its entries as $N$ tends to $\infty$. 

Over the past decade, this conjecture has seen remarkable progress. In particular, the Wigner-Dyson-Mehta conjecture has been proved \cite{FEUGM, SSGESEE, UM, BUMSD, URMLRF, LRFULSRM, BUCM, BUGM, BUSM, FEUM, RCMULS, ULES, UCM} for Wigner matrices whose entries are ``restrained from being too large.'' In the past \cite{UM, BUMSD, URMLRF, LRFULSRM, BUGM, ULES}, this condition had been typically quantified by imposing that the laws of its entries exhibit some type of subexponential decay. This was later \cite{FEUGM, SSGESEE, BUCM, FEUM, RCMULS, UCM} replaced by the less stringent constraint that the entries have finite $C$-th moment for some sufficiently large constant $C > 0$; until now, the smallest value of $C$ one could take had been $4 + \varepsilon$, for any $\varepsilon > 0$ \cite{SSGESEE, FEUM}. 

Although the known proofs of the Wigner-Dyson-Mehta conjecture appear to rely quite strongly on such growth assumptions, it is widely believed that these constraints are irrelevant. Namely, the bulk local spectral statistics of a Wigner matrix should be universal only under the assumption that its entries have finite $(2 + \varepsilon)$-moment for some fixed, but arbitrary small, $\varepsilon > 0$ (for our model, having finite second moment will not suffice; see Remark \ref{second}). 

The purpose of this paper is to prove this statement, which we do by following and extending upon parts of what is known as the ``three-step	strategy'' for establishing bulk universality in random matrices, as set forth in the papers \cite{UM, URMLRF, ERLRM, LRFULSRM, DRM}. Before explaining this in more detail, we will explain the model and our results in Section \ref{Model}. In Section \ref{Outline1} we provide some context for our results. In Section \ref{LocalCircleLaw} we state a version of the heavy-tailed local semicircle law, which is the main new estimate needed to prove universality in the low-moment setting.

\subsection{The Model and Results}

\label{Model}

We begin this section by defining the random matrix model that we will study.

\begin{definition}

\label{momentassumption}

For each positive integer $N$, let $\textbf{H} = \textbf{H}_N = \{ h_{ij} \} = \{ h_{i, j}^{(N)} \}$ be a real symmetric $N \times N$ random matrix whose entries are centered, mutually independent random variables, subject to the symmetry constraint $h_{ij} = h_{ji}$. We call the set of matrices $\{ \textbf{H}_N \}$ a \emph{family of generalized Wigner matrices} if there exist constants $0 < \varepsilon < 1$; $0 < c_1 < 1 < C_1$; and $C_2 > 1$ (all independent of $N$) satisfying the following three assumptions.

\begin{enumerate}

\item[\namedlabel{generalized}{\textbf{A1}}\textbf{.}]{ Denote $s_{ij} = \mathbb{E} \big[ |h_{ij}|^2 \big]$. Then, $c_1 < N s_{ij} < C_1$ for all $i, j$.}

\item[\namedlabel{stochastic}{\textbf{A2}}\textbf{.}]{ For each $i$, set $t_i = \sum_{j = 1}^N s_{ij} - 1$; then, $|t_i| < C_1 N^{- \varepsilon}$ for each $i$. }

\item[\namedlabel{moments}{\textbf{A3}}\textbf{.}]{ For each $i, j$, we have that $\mathbb{E} \big[ |h_{ij} \sqrt{N}|^{2 + \varepsilon} \big] < C_2$.  }

\end{enumerate}  	

\noindent We refer to each individual $\textbf{H} = \textbf{H}_N$ as a \emph{generalized Wigner matrix}. 

\end{definition}

When $\Var h_{ij} = N^{-1}$ for each $i, j$, generalized Wigner matrices become \emph{Wigner matrices}, which were first analyzed in \cite{CVBMWID}. In that work, Wigner studied the large $N$ limiting profile of the \emph{empirical spectral distribution} of $\textbf{H}$, defined by $\mu_{\textbf{H}} = N^{-1} \sum_{j = 1}^N \delta_{\lambda_j}$, where $\lambda_1, \lambda_2, \ldots , \lambda_N$ denote the eigenvalues of $\textbf{H}$. He showed, if one assumes that all moments of $|h_{ij} \sqrt{N}|$ are finite, then $\mu_{\textbf{H}}$ converges weakly to the \emph{semicircle law} 
\begin{flalign}
\label{rhodefinition}
\rho_{\semicircle} (x) = (2 \pi)^{-1} \textbf{1}_{|x| < 2} \sqrt{4 - x^2},
\end{flalign}

\noindent as $N$ tends to $\infty$. This is an example of the convergence of \emph{global spectral statistics} of the random matrix $\textbf{H}$ to a deterministic limit shape. 

Also of interest, and in fact the original impetus for Wigner to initiate his study on random matrices, are the \emph{local spectral statistics} of $\textbf{H}$; these concern the behavior of (nearly) neighboring eigenvalues of $\textbf{H}$ close to a fixed \emph{energy level} $E \in \mathbb{R}$. There are two ways in which this behavior is typically quantified. 

The first is in terms of the joint distribution of the (normalized) gaps $\{ N (\lambda_i - \lambda_{i + j_r}) \}_{1 \le r \le k}$ of the eigenvalues; here, $j_1, j_2, \ldots , j_k$ are bounded independently of $N$ and $i$ can grow linearly with $N$. In fact, we will be interested in these statistics in the \emph{bulk}, meaning that we take $i \in [\kappa N, (1 - \kappa) N]$ for some fixed $\kappa > 0$ independent of $N$; this corresponds to imposing that the energy level $E$ be inside the interval $(-2, 2)$ and uniformly bounded away from its endpoints.  

The second is in terms of the correlation functions of $\textbf{H}$, which are defined as follows. 

\begin{definition}

\label{correlation}

Let $N$ be a positive integer and $\textbf{H}$ be an $N \times N$ real symmetric random matrix. Denote by $p_{\textbf{H}}^{(N)} (\lambda_1, \lambda_2, \ldots , \lambda_N)$ the joint eigenvalue distribution of $\textbf{H}$. 

For each integer $k \in [1, N]$, define the \emph{$k$-th correlation function} of $\textbf{H}$ by 
\begin{flalign*}
p_{\textbf{H}}^{(k)} (x_1, x_2, \ldots , x_k) = \displaystyle\int_{\mathbb{R}^{N - k}} p_{\textbf{H}}^{(N)} (x_1, x_2, \ldots , x_k, y_{k + 1}, y_{k + 2}, \ldots , y_N) \displaystyle\prod_{j = k + 1}^N d y_j. 
\end{flalign*} 
\end{definition}

One form of the Wigner-Dyson-Mehta conjecture essentially states that, in the bulk of the spectrum, the local gap statistics and correlation functions of a Wigner matrix should be independent of the explicit laws of the matrix entries satisfying assumption \ref{moments}, in the large $N$ limit. 

\begin{rem}

\label{second}

Observe that this universality can become false if we remove assumption \ref{moments}. Indeed, suppose that the $h_{ij}$ are independent, identically distributed random variables (up to the symmetry constraint $h_{ij} = h_{ji}$) with $h_{ij} \in \{ -1, 0, 1 \}$, equal to $-1$ and $1$ each with probability $(2N)^{-1}$ and equal to $0$ otherwise. Then $\Var h_{ij} = N^{-1}$, so $\textbf{H}$ is a Wigner matrix. 

However, a given row contains no nonzero entries with probability $(1 - N^{-1})^N \ge 1 / 4$. Therefore we expect with high probability to see, for example, at least $N / 5$ rows $\textbf{H}$ whose entries are all equal to $0$. Thus, with high probability, $\textbf{H}$ has the eigenvalue $0$ with very large multiplicity. This violates universality of both the gap statistics and correlation functions near $0$; it also violates the macroscopic Wigner semicircle law around $0$. 

\end{rem}

\begin{rem}

\label{smallmoments}

On a different note, one can choose a more restrictive family of matrix entries $h_{ij}$, which do not satisfy assumption \ref{moments}, and still expect universality to hold. For instance, let $X$ be a random variable with variance $1$ (but infinite $(2 + \varepsilon)$-th moment for any $\varepsilon > 0$), take $\binom{N + 1}{2}$ mutually independent copies $\{ X_{ij} \}_{1 \le i \le j \le N}$ of $X$, and set $h_{ij} = h_{ji} = N^{-1 / 2} X_{ij}$ for each $1 \le i, j \le N$. Then it is plausible that the local spectral statistics of the resulting matrix $\textbf{H}$ are universal.

In fact, one could also consider matrices whose entries have infinite variance; this leads to the study of \emph{L\'{e}vy matrices} \cite{SHTRM, TM, LSLTM}. For these matrices, the semicircle law no longer governs the limit shape of the empirical spectral density \cite{SHTRM}. However, it is still predicted \cite{LSLTM} that the local statistics of these matrices should be universal at sufficiently small energy levels $E \in \mathbb{R}$. We will not pursue this here but refer to the papers \cite{DSERM, LDEHTRM} for partial progress in that direction. 
\end{rem}

In particular, one can consider the generalized Wigner matrices given by the \emph{Gaussian Orthogonal Ensemble} (GOE). This is defined to be the $N \times N$ real symmetric random matrix $\textbf{GOE}_N = \{ g_{ij} \}$, where $g_{ij}$ is a Gaussian random variable with variance $2 N^{-1}$ if $i = j$ and $N^{-1}$ otherwise. This particular ensemble of matrices is exactly solvable through the framework of orthogonal polynomials and Pfaffian point processes, and the local gap statistics and correlation functions can be evaluated explicitly in the large $N$ limit; we will not state these results here, but they can be found in Chapter 6 of \cite{RM} or Chapter 3.9 of \cite{TRM}. 

That said, the Wigner-Dyson-Mehta conjecture can be rephrased by stating that the gap statistics and correlation functions of an $N \times N$ generalized Wigner matrix converge to those of $\textbf{GOE}_N$, as $N$ tends to $\infty$. This can be written more precisely as follows. 

\begin{definition}

\label{gapscorrelations} 

For each integer $N \ge 1$, let $\textbf{H} = \textbf{H}_N$ be an $N \times N$ real symmetric random matrix and let $i_N$ be a positive integer. We say that \emph{the gap statistics of $\textbf{\emph{H}}$ are universal near the $i$-th eigenvalue} if the following holds. Fix positive integers $k$ and $j_1, j_2, \ldots , j_k$. For any compactly supported smooth function $F \in \mathcal{C}_0^{\infty} (\mathbb{R}^k)$, we have that 
\begin{flalign}
\label{universality1}
\begin{aligned}
\displaystyle\lim_{N \rightarrow \infty} \bigg| & \mathbb{E}_{\textbf{H}} \Big[ F \big( N (\lambda_i - \lambda_{i + j_1}), N (\lambda_i - \lambda_{i + j_2}), \ldots , N (\lambda_i - \lambda_{i + j_k} )\big) \Big] \\
& - \mathbb{E}_{\textbf{GOE}_N} \Big[ F \big( N (\lambda_i - \lambda_{i + j_1}), N (\lambda_{i} - \lambda_{i + j_2}), \ldots , N (\lambda_i - \lambda_{i + j_k} ) \big) \Big]  \bigg| = 0. 
\end{aligned}
\end{flalign}

Furthermore, for a fixed real number $E \in \mathbb{R}$, we say that \emph{the correlation functions of $\textbf{\emph{H}}$ are universal at energy $E$} if the following holds. For any positive integer $k$ and any compactly supported smooth function $F \in \mathcal{C}_0^{\infty} (\mathbb{R}^k)$, we have that 
\begin{flalign}
\label{universality2}
\begin{aligned}
\displaystyle\lim_{N \rightarrow \infty} \Bigg| \displaystyle\int_{\mathbb{R}^k} & F (a_1, a_2, \ldots , a_k) \bigg( p_{\textbf{H}}^{(k)} \Big( E + \displaystyle\frac{a_1}{N \rho_{\semicircle} (E)}, E + \displaystyle\frac{a_2}{N \rho_{\semicircle} (E)}, \ldots , E + \displaystyle\frac{a_k}{N \rho_{\semicircle} (E)} \Big)  \\
& \qquad - p_{\textbf{GOE}_N}^{(k)} \Big( E + \displaystyle\frac{a_1}{N \rho_{\semicircle} (E)}, E + \displaystyle\frac{a_2}{N \rho_{\semicircle} (E)}, \ldots , E + \displaystyle\frac{a_k}{N \rho_{\semicircle} (E)} \Big) \bigg) \displaystyle\prod_{j = 1}^k d a_j \Bigg| = 0. 
\end{aligned}
\end{flalign}

\end{definition} 

The purpose of this paper is to establish the following two results, which establish that both the gap statistics and correlation functions of generalized Wigner matrices are universal in the bulk.

\begin{thm}

\label{gapsfunctions}

Fix real numbers $\varepsilon, \kappa, c_1, C_1, C_2 > 0$. Let $\{ \textbf{\emph{H}} = \textbf{\emph{H}}_N \}_{N \in \mathbb{Z}_{\ge 1}} $ denote a family of generalized Wigner matrices, as in Definition \ref{momentassumption} (with parameters $\varepsilon, c_1, C_1, C_2$), and let $i = i_N \in [\kappa N, (1 - \kappa) N]$ be a positive integer. Then, the gap statistics of $\textbf{\emph{H}}$ are universal near the $i$-th eigenvalue as in \eqref{universality1} of Definition \ref{gapscorrelations}. 

\end{thm}

\begin{thm}

\label{bulkfunctions}

Fix real numbers $\varepsilon, \kappa, c_1, C_1, C_2 > 0$ and a real number $E \in [\kappa - 2, 2 - \kappa]$. Let $\{ \textbf{\emph{H}} = \textbf{\emph{H}}_N \}_{N \in \mathbb{Z}_{\ge 1}} $ denote a family of generalized Wigner matrices, as in Definition \ref{momentassumption} (with parameters $\varepsilon, c_1, C_1, C_2$). Then, the correlation functions of $\textbf{\emph{H}}$ are universal at energy level $E$ as in \eqref{universality2} of Definition \ref{gapscorrelations}. 

\end{thm}

Observe that Theorem \ref{gapsfunctions} and Theorem \ref{bulkfunctions} above are only stated for real symmetric matrices. However, after minor modification, our methods and results should also apply to complex Hermitian random matrices (whose local statistics will instead converge to those of the GUE as $N$ tends to $\infty$); in order to avoid complicated notation later in the paper, we will not pursue this further. 

We conclude this section by mentioning that it is necessary to take the index $i$ and the energy level $E$ (from Theorem \ref{gapsfunctions} and Theorem \ref{bulkfunctions}, respectively) to be in the bulk of the spectrum. Indeed, it is possible for the \emph{edge} local spectral statistics of Wigner matrices satisfying assumption \ref{moments} to be non-universal \cite{CLEHTRM}. In fact, in \cite{CEUM} Lee and Yin showed that edge universality of Wigner matrices holds if and only if the $h_{ij} \sqrt{N}$ have finite weak fourth moment.

\subsection{Context} 

\label{Outline1}

In this section we provide some context for Theorem \ref{gapsfunctions} and Theorem \ref{bulkfunctions} by explaining their relationship with some previous results in the field. 

Although the local spectral statistics of the GUE (Gaussian Unitary Ensemble) and GOE were found explicitly by Mehta-Gaudin \cite{ODERM} and Mehta \cite{RM} in the early 1960s, the question of universality for Wigner matrices had seen few results until the work of Johansson \cite{ULSDCM} in 2001. In that paper Johansson considered Hermitian \emph{Gaussian divisible matrices}, that is, matrices of the form $\textbf{H} + t \textbf{GUE}_N$, where $\textbf{H}$ is an $N \times N$ Hermitian Wigner matrix, $\textbf{GUE}_N$ is an $N \times N$ independent GUE matrix, and $t$ is a constant of order $1$. Through asymptotic analysis of the Br\'{e}zin-Hikami identity, Johansson showed \cite{ULSDCM} that the correlation functions of these matrices are universal. 

There were two limitations to this method. The first is that the Br\'{e}zin-Hikami identity is only valid for complex Hermitian matrices and thus gave no results for real symmetric matrices. The second is that the Gaussian perturbation $t \textbf{GUE}_N$ happens to not be so immediately removed. 

Now these issues have been overcome through what is known as the \emph{three-step strategy} for establishing bulk universality in random matrices, developed in the papers \cite{UM, URMLRF, ERLRM, LRFULSRM, DRM, ULES}, almost 10 years after Johansson's work \cite{ULSDCM}. Since this is the route we will eventually follow, we briefly outline it below; for a detailed review of the method, we refer to the survey \cite{URM} or the more comprehensive book \cite{DRM}. 

\begin{enumerate}

\item{ \label{localcircle} The first step is to establish a \emph{local semicircle law} for the generalized Wigner matrix $\textbf{H}$, meaning that the spectral density of $\textbf{H}$ asymptotically follows that of the semicircle law \eqref{rhodefinition} on scales nearly of order $N^{-1}$. }

\item{ \label{perturb} The second step is to consider a perturbation $\textbf{H} + t \textbf{GOE}_N$ of the original random matrix $\textbf{H}$, where $t$ is small (optimally nearly of order $N^{-1}$). Using the local semicircle law from step \ref{localcircle}, one shows that the local statistics of the perturbed matrix are universal. }

\item{ \label{originalmatrix} The third step is to compare the local statistics of the original matrix $\textbf{H}$ and the perturbed matrix $\textbf{H} + t \textbf{GOE}_N$, and show that they are asymptotically the same if $t$ is small. }

\end{enumerate}

The first proofs of the local semicircle law appeared in the papers \cite{ERLRM, LSLCDRM}, although it has seen several improvements \cite{SSG, LSLGCRM, BUGM, LSLUMC, LSLM, CEUM} since then. 

The first proof of the second step appeared in \cite{UM} for complex Hermitian Wigner matrices, this time by combining the Br\'{e}zin-Hikami identity with the local semiciricle law. Later, however, through a very different and more analytic method, the second step was extended to real symmetric Wigner matrices \cite{URM, LRFULSRM}, under a slightly weaker topology than stated in \eqref{universality2}. The originally stringent rigidity conditions under which this universality could be proven were later weakened \cite{URMFTD, CLSM}. More recently \cite{FEUGM, FEUM}, the topology under which this universality held was strengthened to what was stated in \eqref{universality2} above.  

When the laws of the entries of the matrix $\textbf{H}$ are not smooth, the third step was originally performed by Tao-Vu \cite{RME, ULES}, in which works they developed the \emph{Four Moment Theorem}, which essentially states that if $\textbf{H}$ and $\widetilde{\textbf{H}}$ are complex Hermitian (or real symmetric) matrices with mutually independent entries whose first four moments are finite and coincide, then the local spectral statistics of $\textbf{H}$ and $\widetilde{\textbf{H}}$ converge in the large $N$ limit. This result had in the past been used to establish universality for a wide class of models \cite{SSGESEE, CEUM, RCMULS, ULES, UCM}. However, it was later \cite{FEUGM, SSSG, BUSM} realized that this method could be significantly simplified through an application of It\^{o}'s Lemma, if the value of $t$ from step \ref{perturb} is sufficiently small. 

This three-step strategy is remarkably general; it has been applied to establish universality for local statistics of Wigner matrices in many different contexts \cite{FEUGM, SSGESEE, UM, BUMSD, URMLRF, LRFULSRM, BUCM, BUGM, BUSM, FEUM, RCMULS, ULES, UCM}. In fact, it has also been recently used to establish bulk universality of random matrices whose entries exhibit various forms of correlation \cite{LESRMWGSRC, LSLRRG, BESRRG, URMCE}. 

However, until now, all known proofs of these universality results required a growth hypothesis on the entries of the matrix \textbf{H} that is stronger than assumption \ref{moments}. Originally \cite{UM, BUMSD, URMLRF, LRFULSRM, BUGM, ULES}, this took the form of a subexponential decay condition on the entries $h_{ij}$ of $\textbf{H}$, which stipulated that $\mathbb{P} \big[ |h_{ij} \sqrt{N}| > r \big] < C \exp(-c r^c)$ for all $i, j$ and some constants $c, C > 0$. 

The reason for this comes from the proof the local semicircle law, which is needed to proceed with the second and third steps of the three-step strategy. Specifically, the proof of the local semicircle law requires large deviation estimates that are obtained by taking very large moments of (functionals of) the entries of $\textbf{H}$. If higher moments of $h_{ij} \sqrt{N}$ are infinite, it is no longer immediately apparent that these large deviation estimates remain valid. 

There have been several attempts to weaken this decay condition, typically through different truncation arguments. For instance, in \cite{RCMULS}, Tao-Vu establish universality for Wigner matrices (and also covariance matrices) under the assumption that $h_{ij} \sqrt{N}$ has finite $C$-th moment for sufficiently large $C$; they took $C = 10^4$ and made no attempt to optimize, but it seems unlikely that their method could take $C$ close to $4$. 

Later, in \cite{UCMMC}, Johansson used a more refined analysis of the Br\'{e}zin-Hikami identity to establish bulk universality of complex Hermitian Gaussian divisible matrices $\textbf{H}$ whose entries $|h_{ij} \sqrt{N}|$ have only two moments. Again, this only applied to complex Hermitian random matrices and therefore did not yield results on real symmetric Wigner matrices. Furthermore, removing the Gaussian component remained troublesome, particularly since the Four Moment Theorem of Tao and Vu could no longer be applied (the $|h_{ij} \sqrt{N}|$ do not have four moments). 

Moreover, in \cite{SSGESEE}, Erd\H{o}s-Knowles-Yau-Yin implemented a new truncation procedure, based on the local semicircle law for sparse graphs \cite{SSG}, to prove bulk universality for $\textbf{H}$ when the $h_{ij} \sqrt{N}$ have $C = 4 + \varepsilon$ moments. Before this work, $4 + \varepsilon$ had been the lowest value of $C$ one could take. 

Again, the main part in the three-step strategy that requires the moment condition is the proof of the local semicircle law. Partly for that reason, there have been several works \cite{LSLUMC, LSLM, LENM} analyzing the extent to which the local semicircle law remains valid under perturbed moment conditions. For example, the recent work of G\"{o}tze-Naumov-Tikhomirov \cite{LSLUMC} and G\"{o}tze-Naumov-Timushev-Tikhomirov \cite{LSLM} establishes a strong local semicircle law for Wigner matrices whose entries $h_{ij} \sqrt{N}$ again have at least $C = 4 + \varepsilon$ moments; as before, this had until now been the smallest value of $C$ one could take in order to prove a local semicircle law. 

Hence, each of the methods mentioned above appears to exhibit a block preventing verification of bulk universality for (generalized) Wigner matrices whose entries have less than four moments. One possible reason for this is that the qualitative behavior of a matrix $\textbf{H}$ with infinite $(4 - \varepsilon)$-th moment is different from that of a matrix $\widetilde{\textbf{H}}$ with finite $(4 + \varepsilon)$-moment. In particular, the entries of $\widetilde{\textbf{H}}$ are expected to decay with $N$; one can show that, with high probability, the largest entries of $\widetilde{\textbf{H}}$ are at most of order $N^{-\varepsilon / 4}$. Until now, this decay of the entries seemed to be what was needed in the proof of a local semicircle law (see, for example, Theorem 3.1 of \cite{CEUM}). However, it is possible (and in some cases expected) that several entries of the more heavy-tailed matrix $\textbf{H}$ will grow with $N$. This poses issues in all known proofs \cite{SSG, LSLGCRM, ERLRM, LSLCDRM, LSLUMC, LSLM, CEUM} of local semicircle laws. 

Our purpose here is to overcome these issues and establish bulk universality for generalized Wigner matrices only subject to the (essentially weakest possible) assumption \ref{moments}. The main novelty of this paper that allows us to do this is the proof of a local semicircle law for Wigner matrices whose entries have less than four (and in fact only $2 + \varepsilon$) moments.

\subsection{The Local Semicircle Law}

\label{LocalCircleLaw}

The local semicircle law for random matrices can be stated in several different ways. The formulation of interest to us is in terms of the \emph{Stieltjes transform} of empirical spectral distribution $\mu_{\textbf{H}}$; this is defined by the function 
\begin{flalign}
\label{mn}
m_N = m_N (z) = m_{N; \textbf{H}} (z) = \displaystyle\frac{1}{N} \displaystyle\sum_{j = 1}^N \displaystyle\frac{1}{\lambda_j - z} = \displaystyle\frac{1}{N} \Tr \big( \textbf{H} - z \big)^{-1}, 
\end{flalign}

\noindent for any $z \in \mathbb{H}$, where $\mathbb{H} = \{ z \in \mathbb{C} : \Im z > 0 \}$ denotes the upper half plane. 

In view of the fact that $\mu_{\textbf{H}}$ converges weakly to the semicircle law $\rho_{\semicircle}$ \eqref{rhodefinition} as $N$ tends to $\infty$, one expects $m_N (z)$ to converge to $m_{\semicircle} = m_{\semicircle} (z) = \int_{\mathbb{R}} \rho_{\semicircle} (x) dx / (x - z)$. It can be quickly verified that $m_{\semicircle} (z)$ is the unique solution $m$ with positive imaginary part to the quadratic equation 
\begin{flalign}
\label{mquadratic} 
m^2 + zm + 1 = 0. 
\end{flalign}

Denoting $E = \Re z$ and $\eta = \Im z$ (so that $z = E + \textbf{i} \eta$), observe that $\lim_{\eta \rightarrow 0} \Im m_N (z)$ converges weakly to the distribution $(\pi N)^{-1} \sum_{j = 1}^N \delta_{\lambda_j = E}$, which provides information about the spectral behavior of $\textbf{H}$ near the energy level $E$. This suggests that, in order to understand the local spectral properties of $\textbf{H}$ near $\eta$, one might try to understand the behavior of $m_N (z)$ when $\eta$ is very small. 

A \emph{local semicircle law} for $\textbf{H}$ is an estimate on $\big| m_N (z) - m_{\semicircle} (z) \big|$ when $\eta$ is nearly of order $N^{-1}$. Establishing local semicircle laws is also often a task of independent interest \cite{LSSRRGFD, LSLRRG, SSG, LSLGCRM, LSLCDRM, LSLUMC, LSLM, BCESDSL, LENM}, since it quantifies the convergence of the spectral distribution of the random matrix $\textbf{H}$ to the semicircle law on very small sub-intervals (approximately of size $N^{-1}$) of $[-2, 2]$. It also has other consequences, such as complete eigenvector delocalization (to be discussed below). 

The following theorem provides such an estimate for generalized Wigner matrices. In what follows, we define the domain  
\begin{flalign}
\label{dkappanr}
\mathscr{D}_{\kappa; N; r} = \big\{ z = E + \textbf{i} \eta \in \mathbb{H} : E \in [\kappa - 2, 2 - \kappa], \eta \in [r, 5] \big\}, 
\end{flalign}

\noindent for any integer $N > 0$ and real numbers $\kappa, r = r_N > 0$. Further define $\mathscr{D}_{\kappa; N} = \mathscr{D}_{\kappa; N; \varphi}$, where $\varphi = \varphi_N = (\log N)^{8 \log \log N} N^{-1}$.

\begin{thm}

\label{localmoments}

Fix $\kappa > 0$, let $N > 0$ be a positive integer, and let $\textbf{\emph{H}} = \textbf{\emph{H}}_N$ be an $N \times N$ generalized Wigner matrix as in Definition \ref{momentassumption}. Then, there exist constants $C, c, \xi > 0$ (that only depend on $\varepsilon, c_1, C_1, C_2, \kappa$) such that 
\begin{flalign}
\label{localmomentsequation}
\mathbb{P} \Bigg[ \bigcup_{z \in \mathscr{D}_{\kappa; N}} \bigg\{ \big| m_N (z) - m \big| > C (\log N)^{\xi} \Big( \displaystyle\frac{1}{\sqrt{N \eta}} + N^{-c \varepsilon	} \Big) \bigg\}    \Bigg] < C N^{-c \log \log N}. 
\end{flalign} 

\end{thm}

\begin{rem}

\label{stronglaw}

Theorem \ref{localmoments} is known as a \emph{weak local semicircle law}, since the error term in \eqref{localmomentsequation} is not optimal in terms of $N \eta$. A \emph{strong local semicircle law} \cite{LSLGCRM} would correspond to the estimate \eqref{localmomentsequation}, with the term $(N \eta)^{-1 / 2}$ replaced by the smaller error $(N \eta)^{-1}$. It is possible to establish this strong version of Theorem \ref{localmoments} by suitably combining our method with Theorem 5.6 of \cite{SSG}. However, this will not be necessary for us, so we do not pursue it. 

\end{rem}

\begin{rem}

Observe the term $N^{-c \varepsilon}$ appearing as an error in \eqref{localmomentsequation}; it is also present in the local semicircle law for sparse Erd\H{o}s-Renyi graphs and sparse regular random graphs (see Theorem 2.8 of \cite{SSG} and Theorem 1.1 of \cite{LSLRRG}, respectively). Although this additional term has little impact when $\eta = \Im z$ is small (nearly of order $N^{-1}$), it indicates a possible obstruction of convergence from $m_N (z)$ to $m_{\semicircle}$ when $\eta = \Im z$ is large (nearly of order $1$). 

In particular, it suggests that $m_N (z)$ might converge to $m_{\semicircle}$ at rate $N^{-c \varepsilon}$, instead of at the fastest possible rate $N^{-1}$, which was established by G\"{o}tze-Tikhomirov \cite{BCESDSL} in the case when the fourth moments of $h_{ij} \sqrt{N}$ are bounded. In view of the results of \cite{LSHTRM}, we do not believe that \eqref{localmomentsequation} is optimal for small $\varepsilon$, when the $h_{ij} \sqrt{N}$ only have $(2 + \varepsilon)$ moments. Instead, we find it plausible that the error term $N^{-c \varepsilon}$ in \eqref{localmomentsequation} should replaced by $N^{-1 / 2 - c \varepsilon}$. However, the weaker estimate \eqref{localmomentsequation} will suffice to establish bulk universality, so we do not pursue these improvements any further. 

\end{rem}

To the best of our knowledge, all known proofs of local semicircle laws rely on a detailed understanding of the \emph{resolvent} of $\textbf{H}$, defined to be the $N \times N$ matrix $\textbf{G} = \textbf{G} (z) = \textbf{G} (z, \textbf{H}) = \big( \textbf{H} - z \big)^{-1} = \big\{ G_{ij} (z) \big\} = \big\{ G_{ij} \big\}$. Indeed, since $m_N = N^{-1} \Tr \textbf{G}$, it suffices to estimate the diagonal entries of $\textbf{G}$. In most known cases (with the exception of the very recent work \cite{LSSRRGFD} on regular random graphs of finite degree), it happens that all entries $G_{ij}$ of the resolvent will be close to $\textbf{1}_{i = j} m_{\semicircle}$ with very high probability. 

In the low-moment setting, this will not quite be the case. Similar to in \cite{LSSRRGFD}, the following will instead hold. For ``most'' pairs $i, j \in [1, N]$, we will have that $G_{ij}$ is close to $\textbf{1}_{i = j} m_{\semicircle}$; however, there will be a small fraction of index pairs $(i, j)$ for which this will not be true. Still, it will hold that these few entries remain uniformly bounded with very high probability as $N$ tends to $\infty$. 

The result is more specifically stated as follows. 

\begin{thm}

\label{estimate1gij}

Under the same assumptions as in Theorem \ref{localmoments}, there exist constants $C, c, \xi > 0$ (that only depend on $\varepsilon, c_1, C_1, C_2, \kappa$) such that the following two estimates hold. 

First, we have that 
\begin{flalign}
\label{gijestimate}
\mathbb{P} \Bigg[   \displaystyle\max_{1 \le i, j \le N} \displaystyle\sup_{z \in \mathscr{D}_{\kappa; N}} \big| G_{ij} (z) \big| > C    \Bigg] < C N^{- c \log \log N}. 
\end{flalign} 

Second, if we set $s = s_N = \lfloor C N^{1 - c \varepsilon} \rfloor$, then 
\begin{flalign}
\label{gijestimate2}
\begin{aligned}
\mathbb{P} \Bigg[ \bigcup_{\substack{I \subseteq [1, N] \\ |I| \ge s}} & \bigg\{ \displaystyle\min_{i \in I} \displaystyle\max_{1 \le j \le N} \displaystyle\sup_{z \in \mathscr{D}_{\kappa; N}}  \big| G_{i j} (z) - \textbf{\emph{1}}_{i = j} m_{\semicircle} (z) \big| > C (\log N)^{\xi} \Big( \displaystyle\frac{1}{\sqrt{N \eta}} + N^{-c \varepsilon} \Big) \bigg\} \Bigg] \\
& < C N^{- c \log \log N}. 
\end{aligned}
\end{flalign}

\end{thm}

Observe that Theorem \ref{localmoments} is a direct consequence of Theorem \ref{estimate1gij} and the estimate 
\begin{flalign}
\label{gdt}
\big| m_N - m_{\semicircle} \big| = \big| N^{-1} \Tr \textbf{G} - m_{\semicircle} \big| \le N^{-1} \displaystyle\sum_{i \in I} \big| G_{ii} - m_{\semicircle} \big| + N^{-1} \displaystyle\sum_{i \in [1, N] \setminus I} \big| G_{ii} - m_{\semicircle} \big|,
\end{flalign}

\noindent which holds for any subset $I \subseteq [1, N]$. In particular, we apply \eqref{gdt} with $I$ equal to the set of indices $i$ such that $G_{ij}$ is not close to $\textbf{1}_{i = j} m_{\semicircle}$ for some $j$ (in the sense of \eqref{gijestimate2}), which has cardinality at most $s_N$ with high probability due to \eqref{gijestimate2}. We then use \eqref{gijestimate} to estimate the first sum on the right side of \eqref{gdt} and \eqref{gijestimate2} to estimate the second sum; this yields Theorem \ref{localmoments} assuming Theorem \ref{estimate1gij}. Thus, it suffices to establish the latter theorem. 

Another consequence of Theorem \ref{estimate1gij} is the \emph{complete delocalization of eigenvectors} of $\textbf{H}$ corresponding to eigenvalues in the bulk; this is stated precisely by the following corollary. We mention that, in the low-moment case given by assumption \ref{moments}, this delocalization does not always hold for eigenvectors whose eigenvalues are at the edge of the spectrum of $\textbf{H}$ \cite{CLEHTRM, CEUM}. 

\begin{cor}

\label{delocalization}

Under the same assumptions as in Theorem \ref{localmoments}, there exist constants $C, c, \xi > 0$ (that only depend on $\varepsilon, c_1, C_1, C_2, \kappa$) such that the following holds. The probability that there exists an eigenvalue $\lambda \in [\kappa - 2, 2 - \kappa]$ of $\textbf{\emph{H}}$ with eigenvector $\textbf{\emph{v}}$ such that $\| \textbf{\emph{v}} \|_2 = 1$ and $\| \textbf{\emph{v}} \|_{\infty} > C (\log N)^{\xi} N^{- 1 / 2}$ is less than $C N^{- c \log \log N}$. 

\end{cor}

Given Theorem \ref{estimate1gij}, the proof of Corollary \ref{delocalization} is very similar to that of Theorem 2.21 of \cite{URM} and is therefore omitted. 	

Thus, in order to establish the local semicircle law Theorem \ref{localmoments} and complete eigenvector delocalization Corollary \ref{delocalization}, it remains to establish Theorem \ref{estimate1gij}. We will outline the proof of this result in Section \ref{ProofOutline}, explaining why previous proofs of semicircles no longer seem to directly apply and also indicating some of the new elements needed to show Theorem \ref{estimate1gij}. In Sections \ref{DefineIndices}, \ref{LawLarge}, \ref{LawSmallNotDeviant}, \ref{LawSmallDeviant}, and \ref{ProofLocalLaw}, we implement this outline in detail. In Section \ref{LocalH}, we proceed with the remaining two parts of the three-step strategy and establish the bulk universality results Theorem \ref{gapsfunctions} and Theorem \ref{bulkfunctions}; given Theorem \ref{localmoments}, this will mainly involve recalling what was done in the recent works \cite{FEUM, BUSM} that individually address these second and third steps, respectively.

\subsection*{Acknowledgements}

The author heartily thanks Horng-Tzer Yau for proposing this question, for several valuable discussions, and for looking through an early version of this paper. The author is also very grateful to Jiaoyang Huang and Benjamin Landon for many fruitful conversations and helpful explanations. This work was funded by the NSF Graduate Research Fellowship under grant number DGE1144152 and partially by the Eric Cooper and Naomi Siegel Graduate Student Fellowship I.

\section{Outline of the Proof of the Local Semicircle Law}

\label{ProofOutline}

In this section we outline a proof of Theorem \ref{estimate1gij} and explain in what ways it differs from known proofs of the local semicircle law for Wigner matrices with less singular entries. We begin in Section \ref{DiagonalResolvent} with some matrix identities used for the analysis in both the original and heavy-tailed setting. In Section \ref{LargeMoment} we recall the idea of the proof when the $h_{ij} \sqrt{N}$ have all moments and explain the issues that arise when this restriction is no longer assumed. Then, in Section \ref{SmallMoment} we outline how to resolve these issues. Throughout this section, we adopt the notation of Theorem \ref{estimate1gij}.

\subsection{Estimating the Diagonal Resolvent Entries}

\label{DiagonalResolvent} 

We begin by collecting several matrix identities that will be useful for us. In what follows, for any $N \times N$ matrix $\textbf{M} = \{ M_{ij} \}$ and subset $\mathcal{S} \subset [1, N]$, let $\textbf{M}^{(\mathcal{S})} = \{ M_{ij}^{(\mathcal{S})} \}$ denote the $(N - |\mathcal{S}|) \times (N - |\mathcal{S}|)$ matrix formed from removing the $j$-th row and column from $\textbf{M}$, for each $j \in \mathcal{S}$. If $S = \{ i \}$ consists of one element, we abbreviate $\textbf{M}^{(\{ i \})} = \textbf{M}^{(i)}$. 

\begin{lem}

Let $\textbf{\emph{H}}$ be an $N \times N$ (deterministic or random) matrix, $z \in \mathbb{H}$, and $\eta = \Im z$. Denote $\textbf{\emph{G}} = (\textbf{\emph{H}} - z)^{-1}$. 

\label{matrixidentities} 

\begin{enumerate}[align=left]
\item[\emph{1. Schur complement identity:} ]{ Suppose that $\textbf{\emph{A}}$, $\textbf{\emph{B}}$, $\textbf{\emph{C}}$, and $\textbf{\emph{D}}$ are generic $k \times k$, $k \times m$, $m \times k$, and $m \times m$ matrices, respectively. Then,

\begin{flalign}
\label{blockinverse}
\left[ \begin{array}{cc} \textbf{\emph{A}} & \textbf{\emph{B}} \\ \textbf{\emph{C}} & \textbf{\emph{D}} \end{array} \right]^{-1} = \left[ \begin{array}{cc} (\textbf{\emph{A}} - \textbf{\emph{B}} \textbf{\emph{D}}^{-1} \textbf{\emph{C}})^{-1} & (\textbf{\emph{B}} \textbf{\emph{D}}^{-1} \textbf{\emph{C}} - \textbf{\emph{A}})^{-1} \textbf{\emph{B}} \textbf{\emph{D}}^{-1}  \\ \textbf{\emph{D}}^{-1} \textbf{\emph{C}} (\textbf{\emph{B}} \textbf{\emph{D}}^{-1} \textbf{\emph{C}} - \textbf{\emph{A}})^{-1} & \textbf{\emph{D}}^{-1} + \textbf{\emph{D}}^{-1} \textbf{\emph{C}} (\textbf{\emph{A}} - \textbf{\emph{B}} \textbf{\emph{D}}^{-1} \textbf{\emph{C}}) \textbf{\emph{B}} \textbf{\emph{D}}^{-1} \end{array} \right]. 
\end{flalign}

\noindent In particular, for any $i \in [1, N]$ we have that 
\begin{flalign}
\label{gii}
\displaystyle\frac{1}{G_{ii}} = h_{ii} - z - \sum_{\substack{1 \le j, k \le N \\ j, k \ne i }} h_{ij} G_{jk}^{(i)} h_{ki}, 
\end{flalign}}

\item[\emph{2. Resolvent identity:}]{If $\textbf{\emph{A}}$ and $\textbf{\emph{B}}$ are square matrices of the same dimension, then
\begin{flalign}
\label{resolvent}
\textbf{\emph{A}}^{-1} - \textbf{\emph{B}}^{-1} = \textbf{\emph{A}}^{-1} \big( \textbf{\emph{B}} - \textbf{\emph{A}} \big) \textbf{\emph{B}}^{-1}. 
\end{flalign} 

\noindent In particular, for any distinct $i, j \in [1, N]$, we have that 
\begin{flalign}
\label{gij}
G_{ij} & = - G_{ii} \displaystyle\sum_{\substack{k \in [1, N] \\ k \ne i}} h_{ik} G_{kj}^{(i)} = -G_{jj} \displaystyle\sum_{\substack{k \in [1, N] \\ k \ne j}} h_{kj} G_{ik}^{(j)}. 
\end{flalign}

\noindent Thus, for all $i, j, k \in [1, N]$ such that $i \notin \{ j, k \}$, we have that
\begin{flalign}
\label{gkj}
G_{kj} = G_{kj}^{(i)} + \displaystyle\frac{G_{ki} G_{ij}}{G_{ii}}. 
\end{flalign}} 

\item[\emph{3. Deterministic estimates:}]{For any $i, j \in [1, N]$, we have that 
\begin{flalign}
\label{gijeta}
\big| G_{ij} \big| < \eta^{-1}. 
\end{flalign}}

\item[\emph{4. Ward identity:}]{Let $\mathcal{S} \subset [1, N]$. For any $j \in [1, N] \setminus \mathcal{S}$, we have that 
\begin{flalign}
\label{sumgij}
\displaystyle\sum_{j \in [1, N] \setminus \mathcal{S}} \big| G_{jk}^{(\mathcal{S})} \big|^2 = \displaystyle\frac{\Im G_{jj}^{(\mathcal{S})}}{\eta}. 
\end{flalign}}

\end{enumerate}

\end{lem}

Each of the statements above can be found in the book \cite{DRM}. Specifically, \eqref{blockinverse} can be found as (7.2) and (7.3) there; \eqref{gii} as (7.7); \eqref{resolvent} as (8.4); \eqref{gij} as (8.2); \eqref{gkj} as (8.1); \eqref{gijeta} as (8.34); and \eqref{sumgij} as (8.3). 

The proof of the local semicircle law is based on the identity \eqref{gii}; let us explain why it is useful. To that end, observe that sum on the left side of \eqref{gii} can be rewritten as
\begin{flalign}
\label{sumedmf}
\sum_{\substack{1 \le j, k \le N \\ i \notin \{ j , k \} }} h_{ij} G_{jk}^{(i)} h_{ki} = F_i + E_i + D_i + M_i + m_{\semicircle} + m_{\semicircle} t_i
\end{flalign}

\noindent where 
\begin{flalign}
\label{edmf}
\begin{aligned}
F_i & = \sum_{\substack{1 \le j \ne k \le N \\ j, k \ne i }} h_{ij} G_{jk}^{(i)} h_{ki}; \qquad \quad  E_i = \sum_{\substack{1 \le j \le N \\ j \ne i }} \big( | h_{ij} |^2 - s_{ij} \big) G_{jj}^{(i)}; \\
 D_i & = \sum_{\substack{1 \le j \le N \\ j \ne i }} s_{ij} \big( G_{jj}^{(i)} - G_{jj} \big); \qquad M_i = \sum_{\substack{1 \le j \le N \\ j \ne i }} s_{ij} \big( G_{jj} - m_{\semicircle} \big), 
\end{aligned}
\end{flalign}

\noindent where we recalled from assumption \ref{generalized} that $s_{ij} = \Var h_{ij}$ and from assumption \ref{stochastic} that $t_i = \sum_{j = 1}^N s_{ij} - 1$. Thus, if we denote 
\begin{flalign}
\label{vj}
v_j = G_{jj} - m_{\semicircle},
\end{flalign}

\noindent for each $j \in [1, N]$ and insert \eqref{sumedmf} into \eqref{gii}, we deduce that 
\begin{flalign*}
\displaystyle\frac{1}{v_i + m_{\semicircle}} + z + m_{\semicircle} = h_{ii} - F_i - E_i - D_i - M_i - m_{\semicircle} t_i. 
\end{flalign*}

\noindent Applying \eqref{mquadratic}, it follows that 
\begin{flalign}
\label{vmfed}
\displaystyle\frac{v_i}{1 + m_{\semicircle}^{-1} v_i} - m_{\semicircle}^2 \displaystyle\sum_{j = 1}^N s_{ij} v_i = m_{\semicircle}^2 \big( F_i + E_i + D_i - h_{ii} + m_{\semicircle} t_i \big). 
\end{flalign}

\noindent Now, recall from assumption \ref{stochastic} that $|t_i| < C_1 N^{-varespilon}$. Suppose that we additionally knew that 
\begin{flalign}
\label{vmfedsmall}
|F_i|, |E_i|, |D_i|, |h_{ii}| = \mathcal{O} \Bigg( (\log N)^{C_3} \left( \displaystyle\frac{1}{\sqrt{N \eta}} + N^{- c_2 \varepsilon}  \right) \Bigg); \qquad |v_i| = o (1). 
\end{flalign} 

\noindent Then, \eqref{vmfed} would imply that 
\begin{flalign}
\label{vsmall}
\big( \Id - m_{\semicircle}^2 \textbf{S} \big) \textbf{v} = \mathcal{O} \Bigg( (\log N)^{C_3} \left( \displaystyle\frac{1}{\sqrt{N \eta}} + N^{- c_2 \varepsilon}  \right) + \displaystyle\max_i |v_i|^2 \Bigg) , 
\end{flalign}

\noindent where the $N \times N$ matrix $\textbf{S} = \{ s_{ij} \}$ and the $N$-dimensional vector $\textbf{v} = (v_1, v_2, \ldots , v_N)$. It can be shown that $\| \Id - m_{\semicircle}^2 \textbf{S} \|^{-1} = \mathcal{O} (\log N)$, from which we would obtain that 
\begin{flalign}
\label{largevsmall} 
\displaystyle\max_i |v_i| = \mathcal{O} \Bigg( (\log N)^{C_4} \left( \displaystyle\frac{1}{\sqrt{N \eta}} + N^{- c_2 \varepsilon}  \right) \Bigg),  
\end{flalign}

\noindent from which it would follow that each $|G_{ii} - m_{\semicircle}|$ is small as in \eqref{gijestimate2}. Thus, we would like to establish the two estimates \eqref{vmfedsmall}.

\subsection{When the \texorpdfstring{$h_{ij}$}{} Have All Moments}

\label{LargeMoment}

Let us first restrict to a case that is well understood, namely when the laws of the $h_{ij}$ exhibit subexponential decay (in this case, we can take $\varepsilon = \infty$). For simplicity, we also assume in this section that $s_{ij} = N^{-1}$ for all $i, j$. Under these assumptions, the estimates \eqref{vmfedsmall} will hold with very high probability (greater than $1 - N^{-10}$, for instance) for all indices $i \in [1, N]$; let us provide a heuristic as to why. 

The second estimate in \eqref{vmfedsmall} (on $|v_i|$) will be a consequence of what is known as a \emph{multiscale argument} on $\eta$. First, it can be shown directly (see, for instance, Section \ref{LawLarge}) that the estimate \eqref{largevsmall} holds deterministically when $\Im z = \eta = \eta_0$ is relatively large, of order $1$; this is sometimes referred to as an \emph{initial estimate}. Then, we will slowly decrement $\eta$ through a sequence $\eta_0 > \eta_1 > \cdots > \eta_r \approx N^{-1}$; here, $r$ is some large integer (which will be of order $(\log N)^2$). 

We will have that $\eta_k \approx \eta_{k + 1}$  for all $k$, so that $v_i (\eta_k) \approx v_i (\eta_{k + 1})$ for all $i$. Combining this with the estimate \eqref{largevsmall} on $\big| v_i (\eta_k) \big|$ will yield $\big| v_i (\eta_{k + 1}) \big| = o(1)$, which is the second estimate of \eqref{vmfedsmall}. Thus, if we could establish the first estimate (on $F_i$, $E_i$, and $D_i$) of \eqref{vmfedsmall}, we would deduce \eqref{vsmall} and thus \eqref{largevsmall} for $v_i (\eta_{k + 1})$. This allows us to increment $k$ from $0$ to $r$. 

Now, let us explain the first estimate in \eqref{vmfedsmall}. As an example, we provide a heuristic as to why one would expect (with very high probability) the estimate on $|F_i|$ to hold. To that end, we take the second moment of $F_i$ to obtain that 
\begin{flalign}
\label{fi21}
\begin{aligned}
\mathbb{E} \big[ |F_i|^2 \big] & = \sum_{\substack{1 \le j \ne k \le N \\ j, k \ne i }} \sum_{\substack{1 \le j' \ne k' \le N \\ j', k' \ne i }} \mathbb{E} \big[ h_{ij} G_{jk}^{(i)} h_{ki} h_{ij'} G_{j'k'}^{(i)} h_{k'i} \big] \\
& = 2 N^{-2} \sum_{\substack{1 \le j \ne k \le N \\ j, k \ne i }} \mathbb{E} \Big[ \big| G_{jk}^{(i)} \big|^2 \Big] = \displaystyle\frac{2}{N^2 \eta} \sum_{\substack{1 \le j \le N \\ j \ne i }} \mathbb{E} \Big[ \Im G_{jj}^{(i)} \Big] = \displaystyle\frac{1}{N \eta} \Big( 2 N^{-1} \Im \mathbb{E} \big[ \Tr \textbf{G}^{(i)} \big] \Big), 
\end{aligned}
\end{flalign}
\noindent where we used the mutual independence of the $h_{ij}$, the independence of $\textbf{G}^{(i)}$ from the $h_{ij}$, and the Ward identity \eqref{sumgij}. The estimate $|v_i| = o(1)$ yields that $N^{-1} \Im \Tr \textbf{G}^{(i)} = m_{\semicircle} + o (1) = \mathcal{O} (1)$. 

Inserting this into \eqref{fi21} yields $\mathbb{E} \big[ |F_i|^2 \big] = \mathcal{O} \big( (N \eta)^{-1} \big)$, which suggests that $|F_i|$ should be of order $(N \eta)^{-1 / 2}$. By taking very large moments (instead of only the second moment), one can show that this in fact holds with very high probability. Observe here that bounding higher moments of $|F_i|$ requires that the $h_{ij}$ have all moments. The estimates on $E_i$ and $D_i$ can either be done similarly or through other matrix identities. 

The above procedure essentially describes the framework for establishing local semicircle laws when the matrix entries $h_{ij}$ are quite regular. However, this method does not seem to immediately work in the same way when the $h_{ij}$ have few (for instance, less than four) moments. In the next section, we explain why and outline how to resolve this.

\subsection{When the \texorpdfstring{$h_{ij}$}{} Have Few Moments}

\label{SmallMoment} 

In this section we again assume that $s_{ij} = N^{-1}$ for all $i, j$, but we no longer require that the $h_{ij}$ have all moments. Instead, we only require that they have $2 + \varepsilon$ moments, as in assumption \ref{moments}. Let us explain which parts of Section \ref{LargeMoment} still apply and which do not. 

The multiscale argument for the bound on $|v_i|$ would still essentially be valid, assuming that one were able to establish the first estimate of \eqref{vmfedsmall} with high probability. However, this does not appear to be possible. Indeed, the high-moment method used to establish these estimates (outlined in the previous section) no longer applies since the $h_{ij}$ have very few moments. 

Instead, for fixed $i$, these bounds only hold with probability $1 - \mathcal{O} \big( N^{1 - c \varepsilon} \big)$, for some constant $c > 0$. This tends to $1$ as $N$ tends to $\infty$, but not as quickly as one may like. In particular, one cannot apply a union estimate to deduce that \eqref{vmfedsmall} likely holds for all $i \in [1, N]$ simultaneously. 

In fact, we generally expect to see $\mathcal{O} \big( N^{1 - c \varepsilon} \big)$ indices $i \in [1, N]$ for which \eqref{vmfedsmall} is false. This is the phenomenon that makes Wigner matrices whose entries only have $2 + \varepsilon$ moments different from Wigner matrices whose entries have $4 + \varepsilon$ moments; in the latter model, a truncation procedure can be applied \cite{CEUM} to deduce that \eqref{vmfedsmall} holds for all indices $i$ with high probability. 

In our setting, the set of indices $i$ for which this estimate does not hold will fall into an often non-empty class of what we call \emph{deviant} indices. Indices that are not deviant will be called \emph{typical}; both estimates \eqref{vmfedsmall} will hold for each typical index $j$, which will comprise the majority of $[1, N]$. 

Thus, our task is essentially four-fold. 

\begin{enumerate}

\item{ \label{definetd} We must give a precise definition of deviant and typical indices. Denoting the set of typical indices by $\mathcal{T}$ and the set of deviant indices by $\mathcal{D}$, we must also show that $\big| \mathcal{D} \big| = \mathcal{O} \big( N^{1 - c \varepsilon} \big)$, for some constant $c > 0$.}

\item{\label{initial} We require an initial estimate on $v_i$, in the case $\eta \approx 1$ is relatively large.} 

\item{ \label{vmfedsmallt} Given the definitions from step \ref{definetd}, we must establish \eqref{vmfedsmall} for all $i \in \mathcal{T}$. After this, it will be possible to essentially ``restrict" \eqref{vsmall} to the set of all typical indices, which will suffice to establish \eqref{largevsmall}, where the maximum in that estimate instead ranges over all $i \in \mathcal{T}$. This will suffice to establish the estimate \eqref{gijestimate2} of Theorem \ref{estimate1gij}. }

\item{ \label{vmfedd} In order to establish the estimate \eqref{gijestimate} of Theorem \ref{estimate1gij}, we must also show that $|v_i| = \mathcal{O} (1)$ for each (deviant and typical) index $i$.}

\end{enumerate}

 Combining results of step \ref{definetd}, step \ref{initial}, step \ref{vmfedsmallt}, and step \ref{vmfedd}, and the procedure outlined in Section \ref{DiagonalResolvent} will lead to the proof of Theorem \ref{estimate1gij}. The next several sections will go through these four steps in more detail. In particular, we implement step \ref{definetd}, step \ref{initial}, step \ref{vmfedsmallt}, and step \ref{vmfedd} in Section \ref{DefineIndices}, Section \ref{LawLarge}, Section \ref{LawSmallNotDeviant}, and Section \ref{LawSmallDeviant}, respectively. We will then conclude the proof of Theorem \ref{estimate1gij} in Section \ref{ProofLocalLaw}.

\section{Typical and Deviant Indices}

\label{DefineIndices}

In this section, we define and give properties of the sets of \emph{typical} and \emph{deviant} indices. The stimulus for these definitions comes from a comparison  between heavy-tailed random matrices and adjacency matrices of sparse random graphs, for which a local semicircle law has already been established \cite{SSG}. 

In particular, the authors of \cite{SSG} established a local semicircle law for random matrices $\textbf{H}$ whose entries satisfy a more stringent constraint than assumption \ref{moments}. We will not state their constraint in full generality, but for our purposes it was essentially that $\big| h_{ij} \big| < N^{-c}$ holds deterministically (this can be mildly weakened), for some constant $c > 0$ independent of $N$. 

This assumption does not always hold in our setting. In particular, a Markov estimate shows (for example) that 
\begin{flalign}
\label{pijestimate1}
\mathbb{P} \big[ |h_{ij}| \ge N^{-\varepsilon / 10} \big] \le \displaystyle\frac{\mathbb{E} \big[ |h_{ij} \sqrt{N}|^{2 + \varepsilon} \big]}{\big| N^{1 / 2 - \varepsilon / 10} \big|^{2 + \varepsilon}} < \displaystyle\frac{C_2}{N^{1 + 3 \varepsilon / 10 - \varepsilon^2 / 10}} \le C_2 N^{-1 - \varepsilon / 10}, 
\end{flalign}

\noindent if $\varepsilon \le 2$. Thus, we expect there to exist $\mathcal{O} \big( N^{1 - \varepsilon / 10} \big)$ pairs $(i, j) \in [1, N]^2$ for which $|h_{ij}| > N^{-\varepsilon / 10}$. Such pairs $(i, j)$ will be called \emph{big}; pairs $(i, j)$ satisfying $|h_{ij}| \le N^{-\varepsilon / 10}$ are called \emph{amenable}.

Informally, an index $i \in [1, N]$ will be \emph{deviant} if there exists a $j \in [1, N]$ for which the pair $(i, j)$ is big; otherwise, $i$ will be \emph{typical}. Unfortunately, the sets of deviant and typical indices are random subsets of $[1, N]$; this will complicate the analysis in future sections. 

Thus, in Section \ref{EntryTruncation}, we first resample the entries of $\textbf{H}$, essentially by conditioning on which pairs $(i, j) \subset [1, N]^2$ are amenable or big. This produces a symmetric $N \times N$ array, each of entry of which is either $A$ (amenable) or $B$ (big); we refer to this array as the \emph{$AB$ label} $\textbf{L} (\textbf{H})$ of $\textbf{H}$. Conditioning on $\textbf{L} (\textbf{H})$, the deviant and typical indices of $\textbf{H}$ become deterministic. We give a definition (as well as some properties) of these indices in Section \ref{IndicesGraph}. In Section \ref{deviantnondeviantlocallaw}, we explain how the notions introduced in Section \ref{EntryTruncation} and Section \ref{IndicesGraph} can be used to provide a reformulation of Theorem \ref{estimate1gij}.

\subsection{Resampling }

\label{EntryTruncation}

As outlined above, we first define the $AB$ label associated with a matrix $\textbf{M}$. 

\begin{definition}

\label{ml}

Let $\textbf{M} = \{ m_{ij} \}$ be an $N \times N$ matrix. The \emph{$AB$ label} of $\textbf{M}$, denoted $\textbf{L} (\textbf{M}) = \{ L_{ij} \}$, is the $N \times N$ array, whose entries are either equal to $A$ or $B$, such that $L_{ij} = A$ if $|m_{ij}| \le N^{-\varepsilon / 10}$ and $L_{ij} = B$ otherwise. 

\end{definition}

Now, we can resample $\textbf{H}$ by first choosing its $AB$ label $\textbf{L}$ and then by sampling the entries conditioned on $\textbf{L}$. Let us explain this in more detail. In what follows, we assume that the densities of the matrix entries $h_{ij}$ are smooth and nonzero everywhere. This is primarily for notational convenience and can be arranged by adding a small Gaussian component to $\textbf{H}$ (of order $e^{-N}$, for instance); using \eqref{resolvent} and \eqref{gijeta}, one can quickly verify that (with very large probability) this perturbation does not affect the asymptotics of the entries of $\textbf{G}$, as $N$ tends to $\infty$.

To explain this resampling further, we require some additional terminology. In what follows, we denote $p_{ij} = \mathbb{P} \big[ |h_{ij}| < N^{-\varepsilon / 10} \big]$ for each $1 \le i, j \le N$; our assumption implies that $p_{ij} \notin \{ 0, 1 \}$. 

\begin{definition}

\label{hl} 

We say that a random $N \times N$ symmetric $AB$ label $\textbf{L} = \{ L_{ij} \}$ is \emph{$\textbf{\emph{H}}$-distributed} if its upper triangular entries $\{ L_{ij} \}_{1 \le i \le j \le N}$ are mutually independent, $\mathbb{P} \big[ L_{ij} = A \big] = p_{ij}$, and $\mathbb{P} \big[ L_{ij} = B \big] = 1 - p_{ij}$, for each $i, j$. 
\end{definition}

We next provide notation for the random variables $h_{ij}$, conditioned on the event that $L_{ij} = A$ or that $L_{ij} = B$.

\begin{definition}

\label{abdefinition}

For each $1 \le i \le j \le N$, let $a_{ij}$ denote the random variable such that $\mathbb{P} [a_{ij} \in I] = p_{ij}^{-1} \mathbb{P} [h_{ij} \in I \cap (-N^{- \varepsilon / 10}, N^{-\varepsilon / 10})]$, for each interval $I \subset \mathbb{R}$; equivalently, $a_{ij}$ is the random variable $h_{ij}$ conditioned on the event that $|h_{ij}| < N^{-\varepsilon / 10}$. 

Furthermore, let $b_{ij}$ denote the random variable such that $\mathbb{P} [b_{ij} \in I] = (1 - p_{ij})^{-1} \mathbb{P} \big[ h_{ij} \in I \cap \big( ( - \infty, -N^{- \varepsilon / 10}] \cup [ N^{- \varepsilon / 10}, \infty) \big) \big]$ for each $I \subset \mathbb{R}$; equivalently, $b_{ij}$ is the random variable $h_{ij}$, conditioned on the event that $|h_{ij}| \ge N^{-\varepsilon / 10}$. Here, $\big\{ a_{ij} \} \cup \{ b_{ij} \}$ are mutually independent. 

\end{definition}

Using the previous definition, we can sample $\textbf{H}$ conditioned on its $AB$ label $\textbf{L} (\textbf{H})$. 

\begin{definition}

\label{hl1}

Fix an $N \times N$ symmetric $AB$ label $\textbf{L}$. Let $\textbf{H} (\textbf{L})$ denote the random symmetric matrix, which is sampled as follows. For each $1 \le i \le j \le N$, place the random variable $a_{ij}$ at entry $(i, j)$ if and only if $L_{ij} = A$; otherwise, place the random variable $b_{ij}$ at this entry. The lower triangular entries of $\textbf{H} (\textbf{L})$ (corresponding to entries $(i, j)$ with $i > j$) are then determined by symmetry. We call the matrix $\textbf{H} (\textbf{L})$ an \emph{$\textbf{\emph{L}}$-distributed symmetric random matrix}. 

\end{definition}

\begin{sampling}

\label{lh}

To sample the random matrix $\textbf{H}$, we perform the following steps. First, sample an $\textbf{H}$-distributed $AB$ label, denoted $\textbf{L}$. Then, given $\textbf{L}$, sample an $\textbf{L}$-distributed symmetric random matrix $\textbf{H} = \textbf{H} (\textbf{L})$.

\end{sampling}

It is quickly verified that the distribution of $\textbf{H}$ resulting from Sampling \ref{lh} coincides with the original distribution of $\textbf{H}$, so the above procedure indeed yields a resampling. 

We conclude this section with the following lemma, which provides some statistics on the $a_{ij}$. 

\begin{lem}

\label{entrysmaller}

If $\varepsilon \le 2$ and $N > 2 C_2$, we have that 
\begin{flalign*}
1 - p_{ij} \le C_2 N^{-1 - \varepsilon / 10}; \qquad \big| \mathbb{E} [a_{ij}] \big| \le 2 C_2 N^{-1 - \varepsilon / 10}; \qquad \big| \mathbb{E} [|a_{ij}|^2] - s_{ij} \big| \le 3 C_2 N^{-1 - \varepsilon / 10}.  
\end{flalign*}
	
\end{lem}

\begin{proof}

Each of these statements follows from a Markov estimate; the first one was verified in \eqref{pijestimate1}. To establish the second one, observe that 
\begin{flalign*}
\big| \mathbb{E} [a_{ij}] \big| \le p_{ij}^{-1} \mathbb{E} \big[ | h_{ij} | \textbf{1}_{h_{ij} < N^{- \varepsilon / 10}} \big] & \le p_{ij}^{-1} N^{\varepsilon (1 + \varepsilon) / 10} \mathbb{E} \big[ | h_{ij}|^{2 + \varepsilon} \big] < 2 C N^{-1 - \varepsilon / 10}, 
\end{flalign*}

\noindent where we have used \eqref{pijestimate1} and the fact that $N > 2 C_2$ to deduce that $p_{ij} > 1 / 2$. The proof of the third estimate is very similar and is thus omitted. 	
\end{proof}

\subsection{Deviant and Typical Indices}

\label{IndicesGraph}

In this section we give a precise definition of typical and deviant indices, which were informally introduced at the start of Section \ref{DefineIndices}. To that end, we begin with the following preliminary notion. 

\begin{definition}

\label{linkedunlinkedconnected}

Fix an $N \times N$ $AB$ label $\textbf{L} = \{ L_{ij} \} $. We call $i, j \in [1, N]$ \emph{linked} (with respect to $\textbf{L}$) if $L_{ij} = B$; otherwise we call them \emph{unlinked}. If there exists a sequence of indices $i = i_1, i_2, \ldots , i_r = j$ such that $i_j$ is linked to $i_{j + 1}$ for each $j \in [1, r - 1]$, then we call $i$ and $j$ \emph{connected}; otherwise, they are \emph{disconnected}. 
\end{definition}

Using Definition \ref{linkedunlinkedconnected}, we can define typical and deviant indices. 

\begin{definition}

\label{deviantnondeviant}

Fix an $N \times N$ $AB$ label $\textbf{L}$. We call an index $i \in [1, N]$ \emph{deviant} (with respect to $\textbf{L}$) if there exists some index $j \in [1, N]$ such that $i$ and $j$ are linked. Otherwise, $i$ is called \emph{typical} (with respect to $\textbf{L}$). Let $\mathcal{D} = \mathcal{D}_{\textbf{L}} \subseteq [1, N]$ denote the set of deviant indices, and let $\mathcal{T} = \mathcal{T}_{\textbf{L}} \subseteq [1, N]$ denote the set of typical indices. 
\end{definition}

Now fix an $N \times N$ symmetric $AB$ label $\textbf{L}$. Let us investigate how the resolvent $\textbf{G}$ of an $\textbf{L}$-distributed symmetric random matrix looks. Equivalently, by Sampling \ref{lh}, we consider how the resolvent $\textbf{G}$ of the generalized Wigner matrix $\textbf{H}$ looks, after conditioning on the event that $\textbf{L} (\textbf{H}) = \textbf{L}$. In view of Theorem \ref{estimate1gij}, we would hope that $\textbf{G} \approx m_{\semicircle} \Id$ for ``most'' $AB$ labels $\textbf{L}$. This is indeed true; the following definition clarifies the meaning of ``most'' $AB$ labels in our context.  

\begin{definition}

\label{admissiblelabel}

Fix an $N \times N$ symmetric $AB$ label $\textbf{L}$. 

\begin{itemize}

\item{\label{delta1} We call $\textbf{L}$ \emph{deviant-inadmissible} if there exist at least $N^{1 - \varepsilon / 20}$ deviant indices.}

\item{\label{delta3} We call $\textbf{L}$ \emph{connected-inadmissible} if there exist distinct indices $j_1, j_2, \ldots , j_r$ that are are pairwise connected, where $r = \lceil \log \log N \rceil$. }

\end{itemize}

We call $\textbf{L}$ \emph{inadmissible} if it is either deviant-inadmissible or connected-inadmissible. Otherwise, $\textbf{L}$ is called \emph{admissible}. Let $\mathcal{A} = \mathcal{A}_N$ denote the set all admissible $N \times N$ $AB$ labels. 

\end{definition}

\begin{definition}

\label{delta}

Let $\textbf{H}$ be a generalized Wigner matrix in the sense of Definition \ref{momentassumption}, and let $\textbf{L}$ be an $\textbf{H}$-distributed $AB$ label. Define $\Delta_1$ and $\Delta_2$ to be the events on which $\textbf{L}$ deviant-inadmissible and connected-inadmissible, respectively. 

Let $\Delta = \Delta_1 \cup \Delta_2$ denote the event on which $\textbf{L}$ is inadmissible, and let $\overline{\Delta}$ denote the complementary event on which $\textbf{L}$ is admissible. Furthermore, let $D(i)$ denote the event on which $i \in \mathcal{D}$, and let $T(i)$ denote the complementary event on which $i \in \mathcal{T}$. 

\end{definition}

 The following lemma shows that the event $\Delta$ occurs with small probability. 

\begin{lem}

\label{probabilitydelta}

There exist constants $c, C > 0$ (only dependent on $C_2$ and $\varepsilon$ from Definition \ref{momentassumption}) such that $\mathbb{P} \big[ \Delta \big] < C N^{- c \log \log N}$. 

\end{lem}

\begin{proof}

To establish this lemma we individually estimate $\mathbb{P} [\Delta_1]$ and $\mathbb{P} [\Delta_2]$. 

We begin with the former. To that end, first observe that if there exist $N^{1 - \varepsilon / 20}$ deviant indices, then there must exist at least $R = N^{1 - \varepsilon / 20} / 2$ indices $(i, j) \in [1, N]^2$ such that $1 \le i \le j \le N$ and each $L_{ij} = B$ or such that $1 \le j \le i \le N$ and each $L_{ij} = B$. 

The two cases are equivalent, so assume that the first holds; then, the $L_{ij}$ are independent. Furthermore, in view of \eqref{pijestimate1}, we have that $\mathbb{P} \big[ L_{ij} = B \big] \le C_2 N^{-1 - \varepsilon / 10}$. Therefore, the independence of the $L_{ij}$ implies that
\begin{flalign}
\label{delta1estimate}
\mathbb{P} \big[ \Delta_1 \big] \le \displaystyle\sum_{j = R}^N \binom{N^2}{j} \big( C_2 N^{-1 - \varepsilon / 10} \big)^j \le C_3 N^{- c_3 \log \log N},  
\end{flalign}

\noindent for some constants $c_3, C_3 > 0$. 

To bound $\mathbb{P} [\Delta_2]$, observe that the event $\Delta_2$ is contained in the event that there exists a sequence $S = \{ i_1, i_2, \ldots , i_r \} \subset [1, N]$ of indices such that there are at least $r - 1$ pairs of distinct indices $(i_j, i_k)$ that are linked. 

There are $\binom{N}{r}$ ways to select such a sequence, and there are less than $\binom{r^2}{r - 1}$ to select $r - 1$ pairs of indices to link in this sequence. Furthermore, the event that $i_j$ and $i_k$ are linked is independent of the event that $i_{j'}$ and $i_{k'}$ is linked, unless $j = j'$ and $k = k'$ or $j = k'$ and $k = j'$. Thus, 
\begin{flalign}
\label{delta3estimate} 
\mathbb{P} \big[ \Delta_3 \big] \le \binom{N}{r} \binom{r^2}{r - 1} \big( C_2 N^{- \varepsilon / 10 - 1} \big)^{r - 1} \le C_4 N^{- c_4 \log \log N},  
\end{flalign}

\noindent for some constants $c_4, C_4 > 0$; here, we have again used \eqref{pijestimate1}. 

Now, the lemma follows from summing \eqref{delta1estimate} and \eqref{delta3estimate}. 
\end{proof}

\subsection{Results and Reductions}

\label{deviantnondeviantlocallaw}

Our next goal is to analyze individual entries of the resolvent $\textbf{G} = \{ G_{ij} \}$, after conditioning on its $AB$ label $\textbf{L} (\textbf{H})$ of $\textbf{H}$. To that end, for any $c, C, \xi > 0$ and integers $i, j \in [1, N]$, we define the events 
\begin{flalign}
\label{omegaij}
\begin{aligned}
\Omega_{C; \xi}^{(c)} (i, j) & = \Omega_{C; \xi}^{(c)} (i, j; \textbf{H}, z)  = \bigg\{ \big| G_{ij} (z) - \textbf{1}_{i = j} m_{\semicircle} \big| \ge C (\log N)^{3 \xi} \left( \displaystyle\frac{1}{\sqrt{N \eta}} + N^{-c \varepsilon} \right) \bigg\}; \\ 
\Omega_C (i, j) & = \Omega_C (i, j; \textbf{H}, z) = \Big\{ \big| G_{ij} (z) \big| \ge C \Big\} . 
\end{aligned} 
\end{flalign} 

As outlined in Section \ref{SmallMoment}, to prove Theorem \ref{estimate1gij} we will show that $\Omega_{C; \xi}^{(c)} (i, j)$ holds with very high probability when $i$ or $j$ is typical (with respect to $\textbf{L}$) and that $\Omega_C (i, j)$ holds with very high probability if both $i$ and $j$ are deviant (with respect to $\textbf{L}$). This is stated more carefully in the following theorem. In what follows $\mathbb{P}_{\textbf{H} (\textbf{L})}$ denotes the probability distribution with respect to the symmetric random matrix $\textbf{H} (\textbf{L})$ (from Definition \ref{hl1}). 

\begin{thm}

\label{localmomentsdeltalabel} 

Fix $\kappa > 0$ and $E \in [\kappa - 2, 2 - \kappa]$. Let $N$ be a positive integer, and take $\eta \in \mathbb{R}_{> 0}$ such that $N \eta > (\log N)^{8 \log \log N}$; denote $z = E + \textbf{\emph{i}} \eta$. Fix an admissible $N \times N$ $AB$ label $\textbf{\emph{L}} \in \mathcal{A}$. Let $\textbf{\emph{H}}$ be an $N \times N$ real generalized Wigner matrix (as in Definition \ref{momentassumption}).  

Then, there exist constants $c, C, \xi > 0$ (only dependent on $\kappa$, $c_1$, $C_1$, $C_2$, and $\varepsilon$ from Definition \ref{momentassumption}) such that the following estimates hold. 

\begin{enumerate}

\item{\label{nondeviantlocalmoments1} If $i \in [1, N]$ is typical with respect to $\textbf{\emph{L}}$, then for each $j \in [1, N]$ we have that 
\begin{flalign}
\label{typicalomega}
\mathbb{P}_{\textbf{\emph{H}} (\textbf{\emph{L}})} \big[ \Omega_{C; \xi}^{(c)} (i, j; \textbf{\emph{H}}, z) \big] \le C \big( -c (\log N)^{\xi} \big). 
\end{flalign}
}

\item{\label{deviantlocalmoments1} If $i \in [1, N]$ is deviant with respect to $\textbf{\emph{L}}$, then for each $j \in [1, N]$ we have that 
\begin{flalign} 
\label{deviantomega}
\mathbb{P}_{\textbf{\emph{H}} (\textbf{\emph{L}})}  \big[ \Omega_C (i, j; \textbf{\emph{H}}, z) \big] \le C \exp \big( - c (\log N)^{\xi} \big) . 
\end{flalign}
} 

\end{enumerate}
\end{thm} 

Let us see how Theorem \ref{estimate1gij} can be established, assuming Theorem \ref{localmomentsdeltalabel}. 

\begin{proof}[Proof of Theorem \ref{estimate1gij} Assuming Theorem \ref{localmomentsdeltalabel}]

To establish Theorem \ref{estimate1gij}, we first show that the estimates \eqref{gijestimate} and \eqref{gijestimate2} hold when $z \in \mathscr{D}_{\kappa; N}$ is fixed; then, we use a union estimate to establish these estimates after taking the supremum over $z \in \mathscr{D}_{\kappa; N}$, as originally stated above. 

To implement the first part, fix $z \in \mathscr{D}_{\kappa; N}$ and sample the generalized Wigner matrix $\textbf{H}$ according to Sampling \ref{lh}, that is, first sample an $\textbf{H}$-distributed $AB$ label $\textbf{L}$ and then sample an $\textbf{L}$-distributed symmetric random matrix $\textbf{H}$. Denote by $\widetilde{c}, \widetilde{C}, \widetilde{\xi}$ the constants $c, C, \xi$ from Theorem \ref{localmomentsdeltalabel}, respectively. Let us restrict to the complement $\overline{\Omega} (z)$ of the event 
\begin{flalign*}
\Omega (z) =  \bigcup_{i, j \in \mathcal{T}} \Omega_{\widetilde{C}; \widetilde{\xi}}^{(\widetilde{c})} (i, j) \cup \bigcup_{i, j \in [1, N]} \Omega_{\widetilde{C}} (i, j).
\end{flalign*}

\noindent By Lemma \ref{probabilitydelta}, \eqref{typicalomega}, and \eqref{deviantomega} we have that
\begin{flalign}
\label{omegaprobability}
\mathbb{P} [\Delta] \le \widetilde{C} N^{-\widetilde{c} \log \log N}; \qquad \mathbb{P}_{\textbf{H} (\textbf{L})} \big[ \Omega (z) \big] \le \widetilde{C} N^2 \exp \big( - \widetilde{c} (\log N)^{\widetilde{\xi}} \big), 
\end{flalign}

\noindent after altering $\widetilde{c}, \widetilde{C}, \widetilde{\xi}$ if necessary. Restricting to $\overline{\Omega} (z)$, we have that 
\begin{flalign}
\label{gijestimatel}
\begin{aligned}
& \big| G_{ij} (z) - \textbf{1}_{i = j} m_{\semicircle} (z) \big| < \widetilde{C} (\log N)^{3 \widetilde{\xi}} \left( \displaystyle\frac{1}{\sqrt{N \eta}} + N^{- \widetilde{c} \varepsilon} \right) \quad \text{if $i$ is typical}; \\
& \big| G_{ij} (z) \big| < \widetilde{C} \qquad \qquad \qquad \qquad \qquad \qquad \qquad \qquad \qquad \quad \text{if $i$ is deviant}. 
 \end{aligned} 
\end{flalign}

Further restricting to the complement $\overline{\Delta}$ of the event $\Delta$, we have that $\textbf{L}$ is admissible. Therefore, there exist less than $N^{1 - \varepsilon / 20}$ deviant indices with respect to $\textbf{L}$. Combining this with the estimates \eqref{omegaprobability} and \eqref{gijestimatel} yields 
\begin{flalign}
\label{zgijestimate}
\begin{aligned}
& \mathbb{P}_{\textbf{H}} \Bigg[   \displaystyle\max_{1 \le i, j \le N} \big| G_{ij} (z) \big| > \widetilde{C} \Bigg] < \widetilde{C} N^{- \widetilde{c} \log \log N}; \\
& \mathbb{P}_{\textbf{H}} \Bigg[ \bigcup_{\substack{I \subseteq [1, N] \\ |I| \ge s}} \bigg\{ \displaystyle\min_{i \in I} \displaystyle\max_{1 \le j \le N} \big| G_{i j} (z) - \textbf{1}_{i = j} m_{\semicircle} (z) \big| > \widetilde{C} (\log N)^{3 \widetilde{\xi}} \Big( \displaystyle\frac{1}{\sqrt{N \eta}} + N^{-\widetilde{c} \varepsilon} \Big) \bigg\} \Bigg] < \widetilde{C} N^{- \widetilde{c} \log \log N}, 
\end{aligned}
\end{flalign}

\noindent where we set $s = s_N = \lceil N^{1 - \varepsilon / 20} \rceil $. 

The estimates \eqref{zgijestimate} hold for each fixed $z \in \mathscr{D}_{\kappa; N}$. To establish the stronger results claimed in Theorem \ref{estimate1gij}, we must take the supremum over all $z \in \mathscr{D}_{\kappa; N}$. To that end, define the sublattice $\mathbb{L}_{\kappa; N} = \big\{ z \in \mathscr{D}_{\kappa; N} : N^{10} \Re z \in \mathbb{Z}, N^{10} \Im z \in \mathbb{Z} \big\}$. Then, from a union estimate, it follows that 
\begin{flalign}
\label{zgijestimate1} 
\begin{aligned}
& \mathbb{P}_{\textbf{H}} \Bigg[   \displaystyle\max_{1 \le i, j \le N} \displaystyle\sup_{z \in \mathbb{L}_{\kappa; z}} \big| G_{ij} (z) \big| > \widetilde{C} \Bigg] < 100 \widetilde{C} N^{20 - \widetilde{c} \log \log N}; \\
& \mathbb{P}_{\textbf{H}} \Bigg[ \bigcup_{\substack{I \subseteq [1, N] \\ |I| \ge s}} \bigg\{ \displaystyle\min_{i \in I} \displaystyle\max_{1 \le j \le N} \displaystyle\sup_{z \in \mathbb{L}_{\kappa; z}} \big| G_{i j} (z) - \textbf{1}_{i = j} m_{\semicircle} (z) \big| > \widetilde{C} (\log N)^{3 \widetilde{\xi}} \Big( \displaystyle\frac{1}{\sqrt{N \eta}} + N^{-\widetilde{c} \varepsilon} \Big) \bigg\} \Bigg] \\
& \qquad \qquad \qquad \qquad \qquad \qquad \qquad \qquad \qquad \qquad \qquad \qquad \qquad \qquad < 100 \widetilde{C} N^{20 - \widetilde{c} \log \log N}. 
\end{aligned}
\end{flalign}

\noindent Now \eqref{gijestimate} and \eqref{gijestimate2} follow from \eqref{zgijestimate1} and the fact that 
\begin{flalign}
\label{ztildez} 
\big| G_{ij} (z) - G_{ij} (\widetilde{z}) \big| \le N^{-1}, 
\end{flalign}

\noindent if $z, \widetilde{z} \in \mathbb{H}$ satisfy $|z - \widetilde{z}| < N^{-4}$ and $\Im z, \Im \widetilde{z} > N^{-1}$; \eqref{ztildez} is a consequence of the resolvent identity \eqref{resolvent} and the deterministic estimate \eqref{gijeta}. 
\end{proof}

Thus, it suffices to establish Theorem \ref{localmomentsdeltalabel}. This will be the topic of the next several sections.

\section{The Initial Estimate}

\label{LawLarge}

The purpose of this section is to establish Theorem \ref{localmomentsdeltalabel} when $\eta$ is relatively large, of order $1$; this is the content of Proposition \ref{locallargeeta}. 

To establish this proposition, we follow the outline from Section \ref{DiagonalResolvent}. To that end, in accordance with \eqref{vmfedsmall}, we first begin with a high probability estimate on $|F_i|$, $|E_i|$, $|D_i|$, and $|h_{ii}|$ (recall \eqref{edmf}) for deviant indices $i$. This estimate is provided by the following lemma, under the assumption that each resolvent entry $G_{jk}^{(i)}$ is already bounded by some $U > 0$. 

\begin{lem}

\label{edfhsmalletalarge}

Fix $\kappa > 0$, $U > 1$, and $E \in [\kappa - 2, 2 - \kappa]$; also, let $\eta \in \mathbb{R}_{> 0}$, and denote $z = E + \textbf{\emph{i}} \eta$. Fix a positive integer $N$ and an admissible $N \times N$ $AB$ label $\textbf{\emph{L}} \in \mathcal{A}$; fix a typical index $i \in \mathcal{T}_{\textbf{\emph{L}}}$. Let $\textbf{\emph{H}}$ be an $N \times N$ generalized Wigner matrix. Recall the definition of $t_i$ from assumption \ref{stochastic} and of $F_i$, $E_i$, and $D_i$ from \eqref{edmf}. 

Then, there exist constants $C, \nu > 0$ (only dependent on $C_1$ and $C_2$) such that
\begin{flalign}
\label{uedfhsmall}
\begin{aligned}
\mathbb{P}_{\textbf{\emph{H}} (\textbf{\emph{L}})} \Bigg[ |\Gamma_i| \displaystyle\prod_{\substack{1 \le j, k \le N \\ j, k \ne i}} \textbf{\emph{1}}_{|G_{jk}^{(i)}| \le U} \ge C U & (\log N)^{2 \xi} \Big( N^{-\varepsilon / 10} + \displaystyle\frac{1}{\sqrt{N \eta}} \Big)  \Bigg] \le \exp \big( - \nu (\log N)^{\xi} \big),
\end{aligned}  
\end{flalign}

\noindent for any $2 \le \xi \le \log \log N$, where for each $i$ we have set
\begin{flalign}
\label{gammai}
\Gamma_i = F_i + E_i + D_i - h_{ii} + m_{\semicircle} t_i.
\end{flalign}
\end{lem}

\begin{proof}

We will establish this lemma by individually bounding each term $|F_i|$, $|E_i|$, $|D_i|$, $|h_{ii}|$, and $|m_{\semicircle} t_i|$ with large probability. To that end first observe that $\big| m_{\semicircle} t_i \big| < C_1 |m_{\semicircle}| N^{-\varepsilon}$ by assumption \ref{stochastic}. Furthermore, observe that, since $i$ is typical with respect to $\textbf{L}$, we have that $|h_{ii}| < N^{- \varepsilon / 10}$ also holds deterministically. 

Next, consider $D_i$. From \eqref{gkj}, it follows that 
\begin{flalign}
\label{etalargedi}
|D_i| \le \displaystyle\sum_{\substack{1 \le j \le N \\ j \ne i}} s_{ij} \Big| G_{jj}^{(i)} - G_{jj} \Big| = \displaystyle\sum_{j = 1}^N \displaystyle\frac{s_{ij} \big| G_{ij} \big|^2}{\big| G_{ii} \big|} \le \displaystyle\frac{C_1}{N |G_{ii}| } \displaystyle\sum_{j = 1}^N \big| G_{ij} \big|^2 \le \displaystyle\frac{C_1}{N \eta}, 
\end{flalign}

\noindent where we applied assumption \ref{generalized} to deduce the third estimate in \eqref{etalargedi}, and we applied Ward's identity \eqref{sumgij} to deduce the fourth estimate. Thus, \eqref{etalargedi} provides a deterministic bound on $|D_i|$. 

We will bound the remaining terms $|E_i|$ and $|F_i|$ with very high probability, using \eqref{bi2} and \eqref{bij2} from Corollary \ref{largeprobability2}. 

We first address $|F_i|$. To that end, recall the definitions from Section \ref{EntryTruncation} (in particular, the definitions of $a_{ij}$ and $b_{ij}$ from Definition \ref{abdefinition}, the definition of $\textbf{H} (\textbf{L})$ from Definition \ref{hl1}, and Sampling \ref{lh}). Since $i \in \mathcal{T}$ is typical, the $(i, j)$ entry of $\textbf{H} (\textbf{L})$ is $a_{ij}$ for each $j \in [1, N]$. 
Now, from Lemma \ref{entrysmaller} we have that $\big| \mathbb{E} [a_{ij}] \big| \le 2 C_2 N^{-1 - \varepsilon / 10}$. Moreover, 
\begin{flalign}
\label{qhijsmallmoments}
\mathbb{E} \big[ |a_{ij}|^p \big] \le q^{2 - p} \mathbb{E} \big[ |a_{ij}|^2 \big] \le \displaystyle\frac{C_1}{N  q^{p - 2}} \le \displaystyle\frac{q^2}{N} \left( \displaystyle\frac{C_1}{q} \right)^p, 
\end{flalign}

\noindent for any $p \ge 2$, where we have set $q = N^{\varepsilon / 10}$. To establish the first estimate in \eqref{qhijsmallmoments}, we used the fact that $|a_{ij}| < q^{-1}$ deterministically; to establish the second estimate, we used the fact that $\Var a_{ij} \le \Var h_{ij} \le C_1 N^{-1}$; and to establish the third estimate, we used the fact that $C_1 > 1$. 

Thus, we can apply \eqref{bij2} with $R_{jk} = G_{jk}^{(i)}$; $X_j = a_{ij}$; and $N^{- \delta} = N^{- \varepsilon / 10} = q^{-1}$. This yields 
\begin{flalign*}
\mathbb{P} \Bigg[ \bigg| \displaystyle\sum_{\substack{ 1 \le j \ne k \le N \\ j, k \ne i}} a_{ij} G_{jk}^{(i)} a_{ki} \bigg| \ge (\log N)^{2 \xi} \bigg( (20 C_2^2 + 1) q^{-1} \displaystyle\max_{\substack{1 \le j \ne k \le N \\ j, k \ne i}} \big| G_{jk}^{(i)} \big| & + \Big( \displaystyle\frac{1}{N^2} \displaystyle\sum_{\substack{1 \le j \ne k \le N\\ j, k \ne i}} \big| G_{jk}^{(i)} \big|^2 \Big)^{1 / 2} \bigg)  \Bigg] \\
& \le \exp \big( -\widetilde{\nu} (\log N)^{\xi} \big), 
\end{flalign*}

\noindent where $\widetilde{\nu} = \nu (C_1, 2C_2)$ from Proposition \ref{largeprobability2}. Applying Ward's identity \eqref{sumgij} and the definition \eqref{edmf} of $F_i$ yields 
\begin{flalign}
\label{flargeeta1}
\begin{aligned}
\mathbb{P} \Bigg[ |F_i| \ge (\log N)^{2 \xi} \bigg( (20 C_2^2 + 1) N^{-\varepsilon / 10} \displaystyle\max_{\substack{1 \le j \ne k \le N \\ j, k \ne i}} \big| G_{jk}^{(i)} \big| + \displaystyle\frac{1}{\sqrt{N \eta}} & \Big( \displaystyle\frac{1}{N} \displaystyle\sum_{\substack{1 \le j \le N\\ j, k \ne i}} \big| G_{jj}^{(i)} \big| \Big)^{1 / 2} \bigg)  \Bigg] \\
& \le \exp \big( -\widetilde{\nu} (\log N)^{\xi} \big), 
\end{aligned}
\end{flalign}

\noindent for all $2 \le \xi \le \log \log N$. In particular, combining \eqref{flargeeta1} with the estimate $\big| G_{jk}^{(i)} \big| \le U$ (due to the factors of $\textbf{1}_{|G_{jk}^{(i)}| \le U}$ on the left side of \eqref{uedfhsmall}), and the fact that $C_2, U > 1$, we deduce that 
\begin{flalign}
\label{flargeeta}
\mathbb{P} \left[ |F_i| > 21 C_2^2 U (\log N)^{2 \xi} \bigg( \displaystyle\frac{1 }{N^{\varepsilon / 10}} + \displaystyle\frac{1}{\sqrt{N \eta}} \bigg) \right] \le \exp \big( - \widetilde{\nu} (\log N)^{\xi} \big), 
\end{flalign}

\noindent which provides an estimate on $|F_i|$. 

To estimate $|E_i|$, we use \eqref{bi2} with $R_j = G_{jj}^{(i)}$; $s_j = s_{ij}$; $X_j = h_{ij}$; and $N^{- \delta} = N^{- \varepsilon / 10} = q^{-1}$ to deduce that 
\begin{flalign}
\label{elargeeta1}
\begin{aligned}
\mathbb{P} \Bigg[ \bigg| \displaystyle\sum_{\substack{1 \le j \le N \\ j \ne i}} G_{jj}^{(i)} \big( |a_{ij}|^2 - s_{ij}  \big) \bigg| \ge (20 C_2^2+ 1) (\log N)^{\xi} N^{-\varepsilon / 10} \displaystyle\max_{\substack{1 \le j \le N \\ j \ne i}} \big| G_{jj}^{(i)} \big|  \Bigg] \le \exp \big( - \widetilde{\nu} (\log N)^{\xi} \big). 
\end{aligned}
\end{flalign}

\noindent Combining \eqref{elargeeta1} with the definition \eqref{edmf} of $E_i$, the estimate $\big| G_{jk}^{(i)} \big| \le U$, and the fact that $C_2 > 1$, we deduce that  
\begin{flalign}
\label{elargeeta}
\begin{aligned}
\mathbb{P} \left[ \big| E_i \big| \ge 21 C_2^2 U (\log N)^{\xi} N^{-\varepsilon / 10} \right] \le \exp \big( - \widetilde{\nu} (\log N)^{\xi} \big). 
\end{aligned}
\end{flalign}

Now, the existence of $C$ and $\nu$ (only dependent on $C_1$ and $C_2$) satisfying \eqref{uedfhsmall} follows from summing \eqref{etalargedi}, \eqref{flargeeta}, \eqref{elargeeta}, and the deterministic estimates $|t_i m_{\semicircle} | < C_1 |m_{\semicircle} N^{-\varepsilon}$ and $|h_{ii}| \le N^{- \varepsilon / 10}$. 
\end{proof}

Using Proposition \ref{edfhsmalletalarge}, we can establish Theorem \ref{localmomentsdeltalabel} in the case when $\eta$ is sufficiently large.

\begin{prop}

\label{locallargeeta}

Fix $\kappa > 0$, and let $E \in [\kappa - 2, 2 - \kappa]$; also, let $\eta \in \mathbb{R}_{> 0}$, and denote $z = E + \textbf{\emph{i}} \eta$. Fix a positive integer $N$ and an admissible $N \times N$ $AB$ label $\textbf{\emph{L}} \in \mathcal{A}$. Let $\textbf{\emph{H}}$ be an $N \times N$ generalized Wigner matrix. Recall the definitions of $F_i$, $E_i$, and $D_i$ from \eqref{edmf}, and the definitions of $\Omega_{C; \xi}^{(c)} (i, j)$ and $\Omega_C (i, j)$ from \eqref{omegaij}. 

Then, there exist constants $C, \nu > 0$ (only dependent on $C_1$ and $C_2$) such that the following estimates hold for any $2 \le \xi \le \log \log N$ and sufficiently large $N$ (in comparison to $C_1$ and $\varepsilon^{-1}$).  

\begin{enumerate}

\item{\label{nondeviantlocalmoments} If $\eta > C$ and $i \in [1, N]$ is typical with respect to $\textbf{\emph{L}}$, then for each $j \in [1, N]$ we have that 
\begin{flalign}
\label{typicalomegaetalarge}
\mathbb{P} \big[ \Omega_{C; \xi}^{(1 / 20)} (i, j) \big| \textbf{\emph{L}} (\textbf{\emph{H}}) = \textbf{\emph{L}} \big] \le \big( - \nu (\log N)^{\xi} \big). 
\end{flalign}
}

\item{\label{deviantlocalmoments} If $\eta > C$ and $i \in [1, N]$ is deviant with respect to $\textbf{\emph{L}}$, then for each $j \in [1, N]$ we have that 
\begin{flalign} 
\label{deviantomegaetalarge}
\mathbb{P} \big[ \Omega_1 (i, j) \big| \textbf{\emph{L}} (\textbf{\emph{H}}) = \textbf{\emph{L}} \big] \le \exp \big( - \nu (\log N)^{\xi} \big) . 
\end{flalign}
} 

\end{enumerate}

\end{prop}

\begin{proof}

First, observe that if $C > 1$, then we have from the deterministic estimate \eqref{gijeta} that $\big| G_{ij} \big| < \eta^{-1} < 1$; this implies \eqref{deviantomegaetalarge}. 

Hence, it suffices to establish \eqref{typicalomegaetalarge}, so let $i \in [1, N]$ be typical with respect to $\textbf{L}$. Denoting $\Gamma_i$ as in \eqref{gammai}, we deduce from \eqref{vmfed} that
\begin{flalign}
\label{vi1}
v_i = \displaystyle\frac{m_{\semicircle} \sum_{j = 1}^N s_{ij} v_j + m_{\semicircle} \Gamma_i}{m_{\semicircle}^{-1} - \Gamma_i - \sum_{j = 1}^N s_{ij} v_j}. 
\end{flalign}

We will next show that the denominator on the right side of \eqref{vi1} is large and that its numerator is small. To that end, we first require an estimate on $\Gamma_i$. Let $\widetilde{C}$ denote the constant $C$ from \eqref{uedfhsmall}. Setting $C > \widetilde{C}$ yields $\big| G_{jk}^{(i)} \big| < \eta^{-1} < \widetilde{C}^{-1}$. Thus, applying \eqref{uedfhsmall} with $U = \widetilde{C}^{-1}$ yields
\begin{flalign}
\label{gammaietalarge}
\mathbb{P} \left[ \big| \Gamma_i \big| \le (\log N)^{2 \xi} \bigg( N^{-\varepsilon / 10} + \displaystyle\frac{1}{\sqrt{N \eta}} \bigg) \right] \le \exp \big( - \widetilde{\nu} \big( \log N \big)^{\xi} \big),
\end{flalign}

\noindent where $\widetilde{\nu}$ is the constant $\nu$ from Proposition \ref{uedfhsmall}. 

Assume further that $C > 30$. Then, it is quickly derived from \eqref{mquadratic} that $| m_{\semicircle} | \le 2 \eta^{-1} < 1 / 10$, from which it follows that $|v_i| \le \big| G_{ii} \big| + |m_{\semicircle}| \le 3 \eta^{-1} < 1 / 10$ from \eqref{gijeta}. 

Hence, if we restrict to the event that $\Gamma_i$ is small in the sense of \eqref{gammaietalarge} and assume that $N$ is sufficiently large in comparison to $\varepsilon^{-1}$ such that 
\begin{flalign*}
\left| (\log N)^{2 \log \log N} \bigg( N^{-\varepsilon / 10} + \displaystyle\frac{1}{\sqrt{N \eta}} \bigg) \right| \le 1, 
\end{flalign*}

\noindent then it follows that  
\begin{flalign*}
\left| m_{\semicircle}^{-1} - \Gamma_i - \sum_{j = 1}^N s_{ij} v_j \right| \ge \big| m_{\semicircle}^{-1} \big| - \big| \Gamma_i \big| - \left| \sum_{j = 1}^N s_{ij} v_j \right| \ge 10 - \big| \Gamma_i \big| - (1 + C_1 N^{-\varepsilon} ) \displaystyle\max_{1 \le j \le N} |v_j| > 1,
\end{flalign*} 

\noindent where we have used the fact that $\big| \sum_{j = 1}^N s_{ij} - 1 \big| = |t_i| \le C_1 N^{-\varepsilon}$ by assumption \ref{stochastic}. Inserting this (and the estimate $|m_{\semicircle}| < 1 / 10$) into \eqref{vi1} yields 
\begin{flalign}
\label{vilargeeta1}
\big| v_i \big| \le \left| \displaystyle\frac{1}{10} \displaystyle\sum_{j = 1}^N s_{ij} v_j + m_{\semicircle} \Gamma_i \right| \le \displaystyle\frac{1}{10} \displaystyle\sum_{j = 1}^N s_{ij} | v_j | + \big| \Gamma_i \big|, 
\end{flalign}

\noindent for each $i \in \mathcal{T}_{\textbf{L}}$. Hence, 
\begin{flalign}
\label{vilargeeta2}
\big| v_i \big| & \le \displaystyle\frac{1}{10} \displaystyle\sum_{j \in \mathcal{T}} s_{ij} | v_j | + \displaystyle\frac{1}{10} \displaystyle\sum_{j \in \mathcal{D}} s_{ij} | v_j | + \big| \Gamma_i \big| \le \displaystyle\frac{1}{10} \displaystyle\sum_{j \in \mathcal{T}} s_{ij} | v_j | + C_1 N^{-\varepsilon / 20} + \big| \Gamma_i \big|,
\end{flalign}

\noindent where we recalled the facts that $\big| \mathcal{D} \big| < N^{1 - \varepsilon / 20}$ (since $\textbf{L}$ is admissible), that $s_{ij} < C_1 N^{-1}$ (from assumption \ref{generalized}), and that $|v_j| < 1 / 10 < 1$ for all $j \in [1, N]$. 

Using \eqref{vilargeeta2} and the fact that $\big| \sum_{j = 1}^N s_{ij} - 1 \big| = |t_i| < C_1 N^{-\varepsilon}$, we deduce that
\begin{flalign}
\label{vilargeeta3}
\displaystyle\max_{j \in \mathcal{T}} \big| v_j \big| & \le \displaystyle\frac{1}{5} \displaystyle\max_{j \in \mathcal{T}} |v_j| + C_1 N^{- \varepsilon / 20} + \displaystyle\max_{j \in \mathcal{T}} \big| \Gamma_j \big|. 
\end{flalign} 

Using \eqref{vilargeeta3} and applying \eqref{gammaietalarge} for all $i \in \mathcal{T}$ and a union estimate, we obtain that 
\begin{flalign*}
\mathbb{P} \left[ \displaystyle\max_{j \in \mathcal{T}} |v_j| \le 2 C_1 (\log N)^{2 \xi} \bigg( N^{-\varepsilon / 20} + \displaystyle\frac{1}{\sqrt{N \eta}} \bigg) \right] \le N \exp \big( - \widetilde{\nu} (\log N)^{\xi} \big), 
\end{flalign*}

\noindent from which \eqref{typicalomegaetalarge} quickly follows.  
\end{proof}

\section{The Multiscale Argument for Typical Indices} 

\label{LawSmallNotDeviant}

Our next goal is to provide a framework for establishing the local semicircle law for typical indices, given by \eqref{typicalomega}. This is done through the following two propositions, whose proofs are similar to that of the local semicircle for sparse graphs from \cite{SSG}. 

What these lemmas approximately yield are estimates on the probability of the event that the local semicircle law \eqref{typicalomega} does not hold for some typical entry of the resolvent $\textbf{G} (z)$, in terms of probabilities of the events the local semicircle law does not hold for some entry of a different resolvent $\textbf{G} (z')$, where $\Im z' > \Im z$. In terms of the notation from \eqref{omegaij}, this can be restated as an estimate on $\mathbb{P} \big[ \Omega_{C; \xi}^{(c)} (i, j; \textbf{H}, z) \big]$ in terms of $\mathbb{P} \big[ \Omega_{C; \xi}^{(c)} (i, j; \textbf{H}, z') \big]$ and $\mathbb{P} \big[ \Omega_C (i, j; \textbf{H}, z') \big]$, conditional on some admissible $AB$ label $\textbf{L}$ of $\textbf{H}$, if $i \in [1, N]$ is typical with respect to $\textbf{L}$.

Since $\Im z' > \Im z$, this suggests that repeated application of these two propositions might estimate $\mathbb{P} \big[ \Omega_{C; \xi}^{(c)} (i, j; \textbf{H}, z) \big]$ in terms of $\mathbb{P} \big[ \Omega_{C; \xi}^{(c)} (i, j; \textbf{H}, z') \big]$ and $\mathbb{P} \big[ \Omega_C (i, j; \textbf{H}, z') \big]$, where $\Im z'$ is very large; then, we could apply the initial estimate Proposition \ref{locallargeeta}. Indeed, this is what we will do in Section \ref{ProofLocalLaw}, after obtaining an analog of the two propositions below for deviant indices in Section \ref{LawSmallDeviant}. 

In what follows, we recall that $\mathbb{P}_{\textbf{H} (\textbf{L})}$ is the probability measure the random matrix $\textbf{H} (\textbf{L})$ from Definition \ref{hl1}. Furthermore, for any subset $\mathcal{S} \subseteq [1, N]$, we let $\mathbb{P}_{\textbf{H}^{(\mathcal{S})} (\textbf{L}^{(\mathcal{S})})}$ denote the probability measure with respect to the symmetric random matrix $\textbf{H}^{(\mathcal{S})} \big( \textbf{L}^{(\mathcal{S})} \big)$, obtained from $\textbf{H} (\textbf{L})$ by removing all rows and columns indexed by some $j \in \mathcal{S}$.  

\begin{prop}

\label{incrementnotdeviant} 

Fix $\kappa > 0$ and $U_1, U_2 > 1$, and let $E \in [\kappa - 2, 2 - \kappa]$; also, let $\eta \in \mathbb{R}_{> 0}$ such that $N \eta > (\log N)^{8 \log \log N}$, and denote $z = E + \textbf{\emph{i}} \eta$. Fix a positive integer $N$, let $\textbf{\emph{H}}$ be an $N \times N$ generalized Wigner matrix. Furthermore, fix an admissible $N \times N$ $AB$ label $\textbf{\emph{L}} \in \mathcal{A}$. Recall the definitions of $v_i$ from \eqref{vj}, and of $\Omega_{C; \xi}^{(c)} (i, j)$ and $\Omega_C (i, j)$ from \eqref{omegaij}. Denote
\begin{flalign}
\label{probabilitynotdeviantevent}
\begin{aligned}
P_i & = \displaystyle\sum_{\substack{1 \le j, k \le N \\ j, k \ne i}} \mathbb{P}_{\textbf{\emph{H}}^{(i)} (\textbf{\emph{L}}^{(i)})}  \Bigg[ \Omega_{U_2} \bigg( j, k; \textbf{\emph{H}}^{(i)}; E + \textbf{\emph{i}} \eta \Big( 1 + \displaystyle\frac{1}{(\log N)^2} \Big)  \bigg) \Bigg] \\
& \qquad + \displaystyle\sum_{j = 1}^N \mathbb{P}_{\textbf{\emph{H}} (\textbf{\emph{L}})} \Bigg[ \Omega_{U_1; \xi}^{(1 / 20)} \bigg( i, j; \textbf{\emph{H}}, E + \textbf{\emph{i}} \Big( 1 + \displaystyle\frac{1}{(\log N)^2} \Big) \bigg)\Bigg] \\
& \qquad + \displaystyle\sum_{1 \le j, k \le N} \mathbb{P}_{\textbf{\emph{H}} (\textbf{\emph{L}})} \Bigg[ \Omega_{U_2} \bigg( j, k; \textbf{\emph{H}}, E + \textbf{\emph{i}} \Big( 1 + \displaystyle\frac{1}{(\log N)^2} \Big) \bigg) \Bigg]. 
\end{aligned}
\end{flalign}

\noindent Then, there exist constants $C, \nu > 0$ (only dependent on $C_1$, $C_2$, and $U_2$) such that
\begin{flalign}
\label{linearvnotdeviantsmall4}
\mathbb{P}_{\textbf{\emph{H}} (\textbf{\emph{L}})} \Bigg[ \displaystyle\max_{j \in \mathcal{T}} |v_j| & \ge C (\log N)^{3 \xi} \left( N^{- \varepsilon / 20 } + \displaystyle\frac{1}{\sqrt{N \eta}} \right) \Bigg]  \le \exp \big( - \nu (\log N)^{\xi} \big) + \displaystyle\sum_{i = 1}^N P_i, 
\end{flalign}

\noindent for any $2 \le \xi \le \log \log N$ and sufficiently large $N$ (in comparison to $C_1$, $C_2$, $U_1$, and $U_2$). 

\begin{rem}

\label{u1u2c} 

Observe in the above proposition that the constants $C$ and $\mu$ are dependent on $U_2$ but independent of $U_1$; however, the minimal value of $N$ for which the result holds depends on both $U_1$ and $U_2$. 
\end{rem}

\end{prop}

\begin{proof}

In what follows, we fix a typical index $i \in \mathcal{T}_{\textbf{L}}$. Recall the definitions of $F_i$, $D_i$, and $E_i$ from \eqref{edmf}. Let us restrict to the event 
\begin{flalign}
\label{notdeviantevent}
\begin{aligned}
\overline{\Omega_i} & = \bigcap_{\substack{j, k \in [1, N] \\ j, k \ne i}} \overline{\Omega}_{U_2} \bigg( j, k; \textbf{H}^{(i)}; E + \textbf{i} \eta \Big( 1 + \displaystyle\frac{1}{(\log N)^2} \Big)  \bigg) \cap \bigcap_{j = 1}^N \overline{\Omega}_{U_1; \xi}^{(1 / 20)} \bigg( i, j; \textbf{H}, E + \textbf{i} \Big( 1 + \displaystyle\frac{1}{(\log N)^2} \Big) \bigg) \\
& \qquad \cap \bigcap_{j, k \in [1, N]} \overline{\Omega}_{U_2} \bigg( j, k; \textbf{H}, E + \textbf{i} \Big( 1 + \displaystyle\frac{1}{(\log N)^2} \Big) \bigg), 
\end{aligned}
\end{flalign}

\noindent where $\overline{E}$ denotes the complement of any event $E$; observe that $\mathbb{P}_{\textbf{H} (\textbf{L})} \big[ \Omega_i \big] \le P_i$. 

We will first obtain a uniform estimate on $G_{jk}^{(i)} (E + \textbf{i} \eta)$. Since we are restricting to the event $\overline{\Omega_i}$, we have that
\begin{flalign}
\label{notdeviantgjk1}
\Bigg| G_{jk}^{(i)} \bigg( E + \textbf{i} \eta \Big( 1 + \displaystyle\frac{1}{(\log N)^2} \Big) \bigg) \Bigg| \le U_2,
\end{flalign}

\noindent for all $j, k \in [1, N]$ with $j, k \ne i$. Thus, assuming that $N > 10$, \eqref{gjjzgjj} implies that 
\begin{flalign}
\label{notdeviantgjj1}
\big| G_{jj}^{(i)} ( E + \textbf{i} \eta ) \big| \le 2 U_2,
\end{flalign}

\noindent for all $j \in [1, N]$ not equal to $i$. Inserting \eqref{notdeviantgjj1} into \eqref{gjkzgjk} yields
\begin{flalign}
\label{notdeviantgjk2}
\Bigg| G_{jk}^{(i)} \bigg( E + \textbf{i} \eta \Big( 1 + \displaystyle\frac{1}{(\log N)^2} \Big) \bigg) - G_{jk}^{(i)} (E + \textbf{i} \eta) \Bigg| \le \displaystyle\frac{2 U_2}{(\log N)^2}, 
\end{flalign}

\noindent for all $j, k \in [1, N]$ with $j, k \ne i$. In view of the fact that $\log N > 2$, \eqref{notdeviantgjk1} and \eqref{notdeviantgjk2} together imply that 
\begin{flalign}
\label{notdeviantgjk3}
\big| G_{jk}^{(i)} ( E + \textbf{i} \eta ) \big| \le \Bigg| G_{jk}^{(i)} \bigg( E + \textbf{i} \eta \Big( 1 + \displaystyle\frac{1}{(\log N)^2} \Big) \bigg) \Bigg| + U_2 \le 2 U_2,
\end{flalign}

\noindent for all $j, k \in [1, N]$ with $j, k \ne i$. 

Having obtained this estimate on $G_{jk}^{(i)}$, we now we apply Proposition \ref{edfhsmalletalarge} with $U$ in that statement replaced by $2 U_2$. In particular, let us insert \eqref{notdeviantgjk3} into \eqref{uedfhsmall}; denote by $\widetilde{C}$ the constant $C$ from the left side of \eqref{uedfhsmall}; and denote by $\widetilde{\nu}$ the constant $\nu$ from the right side of \eqref{uedfhsmall}. Using the fact that $\mathbb{P}_{\textbf{H} (\textbf{L})} \big[ \Omega_i \big] \le P_i$, we deduce that 
\begin{flalign}
\label{notdeviantedfhsmall}
\mathbb{P}_{\textbf{H} (\textbf{L})} \Bigg[ \big| \Gamma_i \big| \ge 2 \widetilde{C} U_2 & (\log N)^{2 \xi} \Big( N^{-\varepsilon / 10} + \displaystyle\frac{1}{\sqrt{N \eta}} \Big)  \Bigg] \le \exp \big( - \widetilde{\nu} (\log N)^{\xi} \big) + P_i, 
\end{flalign}

\noindent where we recall the definition of $\Gamma_i$ from \eqref{gammai}. This bounds the right side of \eqref{vmfed} with very high probability. 

We will now attempt to establish some version of the estimate \eqref{vsmall}. To that end, observe that 
\begin{flalign}
\label{vismall1notdeviant}
\left| \displaystyle\frac{v_i}{1 + m_{\semicircle}^{-1} v_i} - v_i \right| = \left| \displaystyle\frac{v_i^2}{m_{\semicircle} + v_i} \right| \le \displaystyle\frac{2 |v_i|^2}{|m_{\semicircle}|}, \qquad \text{if} \quad |v_i| \le \displaystyle\frac{|m_{\semicircle}|}{2}. 
\end{flalign}

\noindent Let us show that $|v_i| < |m_{\semicircle}| / 2$ indeed holds. Since we are restricting to the event $\overline{\Omega_i}$, we have that 
\begin{flalign}
\label{giigjjsmall1notdeviant1}
\begin{aligned}
\Bigg| G_{ii} \bigg( E + \textbf{i} \eta \Big( 1 + \displaystyle\frac{1}{ (\log N)^2} \Big) \bigg) - m_{\semicircle} \Bigg| & \le U_1 (\log N)^{3 \xi} \left( N^{- \varepsilon / 20} + \displaystyle\frac{1}{\sqrt{N \eta}} \right); \\
 \Bigg| G_{jj} \bigg( E + \textbf{i} \eta \Big( 1 + \displaystyle\frac{1}{ (\log N)^2} \Big) \bigg) \Bigg| & \le U_2, 
\end{aligned} 
\end{flalign}

\noindent for all $j \in [1, N]$. As in \eqref{notdeviantgjk3} (or from the first estimate in \eqref{giigjjsmall1notdeviant1}), we can show that $\big| G_{ii} (E + \textbf{i} \eta) \big| \le 2 U_2$ using the second estimate in \eqref{giigjjsmall1notdeviant1} and \eqref{gjjzgjj}. Inserting this and \eqref{giigjjsmall1notdeviant1} into \eqref{gjkzgjk} yields 
\begin{flalign}
\label{giismall1notdeviant1}
\big| G_{ii} (E + \textbf{i} \eta ) - m_{\semicircle} \big| \le U_1 (\log N)^{3 \xi} \left( N^{- \varepsilon / 20} + \displaystyle\frac{1}{\sqrt{N \eta}} \right) + \displaystyle\frac{2 U_2}{(\log N)^2}, \qquad \text{if $i$ is typical.} 
\end{flalign}

Assuming that $N$ is sufficiently large (in a way that only depends on $U_1$ and $U_2$, since $\xi \le \log \log N$ and $N \eta > (\log N)^{8 \log \log N}$), \eqref{giismall1notdeviant1} implies that $2 |v_i| \le |m_{\semicircle}|$, so that the estimate \eqref{vismall1notdeviant} applies. Inserting \eqref{vismall1notdeviant} into \eqref{vmfed} and applying \eqref{notdeviantedfhsmall}, we deduce that 
\begin{flalign}
\label{linearvnotdeviantsmall}
\begin{aligned}
\mathbb{P} \Bigg[ \bigg| v_i - m_{\semicircle}^2 \displaystyle\sum_{j = 1}^N s_{ij} v_j \bigg| \ge 2 |m_{\semicircle}|^2 \widetilde{C} U_2 (\log N)^{2 \xi} \Big( N^{- \varepsilon / 10 } + & \displaystyle\frac{1}{\sqrt{N \eta}} \Big) + \displaystyle\frac{2 |v_i|^2}{|m_{\semicircle}|}    \Bigg] \\
& \le \exp \big( - \widetilde{\nu} (\log N)^{\xi} \big) + P_i,
\end{aligned}
\end{flalign}

\noindent for any typical index $i \in \mathcal{T}$. 

Now, we would like to restrict the sum $\sum_{j = 1}^N s_{ij} v_j$ in \eqref{linearvnotdeviantsmall} to range over the typical indices $j \in \mathcal{T}$ instead of over all indices $j \in [1, N]$. To that end, we may use similar reasoning as applied to deduce \eqref{notdeviantgjk3} to find that 
\begin{flalign}
\label{notdeviantgij1}
\big| G_{ij} (E + \textbf{i} \eta) \big| \le 2 U_2; \qquad \big| G_{ij} (E + \textbf{i} \eta) - m_{\semicircle} \big| \le 2 U_2 + |m_{\semicircle}|. 
\end{flalign}
	
\noindent Using \eqref{notdeviantgij1}, the fact that $|s_{ij}| \le C_1 N^{-1}$, and the fact that $\big| \mathcal{D} \big| \le N^{1 - \varepsilon / 20}$, we find that 
\begin{flalign}
\label{notdeviantgij2}
\left| \displaystyle\sum_{j = 1}^N s_{ij} v_j - \displaystyle\sum_{j \in \mathcal{T}} s_{ij} v_j \right| \le \big( 2 U_2 + |m_{\semicircle}| \big) C_1 N^{- \varepsilon / 20}. 
\end{flalign}

\noindent Inserting \eqref{notdeviantgij2} into \eqref{linearvnotdeviantsmall}, we deduce that 
\begin{flalign}
\label{linearvnotdeviantsmall1}
\begin{aligned}
\mathbb{P} \Bigg[ \bigg| v_i - m_{\semicircle}^2 \displaystyle\sum_{j \in \mathcal{T}} s_{ij} v_j \bigg| & \ge |m_{\semicircle}|^2 C_1 \big( 2 U_2 + |m_{\semicircle}| \big) N^{- \varepsilon / 20} + \displaystyle\frac{2 |v_i|^2}{|m_{\semicircle}|}   \\
& \quad +  2 |m_{\semicircle}|^2 \widetilde{C} U_2  (\log N)^{2 \xi} \Big( N^{- \varepsilon / 10} + \displaystyle\frac{1}{\sqrt{N \eta}} \Big)  \Bigg] \le \exp \big( - \widetilde{\nu} (\log N)^{\xi} \big) + P_i,
\end{aligned}
\end{flalign}

\noindent for any fixed $i \in \mathcal{T}$. Applying \eqref{linearvnotdeviantsmall1} to all $i \in \mathcal{T}$ and applying a union estimate, we obtain that 
\begin{flalign}
\label{linearvnotdeviantsmall2}
\begin{aligned}
\mathbb{P} \Bigg[ \Big\| \big( \Id - m_{\semicircle}^2 \widetilde{\textbf{S}} \big) \widetilde{\textbf{v}} \Big\|_{\infty} & \ge  |m_{\semicircle}|^2 C_1 \big( 2 U_2 + |m_{\semicircle}| \big) N^{- \varepsilon / 20}  + \displaystyle\frac{2 |v_i|^2}{|m_{\semicircle}|} \\
& \quad + (\log N)^{2 \xi} \Big( N^{- \varepsilon / 10 } + \displaystyle\frac{1}{\sqrt{N \eta}} \Big)  \Bigg]  \le N \exp \big( - \widetilde{\nu} (\log N)^{\xi} \big) + \displaystyle\sum_{i = 1}^N P_i, 
\end{aligned}
\end{flalign}

\noindent where $\widetilde{\textbf{v}} = (v_i)_{i \in \mathcal{T}}$ $|\mathcal{T}|$-dimensional vector, and $\widetilde{\textbf{S}} = \{ s_{jk} \}_{j, k \in \mathcal{T}}$ is a $|\mathcal{T}| \times |\mathcal{T}|$ matrix. 

Now, from \eqref{mtestimate1} of Lemma \ref{testimate} below, we deduce the existence of a constant $\widehat{C} > 0$ (only dependent on $\kappa$) such that $\big\| (\Id - m_{\semicircle}^2 \widetilde{\textbf{S}} )^{-1} \big\| < \widehat{C} \log N$. Inserting this into \eqref{linearvnotdeviantsmall2} yields 
\begin{flalign}
\label{linearvnotdeviantsmall3}
\begin{aligned}
\mathbb{P} \Bigg[ \displaystyle\max_{j \in \mathcal{T}} |v_j| & \ge |m_{\semicircle}|^2 C_1 \big( 2 U_2 + |m_{\semicircle}| \big) \widehat{C} \log N  + 2 |m_{\semicircle}|^2 \widehat{C} \widetilde{C} U_2 (\log N)^{2 \xi + 1} \Big( N^{- \varepsilon / 10} + \displaystyle\frac{1}{\sqrt{N \eta}} \Big) \\
& \qquad + 2 |m_{\semicircle}|^{-1} \widehat{C} \big( \log N \big) \displaystyle\max_{j \in \mathcal{T}} |v_j|^2  \Bigg] \le N \exp \big( - \widetilde{\nu} (\log N)^{\xi} \big) + \displaystyle\sum_{i = 1}^N P_i, 
\end{aligned}
\end{flalign}

Now, by \eqref{giismall1notdeviant1}, we have that 
\begin{flalign*}
\Big| 1 - 2 |m_{\semicircle}|^{-1} \widehat{C} \big( \log N \big) |v_i| \Big| \ge \displaystyle\frac{1}{2},
\end{flalign*}

\noindent for each $i \in \mathcal{T}$ (after restricting to the event $\bigcap_{i \in \mathcal{T}} \overline{\Omega_i}$), assuming that $N$ is sufficiently large (in a way that only depends on $U_1$ and $U_2$). 

Inserting this into \eqref{linearvnotdeviantsmall3} yields the existence of $C, \nu > 0$ (only dependent on $\kappa$, $U_2$, $C_1$, and $C_2$) satisfying \eqref{linearvnotdeviantsmall4}; this implies the proposition. 
\end{proof}

The previous proposition estimates the diagonal terms $\big| G_{ii} \big|$. We must also estimate the off-diagonal terms $\big| G_{ij} \big|$: this is given by the following proposition.  

\begin{prop}

\label{omegaijnotdeviant}

Adopt the notation of Proposition \ref{incrementnotdeviant}. There exist constants $C, \nu > 0$  (only dependent on $\kappa$, $C_1$, $C_2$, and $U_2$) such that
\begin{flalign}
\label{linearvnotdeviantsmall3gij}
\begin{aligned}
\mathbb{P}_{\textbf{\emph{H}} (\textbf{\emph{L}})} \Bigg[ \displaystyle\max_{\substack{i \in \mathcal{T} \\ 1 \le j \le N \\ i \ne j}} \big| G_{ij} \big| & \ge C (\log N)^{3 \xi} \bigg( N^{- \varepsilon / 10}   + \sqrt{ \displaystyle\frac{1}{N \eta} } \bigg) \Bigg] \le \exp \big( - \nu (\log N)^{\xi} \big) + N \displaystyle\sum_{i = 1}^N P_i, 
\end{aligned}
\end{flalign}

\noindent for any $2 \le \xi \le \log \log N$ and all sufficiently large $N$ (in comparison to $\kappa$, $C_1$, $C_2$, $U_1$, and $U_2$). 

\end{prop}

\begin{proof}

We will establish this corollary through \eqref{suma2} from Lemma \ref{largeprobability2}. In particular, fix $i \in \mathcal{T}$ and $j \in [1, N]$; applying that lemma with $R_j = G_{kj}^{(i)}$, $X_j = a_{ij}$, and $N^{\delta} = N^{- \varepsilon / 10} = q^{-1}$ (and using the fact that $\big| \mathbb{E} [a_{ij}] \big| < 2 C_2 N^{-1 - \varepsilon / 10}$ and the estimate \eqref{qhijsmallmoments}) yields 
\begin{flalign}
\label{linearvnotdeviantsmall3gij1}
\begin{aligned} 
\mathbb{P}_{\textbf{H} (\textbf{\emph{L}})} \Bigg[ \bigg| \displaystyle\sum_{\substack{k \in [1, N] \\ k \ne i}} a_{ik} G_{kj}^{(i)} \bigg| \ge (\log N)^{\xi} \bigg( (2 C_2 + 1) N^{- \varepsilon / 10} \displaystyle\max_{\substack{1 \le k \le N \\ k \ne i}} \big| G_{jk}^{(i)} \big| + & \Big( \displaystyle\sum_{\substack{k \in [1, N] \\ k \ne i}} \big| G_{jk}^{(i)} \big|^2 \Big)^{1 / 2} \bigg) \Bigg] \\
& \le \exp \big( -  \widetilde{\nu} (\log N)^{\xi} \big), 
\end{aligned} 
\end{flalign}

\noindent where $\widetilde{\nu}$ is the constant $\nu (C_1, 2 C_2)$ from Lemma \ref{largeprobability2}. In what follows, we restrict to the event $\bigcap_{i = 1}^N \overline{\Omega_i}$, where $\overline{\Omega_i}$ was defined in \eqref{notdeviantevent}. Applying Ward's identity \eqref{sumgij}, we find that 
\begin{flalign}
\label{gjksumnotdeviant}
\displaystyle\frac{1}{N} \displaystyle\sum_{\substack{k \in [1, N] \\ k \ne i}} \big| G_{jk}^{(i)} \big|^2 = \displaystyle\frac{\Im G_{jj}^{(i)}}{N \eta} \le \displaystyle\frac{2 U_2}{N \eta},
\end{flalign}

\noindent where we used \eqref{notdeviantgjj1} to establish the third estimate above. 

Inserting \eqref{notdeviantgjk3} and \eqref{gjksumnotdeviant} into \eqref{linearvnotdeviantsmall3gij1}, and also using the facts that $C_2 > 1$ and $\mathbb{P} \big[ \overline{\Omega_i} \big] \le P_i$, yields 
\begin{flalign}
\label{linearvnotdeviantsmall3gij2}
\mathbb{P}_{\textbf{H} (\textbf{L})} \Bigg[ \bigg| \displaystyle\sum_{\substack{k \in [1, N] \\ k \ne i}} a_{ik} G_{kj}^{(i)} \bigg| \ge (\log N)^{\xi} \bigg( 6 C_2 U_2 N^{- \varepsilon / 10}   + \sqrt{ \displaystyle\frac{2 U_2}{N \eta} }  \bigg) \Bigg] \le \exp \big( - \widetilde{\nu} (\log N)^{\xi} \big) + P_i. 
\end{flalign}

\noindent Applying \eqref{notdeviantgij1} in \eqref{linearvnotdeviantsmall3gij2}, the fact that $h_{ij}$ has the same distribution as $a_{ij}$ (since $i$ is typical), and \eqref{gij} yields 
\begin{flalign}
\label{linearvnotdeviantsmall3gij3}
\mathbb{P}_{\textbf{H} (\textbf{L})} \Bigg[ |G_{ij}| \ge 2 U_2 (\log N)^{\xi} \bigg( 6 C_2 U_2 N^{- \varepsilon / 10}   + \sqrt{ \displaystyle\frac{2U_2}{N \eta} }  \bigg) \Bigg] \le \exp \big( - \widetilde{\nu} (\log N)^{\xi} \big) + P_i. 
\end{flalign}

\noindent Now, the existence of $C, \nu > 0$ satisfying \eqref{linearvnotdeviantsmall3gij} follows from applying \eqref{linearvnotdeviantsmall3gij3} and a union estimate over all $i \in \mathcal{T}$ and $j \in [1, N]$. 
\end{proof}

Above, we used the following lemma, which bounds the (deterministic) norm $\| \big( \Id - m_{\semicircle}^2 \widetilde{\textbf{S}} \big)^{-1} \|$; its proof is very similar to that of Lemma 2.15 in \cite{URM} and is thus omitted. 

\begin{lem}

\label{testimate}

Fix some $\kappa \in (0, 1)$, let $z \in \mathbb{H}$ satifsy $\Re z \in (\kappa - 2, 2 - \kappa)$, let $M \in \mathbb{Z}_{> 0}$, and recall the definition of $m_{\semicircle} = m_{\semicircle} (z)$ from \eqref{mquadratic}. Let $\textbf{\emph{T}} = \{ t_{ij} \}$ be some $M \times M$ symmetric matrix with positive entries satisfying $\sum_{j = 1}^N t_{ij} \le 1$; assume that there exist $\widetilde{c}, \widetilde{C} > 0$ such that $\widetilde{c} < M t_{ij} < \widetilde{C}$ for each $i, j$. Then, there exists some constant $C > 0$ (only dependent on $\kappa$, $\widetilde{c}$, and $\widetilde{C}$) such that 
\begin{flalign}
\label{mtestimate1}
\left\| \big( \Id - m_{\semicircle}^2 \textbf{\emph{T}} \big)^{-1}  \right\|_{\infty} \le C \log M. 
\end{flalign}

\end{lem}

\section{The Multiscale Argument for Deviant Indices} 

\label{LawSmallDeviant}

The following result is an analog of Proposition \ref{incrementnotdeviant} and Proposition \ref{omegaijnotdeviant} that now addresses the resolvent entries indexed by deviant integers. Observe here that the sum on the right side of \eqref{psi} has significantly more terms than does the sum on the right side of \eqref{probabilitynotdeviantevent}. Further observe that, as in Remark \ref{u1u2c}, the constants $C$ and $\nu$ below do not depend on the parameter $U$; however, the smallest value of $N$ for which \eqref{linearvdeviantsmall4} holds might depend on $U$. 

\begin{prop}

\label{incrementdeviant}

Fix $\kappa > 0$ and $U > 1$, and let $E \in [\kappa - 2, 2 - \kappa]$; also, let $\eta \in \mathbb{R}_{> 0}$ such that $N \eta > (\log N)^{8 \log \log N}$, and denote $z = E + \textbf{\emph{i}} \eta$. Fix a positive integer $N$, and let $\textbf{\emph{H}}$ be an $N \times N$ generalized Wigner matrix. Fix an admissible $N \times N$ $AB$ label $\textbf{\emph{L}} \in \mathcal{A}$. Recall the definitions of $v_i$ from \eqref{vj}; and of $\Omega_{C; \xi}^{(c)} (i, j)$ and $\Omega_C (i, j)$ from \eqref{omegaij}. Denote $r = \log \log N$ and 	
\begin{flalign}
\label{psi}
\begin{aligned}
\psi_i & = \displaystyle\sum_{\substack{\mathcal{S} \subset [1, N] \\ |\mathcal{S}| \le 2 r}} \displaystyle\sum_{\substack{j, k \in [1, N] \\ j, k \notin \mathcal{S}}} \mathbb{P}_{\textbf{\emph{H}}^{(\mathcal{S})} (\textbf{\emph{L}}^{(\mathcal{S})})}  \Bigg[ \Omega_U \bigg( j, k; \textbf{\emph{H}}^{(\mathcal{S})}; E + \textbf{\emph{i}} \eta \Big( 1 + \displaystyle\frac{1}{(\log N)^2} \Big)  \bigg) \Bigg] \\
& \qquad + \displaystyle\sum_{\substack{\mathcal{S} \subset [1, N] \\ |\mathcal{S}| \le 2 r}} \displaystyle\sum_{\substack{j \in [1, N] \\ j \notin \mathcal{S} \\ j \in \mathcal{T}_{\textbf{\emph{L}}}}} \mathbb{P}_{\textbf{\emph{H}}^{(\mathcal{S})} (\textbf{\emph{L}}^{(\mathcal{S})})}   \Bigg[ \Omega_{U; \xi}^{(1 / 20)} \bigg( j, j; \textbf{\emph{H}}^{(\mathcal{S})}; E + \textbf{\emph{i}} \eta \Big( 1 + \displaystyle\frac{1}{(\log N)^2} \Big)  \bigg) \Bigg]. 
\end{aligned}
\end{flalign}

\noindent Then, there exist constants $\nu, C > 0$ (only dependent on $\kappa$, $\varepsilon$, $C_1$, and $C_2$) such that
\begin{flalign}
\label{linearvdeviantsmall4}
\begin{aligned}
\mathbb{P}_{\textbf{\emph{H}} (\textbf{\emph{L}})} \Bigg[ \displaystyle\max_{1 \le i, j \le N}  |G_{ij}| & \ge C \Bigg] \le \exp \big( - \nu (\log N)^{\xi} \big) + 3 N \displaystyle\sum_{i = 1}^N \psi_i, 
\end{aligned}
\end{flalign}

\noindent for any $2 \le \xi \le \log \log N$ and sufficiently large $N$ (in comparison to $\kappa$, $\varepsilon$, $C_1$, $C_2$, and $U$). 

\end{prop}

\begin{proof}

In what follows, let us fix deviant indices $i, j \in [1, N]$, and let us restrict to the event 
\begin{flalign}
\label{deviantevent}
\begin{aligned}
\overline{\Psi_i} & = \bigcap_{\substack{\mathcal{S} \subset [1, N] \\ |\mathcal{S}| \le 2 r}} \bigcap_{\substack{j, k \in [1, N] \\ j, k \notin \mathcal{S}}} \overline{\Omega}_U \bigg( j, k; \textbf{H}^{(\mathcal{S})}; E + \textbf{i} \eta \Big( 1 + \displaystyle\frac{1}{(\log N)^2} \Big)  \bigg) \\
& \qquad \cap \bigcap_{\substack{\mathcal{S} \subset [1, N] \\ |\mathcal{S}| \le 2 r}} \bigcap_{\substack{j \in [1, N] \\ j \notin \mathcal{S} \\ j \in \mathcal{T}_{\textbf{L}} }} \overline{\Omega}_{U; \xi}^{(1 / 20)} \bigg( j, j; \textbf{H}^{(\mathcal{S})}; E + \textbf{i} \eta \Big( 1 + \displaystyle\frac{1}{(\log N)^2} \Big)  \bigg). 
\end{aligned}
\end{flalign}

\noindent By a union estimate, we find that $\mathbb{P}_{\textbf{H} (\textbf{L})} \big[ \Psi_i \big] \le \psi_i$, where $\Psi_i$ is the complement of $\overline{\Psi_i}$. 

We would like to estimate $\big| G_{ij} (E + \textbf{i} \eta) \big|$. By conjugating $\textbf{L}$ (and $\textbf{H}$) by a permutation matrix if necessary, we may assume that $i = 1$ and that $j \in \{ 1, 2 \}$ (depending on whether $i = j$ or $i \ne j$). Recalling the definitions from the beginning of Section \ref{IndicesGraph}, let $\mathscr{S} \subset [1, N]$ denote the union of $\{ i, j \}$ and the set of indices connected to either $i$ or $j$; we may assume that $\mathscr{S} = \{ 1, 2, \ldots , k \}$ for some $k < 2 r$, since $\textbf{L}$ is admissible. 

We must consider two possibilities, when $i = j$ or when $i \ne j$. First assume that $i$ and $j$ are distinct, so that $j = 2$. Denote $\textbf{H} = \left[\begin{smallmatrix} \textbf{A} & \textbf{B} \\ \textbf{C} & \textbf{D} \end{smallmatrix} \right]$, where $\textbf{A}$ is the top-left $k \times k$ submatrix of $\textbf{H}$. Then, by the Schur complement identity \eqref{blockinverse}, we deduce that the $(i, j)$ entry of $(\textbf{H} - z)^{-1}$ is equal to the $(i, j)$ entry of the $k \times k$ matrix $\big( \textbf{A} - z - \textbf{B} ( \textbf{D} - z )^{-1} \textbf{C} \big)^{-1}$. 

Denoting $\textbf{Y} = \textbf{B} ( \textbf{D} - z )^{-1} \textbf{C}$ and setting $\textbf{Y} = \{ y_{ij} \} = \{ y_{ij} (z) \}$, we have that 
\begin{flalign}
\label{yij}
y_{ij} & = \displaystyle\sum_{i', j' \in [1, N] \setminus \mathscr{S}} a_{ii'} G_{i' j'}^{(\mathscr{S})} a_{j' j}. 
\end{flalign}

\noindent Here, we are using the fact that the $(i, i')$ entry (and $(j, j')$ entry) of $\textbf{H} (\textbf{L})$ are equal to $a_{ii'}$ (and $a_{jj'}$), which follows since $i$ and $i'$ are unlinked (and as are $j$ and $j'$). 

Since we are restricting to the event $\overline{\Psi_i}$, we have that 
\begin{flalign}
\label{deviantgjk1}
\Bigg| G_{i' j'}^{(\mathscr{S})} \bigg( E + \textbf{i} \eta \Big( 1 + \displaystyle\frac{1}{(\log N)^2} \Big) \bigg) \Bigg| \le U,
\end{flalign}

\noindent for all $i', j'  \in [1, N] \setminus \mathscr{S}$. Thus, through a similar way as in the derivation of \eqref{notdeviantgjk3} from \eqref{notdeviantgjk1}, we find that 
\begin{flalign}
\label{deviantgjk3}
\big| G_{i' j'}^{(\mathscr{S})} ( E + \textbf{i} \eta ) \big| \le 2 U, 
\end{flalign}

\noindent for all $i', j' \in [1, N] \setminus \mathscr{S}$, assuming that $N > 10$. 

Using \eqref{deviantgjk3}, we apply \eqref{abij2} with $X_{i'} = a_{ii'}$; $Y_{j'} = a_{j' j}$; $N^{-\delta} = N^{- \varepsilon / 10} = q^{-1}$; and $R_{i' j'} = G_{i' j'}^{(\mathscr{S})}$. Due to the estimates $\mathbb{E} \big[ | a_{i' j'} | \big] \le 2 C_2 N^{-1 - \varepsilon / 10}$ and \eqref{qhijsmallmoments}, this yields 
\begin{flalign}
\label{deviantestimatey1} 
\begin{aligned}
\mathbb{P} \Bigg[ \bigg| \displaystyle\sum_{i', j' \in [1, N] \setminus \mathscr{S}} a_{ii'} G_{i' j'}^{(\mathscr{S})} a_{j' j} \bigg| \ge & (\log N)^{2 \xi} \bigg( (20 C_2^2 + 2) N^{-\varepsilon / 10} \displaystyle\max_{\substack{1 \le i', j' \le N \\ i', j' \notin \mathscr{S}}} \big| G_{i' j'}^{(\mathscr{S})} \big| \\
& \qquad + \Big( \displaystyle\frac{1}{N^2} \displaystyle\sum_{\substack{1 \le i', j' \le N \\ i', j' \notin \mathscr{S}}} \big| G_{i' j'}^{(\mathscr{S})} \big|^2 \Big)^{1 / 2} \bigg) \Bigg] \le \exp \big( - \widetilde{\nu} (\log N)^{\xi} \big).
\end{aligned}
\end{flalign}

\noindent where we have denoted $\widetilde{\nu}$ as the constant $\nu (C_1, 2C_2)$ from Lemma \ref{largeprobability2}. 

 Inserting \eqref{deviantgjk3} and \eqref{yij} into \eqref{deviantestimatey1} (and using the fact that $C_2 > 1$) yields 
\begin{flalign}
\label{deviantestimatey2}
\mathbb{P} \Bigg[ \big| y_{ij} \big| \ge (\log N)^{2 \xi} \bigg( 44 C_2^2 U N^{-\varepsilon / 10} + & \Big( \displaystyle\frac{1}{N^2} \displaystyle\sum_{\substack{1 \le i', j' \le N \\ i', j' \notin \mathscr{S}}} \big| G_{i' j'}^{(\mathscr{S})} \big|^2 \Big)^{1 / 2} \bigg) \Bigg]  \le \exp \big( - \widetilde{\nu} (\log N)^{\xi} \big) + \mathbb{P} \big[ \Psi_i \big].
\end{flalign}

\noindent Applying Ward's identity \eqref{sumgij} and \eqref{deviantgjk3}, we find that 
\begin{flalign}
\label{sumgijdeviant}
\left| \displaystyle\frac{1}{N^2} \displaystyle\sum_{\substack{1 \le i', j' \le N \\ i', j' \notin \mathscr{S}}} \big| G_{i' j'}^{(\mathscr{S})} \big|^2 \right| \ge \displaystyle\frac{1}{N^2 \eta} \displaystyle\sum_{i' \in [1, N] \setminus \mathscr{S}} \big| \Im G_{i' j'}^{(\mathscr{S})} \big| \le \displaystyle\frac{2 U}{N \eta}. 
\end{flalign}

Inserting \eqref{sumgijdeviant} into \eqref{deviantestimatey2} yields
\begin{flalign}
\label{deviantestimatey3}
\mathbb{P} \Bigg[ \big| y_{ij} \big| \ge (\log N)^{2 \xi} \bigg( 44 C_2^2 U N^{-\varepsilon / 10}  + \sqrt{\displaystyle\frac{2 U}{N \eta}} \bigg) \Bigg]  \le \exp \big( - \widetilde{\nu} (\log N)^{\xi} \big) + \mathbb{P} \big[ \Psi_i \big].
\end{flalign}

\noindent The bound \eqref{deviantestimatey3} shows that $|y_{ij}|$ is small if $i \ne j$. 

Now, assume that $i = j$. Then, using \eqref{yij} and the fact that $\sum_{j = 1}^N s_{ij} = 1 + t_i$, we find that 
\begin{flalign}
\label{yii}
\begin{aligned}
y_{ii} & = \displaystyle\sum_{i', j' \notin \mathscr{S}} a_{ii'} G_{i' j'}^{(\mathscr{S})} a_{j' i} \\
& = m_{\semicircle} + \displaystyle\sum_{\substack{i', j' \notin \mathscr{S} \\ i' \ne j'}} a_{ii'} G_{i' j'}^{(\mathscr{S})} a_{j' i} + \displaystyle\sum_{j \notin S} \big( |a_{ij}|^2 - s_{ij} \big) G_{jj}^{(\mathscr{S})} + \displaystyle\sum_{j \notin \mathscr{S}} s_{ij} \big( G_{jj}^{(\mathscr{S})} - m_{\semicircle} \big) - m_{\semicircle} \left( \displaystyle\sum_{j \in \mathscr{S}} s_{ij} - t_i \right). 
\end{aligned}
\end{flalign}

\noindent Now, since $\big| \mathscr{S} \big| \le 2 r$, $s_{ij} \le C_1 N^{-1}$, and $|t_i| < C_1 N^{-\varepsilon}$, we have that 
\begin{flalign}
\label{sijs}
 \left| m_{\semicircle} \displaystyle\sum_{j \in \mathscr{S}} s_{ij} \right| + \big| m_{\semicircle} t_i \big|  \le 3 r C_1 |m_{\semicircle}| (N^{-1} + N^{-\varepsilon}).
\end{flalign}

\noindent Furthermore, since we are restricting to the event $\overline{\Psi_i}$, we have that 
\begin{flalign}
\label{deviantgjj2}
\Bigg| G_{j' j'}^{(\mathscr{S})} \bigg( E + \textbf{i} \eta \Big( 1 + \displaystyle\frac{1}{(\log N)^2} \Big) \bigg) - m_{\semicircle} \Bigg| \le U (\log N)^{3 \xi} \left( N^{- \varepsilon / 20} + \displaystyle\frac{1}{\sqrt{N \eta}} \right), 
\end{flalign}

\noindent if $j' \in \mathcal{T}_{\textbf{L}}$ is typical. Following the derivation of \eqref{giismall1notdeviant1} from \eqref{giigjjsmall1notdeviant1}, we obtain that 
\begin{flalign}
\label{deviantgjj3}
\big| G_{j' j'}^{(\mathscr{S})} ( E + \textbf{i} \eta ) - m_{\semicircle} \big| \le U (\log N)^{3 \xi} \left( N^{- \varepsilon / 20} + \displaystyle\frac{1}{\sqrt{N \eta}} \right) + \displaystyle\frac{2 U}{(\log N)^2}, 
\end{flalign}

\noindent if $j' \in \mathcal{T}_{\textbf{L}}$ is typical. 

Furthermore, if $j' \in \mathcal{D}_{\textbf{L}}$ is deviant, then we have that $\big| G_{j' j'}^{(\mathscr{S})} - m_{\semicircle} \big| \le 2 U + |m_{\semicircle}|$, in view of \eqref{deviantgjk3}. Thus, since $s_{ij} \le C_1 N^{-1}$ and $\mathcal{D}_{\textbf{L}} \le N^{1 - \varepsilon / 20}$ (since $\textbf{L}$ is admissible), it follows that 
\begin{flalign}
\label{sijgjjm}
\begin{aligned}
\left| \displaystyle\sum_{j \notin \mathscr{S}} s_{ij} \big( G_{jj}^{(\mathscr{S})} - m_{\semicircle} \big) \right| & \le \left| \displaystyle\sum_{\substack{j \notin \mathscr{S} \\ j \in \mathcal{T}_{\textbf{L}}}} s_{ij} \big( G_{jj}^{(\mathscr{S})} - m_{\semicircle} \big) \right| + \left| \displaystyle\sum_{\substack{j \notin \mathscr{S} \\ j \in \mathcal{D}_{\textbf{L}}}} s_{ij} \big( G_{jj}^{(\mathscr{S})} - m_{\semicircle} \big) \right| \\
& \le C_1 U (\log N)^{\xi} \left( N^{-c \varepsilon} + \displaystyle\frac{1}{\sqrt{N \eta}} \right) + \displaystyle\frac{2 C_1 U}{(\log N)^2} + C_1 \big( 2 U + |m_{\semicircle}| \big) N^{-\varepsilon / 20}. 
\end{aligned}
\end{flalign} 

\noindent To estimate the remaining terms in \eqref{yii}, we apply Lemma \ref{largeprobability2}. Specifically, applying \eqref{bij2} with $X_{i'} = a_{ii'}$; $B_{i' j'} = G_{i' j'}^{(\mathscr{S})}$; and $N^{-\delta} = N^{-\varepsilon / 10} = q^{-1}$, we obtain that 
\begin{flalign}
\label{deviantestimateyii1}
\begin{aligned}
\mathbb{P} \Bigg[ \bigg| \displaystyle\sum_{\substack{i', j' \notin \mathscr{S} \\ i' \ne j'}} a_{ii'} G_{i' j'}^{(\mathscr{S})} a_{j' i} \bigg| \ge & (\log N)^{2 \xi} \bigg( (20 C_1^2 + 1) N^{-\varepsilon / 10} \displaystyle\max_{\substack{1 \le i', j' \le N \\ i', j' \notin \mathscr{S}}} \big| G_{i' j'}^{(\mathscr{S})} \big| \\
& \qquad + \Big( \displaystyle\frac{1}{N^2} \displaystyle\sum_{\substack{1 \le i', j' \le N \\ i', j' \notin \mathscr{S}}} \big| G_{i' j'}^{(\mathscr{S})} \big|^2 \Big)^{1 / 2} \bigg) \Bigg] \le \exp \big( - \widetilde{\nu} (\log N)^{\xi} \big).
\end{aligned}
\end{flalign}

\noindent Similar to \eqref{deviantestimatey3}, it quickly follows that 
\begin{flalign}
\label{deviantestimateyii3}
\mathbb{P} \Bigg[ \bigg| \displaystyle\sum_{\substack{i', j' \notin \mathscr{S} \\ i' \ne j'}} a_{ii'} G_{i' j'}^{(\mathscr{S})} a_{j' i} \bigg| \ge (\log N)^{2 \xi} \bigg( 21 C_2^2 N^{-\varepsilon / 10} U  + \sqrt{\displaystyle\frac{2 U}{N \eta}} \bigg) \Bigg]  \le \exp \big( - \widetilde{\nu} (\log N)^{\xi} \big) + \mathbb{P} \big[ \Psi_i \big].
\end{flalign}

\noindent Applying \eqref{bi2} with $X_j = a_{ij}$; $s_j = s_{ij}$; $B_j = G_{jj}^{(\mathscr{S})}$; and $N^{-\delta} = N^{-\varepsilon / 10} = q^{-1}$, we obtain that 
\begin{flalign}
\label{deviantestimateyii4}
\mathbb{P} \Bigg[ \bigg| \displaystyle\sum_{\substack{1 \le j \le N \\ j \notin \mathscr{S}}} \big( | a_{ij} |^2 - s_{ij} \big) G_{jj}^{(\mathscr{S})}  \bigg| \ge & (\log N)^{\xi} (20 C_2^2 + 1) N^{-\varepsilon / 10} \displaystyle\max_{\substack{1 \le j \le N \\ j \notin \mathscr{S}}} \big| G_{jj}^{(\mathscr{S})} \big| \Bigg] \le \exp \big( - \widetilde{\nu} (\log N)^{\xi} \big).
\end{flalign}

\noindent Inserting \eqref{deviantgjk3} into \eqref{deviantestimateyii4} yields 
\begin{flalign}
\label{deviantestimateyii5}
\mathbb{P} \Bigg[ \bigg| \displaystyle\sum_{\substack{1 \le j \le N \\ j \notin \mathscr{S}}} \big( | a_{ij} |^2 - s_{ij} \big) G_{jj}^{(\mathscr{S})}  \bigg| \ge 21 C_2^2 (\log N)^{\xi} U N^{-\varepsilon / 10} \Bigg] \le \exp \big( - \widetilde{\nu} (\log N)^{\xi} \big) + \mathbb{P} \big[ \Psi_i \big]. 
\end{flalign}

\noindent Combining \eqref{yii}, \eqref{sijs}, \eqref{sijgjjm}, \eqref{deviantestimateyii3}, \eqref{deviantestimateyii5}, and a union estimate yields 
\begin{flalign}
\label{yijdeviantsmall}
\mathbb{P} \left[ |y_{ii} - m_{\semicircle}| > \displaystyle\frac{1}{\log N} \right] \le 2 \exp \big( - \widetilde{\nu} (\log N)^{\xi} \big) + 3 \mathbb{P} \big[ \Psi_i \big], 
\end{flalign}

\noindent assuming that $N$ is sufficiently large in comparison to $C_1$, $C_2$, $U$, and $\varepsilon^{-1}$. 

Applying \eqref{deviantestimatey3}, \eqref{yijdeviantsmall}, and a union estimate yields 
\begin{flalign}
\label{deviantestimateyentries} 
\mathbb{P} \left[ \displaystyle\max_{1 \le i, j \le k}  \big| y_{ij} - \textbf{1}_{i = j} m_{\semicircle} \big| \le \displaystyle\frac{1}{\log N} \right] \le 2 N^2 \exp \big( - \widetilde{\nu} ( \log N )^{\xi} \big) + 3 N \displaystyle\sum_{i = 1}^N \mathbb{P} \big[ \Psi_i \big], 
\end{flalign}

\noindent assuming that $N$ is sufficiently large in comparison to $C_1$, $C_2$, $U$, and $\varepsilon^{-1}$.  

From \eqref{deviantestimateyentries}, it follows that 
\begin{flalign}
\label{deviantestimatey} 
\mathbb{P} \left[ \| \textbf{Y} - m_{\semicircle} \Id \|_{\infty}  \le \displaystyle\frac{2 \log \log N }{\log N} \right] \le 2 N^2 \exp \big( - \widetilde{\nu} ( \log N )^{\xi} \big) + 3 N \displaystyle\sum_{i = 1}^N \mathbb{P} \big[ \Psi_i \big], 
\end{flalign}

\noindent where we have used the fact that $k \le 2 r = 2 \log \log N$. 

Now, since $\Im z \in [\kappa - 2, 2 - \kappa]$, there exists some constant $c > 0$ (only dependent on $\kappa$) such that $\big| \Im (m_{\semicircle} + z) \big|> c$. Therefore, since $\textbf{A}$ is Hermitian, there exists some $\widetilde{C} > 0$ (only dependent on $\kappa$) such that $\big\| \big( \textbf{A} - (m_{\semicircle} + z) \Id \big)^{-1} \big\| < \widetilde{C}$. This implies by \eqref{deviantestimatey} that 
\begin{flalign*}
\mathbb{P} \bigg[ \Big\| \big( \textbf{A} - z \Id - \textbf{Y} \big)^{-1} \Big\| \ge 2 \widetilde{C} \bigg] \le 2 N^2 \exp \big( - \widetilde{\nu} ( \log N )^{\xi} \big) + 3 N \displaystyle\sum_{i = 1}^N \mathbb{P} \big[ \Psi_i \big], 
\end{flalign*} 

\noindent for sufficiently large $N$ (in comparison to $C_1$, $C_2$, $U$, $\varepsilon^{-1}$, and $\kappa$), from which it follows that 
\begin{flalign}
\label{gijdeviantestimate}
\mathbb{P} \Big[ \big| G_{ij} \big| \ge 2 \widetilde{C} \Big] \le 2 N^2 \exp \big( - \widetilde{\nu} ( \log N )^{\xi} \big) + 3 N \displaystyle\sum_{i = 1}^N \mathbb{P} \big[ \Psi_i \big], 
\end{flalign}

\noindent where we used the facts that $G_{ij}$ is the $(i, j)$-entry of $\big( \textbf{A} - z \Id - \textbf{Y} \big)^{-1}$ and that both $\textbf{A}$ and $\textbf{Y}$ are symmetric. Now the existence of constants $C, \nu > 0$ satisfying \eqref{linearvdeviantsmall4} follows from \eqref{gijdeviantestimate}. 
\end{proof}

\section{Proof of Theorem \ref{localmomentsdeltalabel}}

\label{ProofLocalLaw}

The purpose of this section is to establish Theorem \ref{localmomentsdeltalabel}, which will essentially follow from repeated application of Proposition \ref{incrementnotdeviant}, Proposition \ref{omegaijnotdeviant}, and Proposition \ref{incrementdeviant}. However, before doing so, we must first choose the constants from those propositions in such a way that those results can be applied simultaneously; to that end, we introduce the following notation. 

\begin{itemize}

\item{Denote by $\gamma_1$ the constant $C$ from Proposition \ref{locallargeeta}; it only depends on $C_1$ and $C_2$.} 

\item{Denote by $\mu_1$ the constant $\nu$ from Proposition \ref{locallargeeta}; it also only depends on $C_1$ and $C_2$.}

\item{Recall that Proposition \ref{locallargeeta} holds when $N$ is sufficiently large in comparison to $\varepsilon$; let $\Phi_1$ (dependent on only $\varepsilon$) be such that it holds whenever $N > \Phi_1$.}

\item{Denote by $\gamma_2$ the constant $C$ from Proposition \ref{incrementdeviant}; it only depends on $\varepsilon$, $\kappa$, $C_1$, and $C_2$.} 

\item{Denote by $\mu_2$ the constant $\nu$ from Proposition \ref{incrementdeviant}; it only depends on $\varepsilon$, $\kappa$, $C_1$, and $C_2$.}

\item{Apply Proposition \ref{incrementnotdeviant} and Proposition \ref{omegaijnotdeviant} with $U_2 = \max \{ 1, \gamma_2 \}$ and $U_1 > 1$ arbitrary. Denote by $\gamma_3$ the resulting constant $C$ from Proposition \ref{incrementnotdeviant}, and denote by $\gamma_4$ the resulting constant from Proposition \ref{omegaijnotdeviant}. Recall that $\gamma_3$ and $\gamma_4$ only depend on $\varepsilon$, $\kappa$, $C_1$, $C_2$, and $U_2$; in particular, since $U_2$ only depends on the first four parameters, $\gamma_3$ and $\gamma_4$ only depend on $\varepsilon$, $\kappa$, $C_1$, and $C_2$.}

\item{Denote by $\mu_3$ and $\mu_4$ denote the constants $\nu$ from Proposition \ref{incrementnotdeviant} and Proposition \ref{omegaijnotdeviant}, respectively; they only depend on $\varepsilon$, $\kappa$, $C_1$, and $C_2$.}

\item{Let $\gamma = \max \{ 1, \gamma_1, \gamma_3, \gamma_4 \}$ and $\widetilde{\gamma} = \max \{ 1, \gamma_2\}$; furthermore, let $\widehat{\gamma} = \max \{ \gamma, \widetilde{\gamma} \}$. These parameters only depend on $\varepsilon$, $\kappa$, $C_1$, and $C_2$.}

\item{Let $\nu = \min \{ \mu_1, \mu_2, \mu_3, \mu_4 \}$; it only depends on $\varepsilon$, $\kappa$, $C_1$, and $C_2$. }

\item{Apply Proposition \ref{incrementnotdeviant} and Proposition \ref{omegaijnotdeviant} with $U_1 = \gamma$ and $U_2 = \widetilde{\gamma}$. Recall that both of these propositions hold when $N$ is sufficiently large in comparison to $\varepsilon$, $\kappa$, $C_1$, $C_2$, $U_1 = \gamma$, and $U_2 = \widetilde{\gamma}$. Let $\Phi_2$ and $\Phi_3$ (only dependent on $\varepsilon$, $\kappa$, $C_1$, and $C_2$, since $\gamma$ and $\widetilde{\gamma}$ are determined from those four parameters) be such that Proposition \ref{incrementnotdeviant} and Proposition \ref{omegaijnotdeviant} hold whenever $N > \Phi_2$ and $N > \Phi_3$, respectively. }

\item{Apply Proposition \ref{incrementdeviant} with $U = \widehat{\gamma}$. Recall that this proposition holds when $N$ is sufficiently large in comparison to $\varepsilon$, $\kappa$, $C_1$, $C_2$, and $\widehat{\gamma}$. Let $\Phi_4$ (dependent only on $\varepsilon$, $\kappa$, $C_1$, and $C_2$) be such this propostion holds for all $N > \Phi_4$.} 

\item{Denote $\Phi = \max \{ \Phi_1, \Phi_2, \Phi_3, \Phi_4 \}$; it only depends on $\varepsilon$, $\kappa$, $C_1$, and $C_2$.} 

\item{Set $r = \log \log N$.}

\end{itemize}

Now, select $N$ to be sufficiently large (in comparison to $\varepsilon$, $\kappa$, $C_1$, and $C_2$) such that 
\begin{flalign}
\label{largen}
\left( 1 + \displaystyle\frac{1}{(\log N)^2} \right)^{(\log N)^4} > \widehat{\gamma} N; \qquad N - 2 r (\log N)^4 > \displaystyle\frac{N}{2} > \Phi.  
\end{flalign}

Let us apply Proposition \ref{incrementnotdeviant} and Proposition \ref{omegaijnotdeviant} with $U_1 = \gamma$ and $U_2 = \widetilde{\gamma}$, and then apply Proposition \ref{incrementdeviant} with $U = \widehat{\gamma}$. From a union estimate, we obtain that 
\begin{flalign}
\label{omega1}
& \mathbb{P}_{\textbf{H} (\textbf{L})} \Big[ \bigcap_{i \in \mathcal{T}_{\textbf{L}}} \bigcap_{j = 1}^N \Omega_{\gamma; \xi}^{(1 / 20)} \big( i, j; \textbf{H}, E + \textbf{i} \eta \big) \Big] \le 2 \exp \big( - \nu (\log N)^{\xi} \big) + (N + 1) \displaystyle\sum_{i = 1}^N P_i, \\
\label{omega2}
& \mathbb{P}_{\textbf{H} (\textbf{L})} \Big[ \bigcap_{1 \le i, j \le N} \Omega_{\widetilde{\gamma}} \big( i, j; \textbf{H}, E + \textbf{i} \eta \big) \Big] \le \exp \big( - \nu (\log N)^{\xi} \big) + 3 N \displaystyle\sum_{i = 1}^N \psi_i,
\end{flalign}

\noindent where $P_i$ and $\psi_i$ were defined in \eqref{probabilitynotdeviantevent} and \eqref{psi}, respectively. 

Now, observe that 
\begin{flalign}
\label{psi1}
\begin{aligned}
P_i, \psi_i & \le N^2 \binom{N}{2 r} \displaystyle\max_{\substack{\mathcal{S} \subset [1, N] \\ |\mathcal{S}| \le 2 r}} \displaystyle\max_{\substack{j, k \in [1, N] \\ j, k \notin \mathcal{S}}} \mathbb{P}_{\textbf{H}^{(\mathcal{S})} (\textbf{L}^{(\mathcal{S})})} \Bigg[ \Omega_{\widetilde{\gamma}} \bigg( j, k; \textbf{H}^{(\mathcal{S})}; E + \textbf{i} \eta \Big( 1 + \displaystyle\frac{1}{(\log N)^2} \Big)  \bigg) \Bigg] \\
& \qquad + N \binom{N}{2 r} \displaystyle\max_{\substack{\mathcal{S} \subset [1, N] \\ |\mathcal{S}| \le 2 r}} \displaystyle\max_{\substack{j \in [1, N] \\ j \notin \mathcal{S} \\ j \in \mathcal{T}_{\textbf{L}}}} \mathbb{P}_{\textbf{H}^{(\mathcal{S})} (\textbf{L}^{(\mathcal{S})})}   \Bigg[ \Omega_{\gamma; \xi}^{(1 / 20)} \bigg( j, j; \textbf{H}^{(\mathcal{S})}; E + \textbf{i} \eta \Big( 1 + \displaystyle\frac{1}{(\log N)^2} \Big)  \bigg) \Bigg]. 
\end{aligned}
\end{flalign}

\noindent Inserting \eqref{psi1} into \eqref{omega1} and \eqref{omega2} yields 
\begin{flalign}
\label{omega3}
\begin{aligned}
& \mathbb{P}_{\textbf{H} (\textbf{L})} \Big[ \bigcap_{i \in \mathcal{T}_{\textbf{L}}} \bigcap_{j = 1}^N \Omega_{\gamma; \xi}^{(1 / 20)} \big( i, j; \textbf{H}, E + \textbf{i} \eta \big) \Big] - 2 \exp \big( - \nu (\log N)^{\xi} \big)\\
& \qquad \le  N^{6r} \displaystyle\max_{\substack{\mathcal{S} \subset [1, N] \\ |\mathcal{S}| \le 2 r}} \displaystyle\max_{\substack{j, k \in [1, N] \\ j, k \notin \mathcal{S}}} \mathbb{P}_{\textbf{H}^{(\mathcal{S})} (\textbf{L}^{(\mathcal{S})})} \Bigg[ \Omega_{\widetilde{\gamma}} \bigg( j, k; \textbf{H}^{(\mathcal{S})}; E + \textbf{i} \eta \Big( 1 + \displaystyle\frac{1}{(\log N)^2} \Big)  \bigg) \Bigg] \\
& \qquad \qquad + N^{6r} \displaystyle\max_{\substack{\mathcal{S} \subset [1, N] \\ |\mathcal{S}| \le 2 r}} \displaystyle\max_{\substack{j \in [1, N] \\ j \notin \mathcal{S} \\ j \in \mathcal{T}_{\textbf{L}}}} \mathbb{P}_{\textbf{H}^{(\mathcal{S})} (\textbf{L}^{(\mathcal{S})})}   \Bigg[ \Omega_{\gamma; \xi}^{(1 / 20)} \bigg( j, j; \textbf{H}^{(\mathcal{S})}; E + \textbf{i} \eta \Big( 1 + \displaystyle\frac{1}{(\log N)^2} \Big)  \bigg) \Bigg], 
\end{aligned}
\end{flalign}

\noindent and 
\begin{flalign} 
\label{omega4}
\begin{aligned}
& \mathbb{P}_{\textbf{H} (\textbf{L})} \Big[ \bigcap_{1 \le i, j \le N} \Omega_{\widetilde{\gamma}} \big( i, j; \textbf{H}, E + \textbf{i} \eta \big) \Big] - 2 \exp \big( - \nu (\log N)^{\xi} \big) \\
 & \qquad \le  N^{6 r} \displaystyle\max_{\substack{\mathcal{S} \subset [1, N] \\ |\mathcal{S}| \le 2 r}} \displaystyle\max_{\substack{j, k \in [1, N] \\ j, k \notin \mathcal{S}}} \mathbb{P}_{\textbf{H}^{(\mathcal{S})} (\textbf{L}^{(\mathcal{S})})} \Bigg[ \Omega_{\widetilde{\gamma}} \bigg( j, k; \textbf{H}^{(\mathcal{S})}; E + \textbf{i} \eta \Big( 1 + \displaystyle\frac{1}{(\log N)^2} \Big)  \bigg) \Bigg] \\
& \qquad \qquad + N^{6 r} \displaystyle\max_{\substack{\mathcal{S} \subset [1, N] \\ |\mathcal{S}| \le 2 r}} \displaystyle\max_{\substack{j \in [1, N] \\ j \notin \mathcal{S} \\ j \in \mathcal{T}_{\textbf{L}}}} \mathbb{P}_{\textbf{H}^{(\mathcal{S})} (\textbf{L}^{(\mathcal{S})})}   \Bigg[ \Omega_{\gamma; \xi}^{(1 / 20)} \bigg( j, j; \textbf{H}^{(\mathcal{S})}; E + \textbf{i} \eta \Big( 1 + \displaystyle\frac{1}{(\log N)^2} \Big)  \bigg) \Bigg].  
\end{aligned}
\end{flalign}

Now, fix an integer $k \ge 0$. Let us apply \eqref{omega1} and \eqref{omega2} again, but with $\textbf{H}$ replaced by $\textbf{H}^{(\mathcal{S})}$, and apply a union estimate over all $\mathcal{S} \subset [1, N]$ satisfying $\big| \mathcal{S} \big| < 2 k r$. We obtain that 
\begin{flalign}
\label{omega5}
\begin{aligned}
& \displaystyle\max_{\substack{\mathcal{S} \subset [1, N] \\ |\mathcal{S}| \le 2 k r}} \displaystyle\max_{\substack{j, k \in [1, N] \\ j, k \notin \mathcal{S}}} \mathbb{P}_{\textbf{H}^{(\mathcal{S})} (\textbf{L}^{(\mathcal{S})})} \Bigg[ \Omega_{\gamma; \xi}^{(1 / 20)} \bigg( j, k; \textbf{H}^{(\mathcal{S})}; E + \textbf{i} \eta \Big( 1 + \displaystyle\frac{1}{(\log N)^2} \Big)^k  \bigg) \Bigg] \\
& \quad \le  N^{12 (k + 1) r} \displaystyle\max_{\substack{\mathcal{S} \subset [1, N] \\ |\mathcal{S}| \le 2 (k + 1) r}} \displaystyle\max_{\substack{j, k \in [1, N] \\ j, k \notin \mathcal{S}}} \mathbb{P}_{\textbf{H}^{(\mathcal{S})} (\textbf{L}^{(\mathcal{S})})} \Bigg[ \Omega_{\widetilde{\gamma}} \bigg( j, k; \textbf{H}^{(\mathcal{S})}; E + \textbf{i} \eta \Big( 1 + \displaystyle\frac{1}{(\log N)^2} \Big)^{k + 1}  \bigg) \Bigg] \\
& \quad \quad + N^{12 (k + 1) r} \displaystyle\max_{\substack{\mathcal{S} \subset [1, N] \\ |\mathcal{S}| \le 2 (k + 1) r}} \displaystyle\max_{\substack{j \in [1, N] \\ j \notin \mathcal{S} \\ j \in \mathcal{T}_{\textbf{L}}}} \mathbb{P}_{\textbf{H}^{(\mathcal{S})} (\textbf{L}^{(\mathcal{S})})}   \Bigg[ \Omega_{\gamma; \xi}^{(1 / 20)} \bigg( j, j; \textbf{H}^{(\mathcal{S})}; E + \textbf{i} \eta \Big( 1 + \displaystyle\frac{1}{(\log N)^2} \Big)^{k + 1}  \bigg) \Bigg] \\
& \quad \quad + 3 N^{12 k r} \exp \big( - \nu (\log N)^{\xi} \big), 
\end{aligned} 
\end{flalign}

\noindent and 
\begin{flalign} 
\label{omega6}
\begin{aligned}
& \displaystyle\max_{\substack{\mathcal{S} \subset [1, N] \\ |\mathcal{S}| \le 2 k r}} \displaystyle\max_{\substack{j, k \in [1, N] \\ j, k \notin \mathcal{S}}} \mathbb{P}_{\textbf{H}^{(\mathcal{S})} (\textbf{L}^{(\mathcal{S})})} \Bigg[ \Omega_{\widetilde{\gamma}} \bigg( j, k; \textbf{H}^{(\mathcal{S})}; E + \textbf{i} \eta \Big( 1 + \displaystyle\frac{1}{(\log N)^2} \Big)^k  \bigg) \Bigg] \\
 & \quad \le N^{12 (k + 1) r} \displaystyle\max_{\substack{\mathcal{S} \subset [1, N] \\ |\mathcal{S}| \le 2 (k + 1) r}} \displaystyle\max_{\substack{j, k \in [1, N] \\ j, k \notin \mathcal{S}}} \mathbb{P}_{\textbf{H}^{(\mathcal{S})} (\textbf{L}^{(\mathcal{S})})} \Bigg[ \Omega_{\widetilde{\gamma}} \bigg( j, k; \textbf{H}^{(\mathcal{S})}; E + \textbf{i} \eta \Big( 1 + \displaystyle\frac{1}{(\log N)^2} \Big)^{k + 1}  \bigg) \Bigg] \\
& \quad \quad + N^{12 (k + 1) r} \displaystyle\max_{\substack{\mathcal{S} \subset [1, N] \\ |\mathcal{S}| \le 2 (k + 1) r}} \displaystyle\max_{\substack{j \in [1, N] \\ j \notin \mathcal{S} \\ j \in \mathcal{T}_{\textbf{L}}}} \mathbb{P}_{\textbf{H}^{(\mathcal{S})} (\textbf{L}^{(\mathcal{S})})}   \Bigg[ \Omega_{\gamma; \xi}^{(1 / 20)} \bigg( j, j; \textbf{H}^{(\mathcal{S})}; E + \textbf{i} \eta \Big( 1 + \displaystyle\frac{1}{(\log N)^2} \Big)^{k + 1}  \bigg) \Bigg] \\
& \quad \quad + 3 N^{12 k r} \exp \big( - \nu (\log N)^{\xi} \big), 
\end{aligned}
\end{flalign}

\noindent Now let $\zeta = \lceil (\log N)^4 \rceil$, and repeatedly apply \eqref{omega5} and \eqref{omega6} for all $k \in [0, \zeta - 1]$. We obtain that 
\begin{flalign} 
\label{omega7}
\begin{aligned}
& \mathbb{P}_{\textbf{H} (\textbf{L})} \Big[ \bigcap_{i \in \mathcal{T}_{\textbf{L}}} \bigcap_{j = 1}^N \Omega_{\gamma; \xi}^{(1 / 20)} \big( i, j; \textbf{H}, E + \textbf{i} \eta \big) \Big] + \mathbb{P}_{\textbf{H} (\textbf{L})} \Big[ \bigcap_{1 \le i, j \le N} \Omega_{\widetilde{\gamma}} \big( i, j; \textbf{H}, E + \textbf{i} \eta \big) \Big] \\
 & \quad \le 2^{\zeta} N^{12 \zeta^2 r } \displaystyle\max_{\substack{\mathcal{S} \subset [1, N] \\ |\mathcal{S}| \le 2 \zeta r}} \displaystyle\max_{\substack{j, k \in [1, N] \\ j, k \notin \mathcal{S}}} \mathbb{P}_{\textbf{H}^{(\mathcal{S})} (\textbf{L}^{(\mathcal{S})})} \Bigg[ \Omega_{\widetilde{\gamma}} \bigg( j, k; \textbf{H}^{(\mathcal{S})}; E + \textbf{i} \eta \Big( 1 + \displaystyle\frac{1}{(\log N)^2} \Big)^{\zeta}  \bigg) \Bigg] \\
& \quad \quad + 2^{\zeta} N^{12 \zeta^2 r} \displaystyle\max_{\substack{\mathcal{S} \subset [1, N] \\ |\mathcal{S}| \le 2 \zeta r}} \displaystyle\max_{\substack{j \in [1, N] \\ j \notin \mathcal{S} \\ j \in \mathcal{T}_{\textbf{L}}}} \mathbb{P}_{\textbf{H}^{(\mathcal{S})} (\textbf{L}^{(\mathcal{S})})}   \Bigg[ \Omega_{\gamma; \xi}^{(1 / 20)} \bigg( j, j; \textbf{H}^{(\mathcal{S})}; E + \textbf{i} \eta \Big( 1 + \displaystyle\frac{1}{(\log N)^2} \Big)^{\zeta}  \bigg) \Bigg] \\
& \quad \quad + 6^{\zeta} N^{12 \zeta^2 r} \exp \big( - \nu (\log (N / 2) )^{\xi} \big), 
\end{aligned}
\end{flalign} 

\noindent for all $2 \le \xi \le \log \log N$. 

Now, observe that $\eta \big( 1 + (\log N)^{-2} \big)^{\zeta} > \gamma$, due to the first estimate in \eqref{largen} and the fact that $\eta > N^{-1}$. This, and the second estimate of \eqref{largen} shows that we can apply Proposition \ref{locallargeeta} to obtain 
\begin{flalign}
\label{omega8}
\begin{aligned}
\displaystyle\max_{\substack{\mathcal{S} \subset [1, N] \\ |\mathcal{S}| \le 2 \zeta r}} \displaystyle\max_{\substack{j, k \in [1, N] \\ j, k \notin \mathcal{S}}} \mathbb{P}_{\textbf{H}^{(\mathcal{S})} (\textbf{L}^{(\mathcal{S})})} \Bigg[ \Omega_{\widetilde{\gamma}} \bigg( j, k; \textbf{H}^{(\mathcal{S})}; E & + \textbf{i} \eta \Big( 1 + \displaystyle\frac{1}{(\log (N / 2))^2} \Big)^{\zeta}  \bigg) \Bigg] \\
& \le N^{2 \zeta r + 2} \exp \big( - \nu (\log (N / 2) )^{\xi} \big)  
\end{aligned}
\end{flalign}

\noindent and 
\begin{flalign}
\label{omega9} 
\begin{aligned}
\displaystyle\max_{\substack{\mathcal{S} \subset [1, N] \\ |\mathcal{S}| \le 2 \zeta r}} \displaystyle\max_{\substack{j \in [1, N] \\ j \notin \mathcal{S} \\ j \in \mathcal{T}_{\textbf{L}}}} \mathbb{P}_{\textbf{H}^{(\mathcal{S})} (\textbf{L}^{(\mathcal{S})})}   \Bigg[ \Omega_{\gamma; \xi}^{(1 / 20)} \bigg( j, j; \textbf{H}^{(\mathcal{S})}; E & + \textbf{i} \eta \Big( 1 + \displaystyle\frac{1}{(\log N)^2} \Big)^{\zeta}  \bigg) \Bigg] \\
& \le N^{2 \zeta r + 2} \exp \big( - \nu (\log (N / 2) )^{\xi} \big), 
\end{aligned}
\end{flalign} 

\noindent for all $2 \le \xi \le \log \log N$; here, we have applied a union estimate to bound the maximum.  

Now \eqref{typicalomega} (with $C = \gamma$, $c = 1 / 20$, and $\xi = 10$, for $N$ sufficiently large) and \eqref{deviantomega} (with $C = \widetilde{\gamma}$ and $\xi = 10$, for $N$ sufficiently large) follow from inserting \eqref{omega8} and \eqref{omega9} into \eqref{omega7} and taking $\xi = 10$.

\section{Universality of Local Statistics}

\label{LocalH}

The goal of this section is to use the local semicircle law Theorem \ref{localmoments} to establish the bulk universality results Theorem \ref{gapsfunctions} and Theorem \ref{bulkfunctions} for local eigenvalue statistics of (heavy-tailed) generalized Wigner matrices. This will comprise the latter two parts of the three-step strategy. 

Recall that the first of those is to apply a Dyson Brownian motion (or, in our case, a matrix Ornstein-Uhlenbeck process) to the original random matrix $\textbf{H}$, thereby forming matrix $\textbf{H}_t$, and then show that the local statistics of $\textbf{H}_t$ converge to those of $\textbf{GOE}_N$. The second is to show that the local statistics of $\textbf{H}_t$ are very similar to those of $\textbf{H}$. 

These two steps have been implemented many times in the random matrix literature \cite{BESRRG, FEUGM, URMCE, SSGESEE, UM, URMLRF, LRFULSRM, URMFTD, BUGM, SSSG, BUSM, CLSM, FEUM, ULES}. In particular, our exposition will closely follow that of \cite{URMCE, SSSG, BUSM}. Therefore, we will only outline the proofs, explaining the differences where they arise.

\subsection{Bulk and Gap Universality Under Gaussian Perturbations}

\label{PerturbedUniversality} 

We begin with the first step, that is, we apply an Ornstein-Uhlenbeck process to $\textbf{H}$ and show that the local statistics of the resulting matrix converge to those of the GOE in the large $N$ limit; see Theorem \ref{universalityperturbation3}. This will largely use the results of the recent works of Landon-Yau \cite{CLSM} and Landon-Sosoe-Yau \cite{FEUM}, which establish very quick convergence of bulk local statistics under Dyson Brownian motion. 

The discussion in this section will be very similar to that in Section 3 of \cite{BUSM} and Section 4.2 of \cite{SSSG}, so we omit most proofs and refer to those papers for a more thorough explanation. 

To proceed, we require the following definition of \cite{CLSM}, which defines a class of initial data for which it is possible to show quick convergence of Dyson Brownian motion. 

\begin{definition}[{\cite[Definition 2.1]{CLSM}}]

Fix some $E_0 \in \mathbb{R}$, and let $\delta$ be a positive real number. For each positive real number $N$, let $r = r_N$ and $R = R_N$ be two parameters satisfying $N^{\delta - 1} \le r \le N^{-\delta}$ and $N^{\delta} r \le R \le N^{-\delta}$. 

We call a diagonal $N \times N$ matrix $\textbf{D} = \textbf{D}_N = \{ V_1, V_2, \ldots , V_N \}$ \emph{$(r, R)$-regular with respect to $E_0$} if there exist constants $c, C > 0$ (independent of $N$) such that the estimates
\begin{flalign*}
c \le \Im m_{\textbf{D}} (E + \textbf{i} \eta)  \le C; \qquad |V_i| \le N^C, 
\end{flalign*}

\noindent both hold for all $E \in (E - R, E + R)$ and $r \le \eta \le 10$, where $m_{\textbf{D}} (z)$ denotes the Stieltjes transform of $\textbf{D}$ for all $z \in \mathbb{H}$, as defined in \eqref{mn}. We call an arbitrary symmetric matrix $\textbf{M}$ \emph{$(r, R)$-regular with respect to $E_0$} if $\textbf{D} (M)$ is $(r, R)$-regular with respect to $E_0$, where $\textbf{D} (M)$ denotes a diagonal matrix whose entries are the eigenvalues of $\textbf{M}$. 

\end{definition}

The results of \cite{CLSM, FEUM} essentially state that, if we start with a $(r, R)$-regular diagonal matrix and then add an independent small Gaussian component of order greater than $r$ but less than $R$, then the local statistics of the result will asymptotically coincide with those of the GOE. To state this more precisely, we must introduce the free convolution \cite{FCSD} of a probability distribution with the semicircle law. 

To that end, fix $N \in \mathbb{Z}_{> 0}$ and a symmetric $N \times N$ matrix $\textbf{A}$. For each $s \ge 0$, define $\textbf{A}^{(s)} = \textbf{A} + s^{1 / 2} \textbf{GOE}_N$. Further denote by $m^{(s)} (z) = m_{\textbf{A}^{(s)}} (z)$ the Stieltjes transform \eqref{mn} of the empirical spectral density of $\textbf{A}^{(s)}$, which we denote by $\rho^{(s)} (x) = \pi^{-1} \lim_{\eta \rightarrow 0} \Im m^{(s)} (E + \textbf{i} \eta)$. 

For each $i \in [1, N]$, let $\gamma_i$ and $\gamma_i^{(s)}$ denote the \emph{classical eigenvalue locations} of the distributions $\rho_{\semicircle}$ and $\rho^{(s)}$, respectively, defined by the equations 
\begin{flalign*}
\displaystyle\int_{-\infty}^{\gamma_i} \rho_{\semicircle} (x) dx = \displaystyle\frac{i}{N}; \qquad \displaystyle\int_{-\infty}^{\gamma_i^{(s)}} \rho^{(s)} (x) dx = \displaystyle\frac{i}{N}. 
\end{flalign*}

The following two theorems establish the universality of gap statistics and correlation functions of the random matrix $\textbf{M}^{(s)}$, assuming that $\textbf{M}$ is regular.

\begin{prop}[{\cite[Theorem 2.5]{CLSM}}]

\label{universalityperturbation1}

Let $N$ be a positive integer, and let $r = r_N$ and $R = R_N$ be positive real parameters dependent on $N$. Fix a real number $\kappa > 0$, and let $\textbf{\emph{M}}$ be a real, symmetric $N \times N$ matrix. Assume that $\textbf{\emph{M}}$ is $(r, R)$-regular with respect to some fixed $E \in (\kappa - 2, 2 - \kappa)$. 

Fix $\delta > 0$ (independent of $N$), and assume that there exists some $s > 0$ satisfying $N^{\delta} r < s < N^{-\delta} R$. Let $i \in [1, N]$ be an integer satisfying $\gamma_i^{(s)} \in [E - G / 2, E + G / 2]$. 

Fix a positive integer $k$. Then, there exists a sufficiently small real number $c = c_{\delta; k} > 0$ such that the following holds. For any compactly supported smooth function $F \in \mathcal{C}_0^{\infty} (\mathbb{R}^k)$ and any positive integers $i_1 < i_2 < \cdots < i_k < N^c$, we have (for sufficiently large $N$) that 
\begin{flalign}
\label{fmt}
\begin{aligned}
\bigg| & \mathbb{E}_{\textbf{\emph{M}}_t} \Big[ 	F \big( N \rho^{(s)} (\gamma_i^{(s)}) (\lambda_i - \lambda_{i + i_1}), N \rho^{(s)} (\gamma_i^{(s)}) (\lambda_i - \lambda_{i + i_2}), \ldots , N \rho_i^{(s)} (\gamma_i^{(s)}) (\lambda_i - \lambda_{i + i_n} )\big) \Big] \\
& - \mathbb{E}_{\textbf{\emph{GOE}}_N} \Big[ F \big( N \rho_{\semicircle} (\gamma_i) (\lambda_i - \lambda_{i + i_1}), N \rho_{\semicircle} (\gamma_i) (\lambda_{i} - \lambda_{i + 1}), \ldots , N \rho_{\semicircle} (\gamma_i) (\lambda_{i + n - 1} - \lambda_{i + n} ) \big) \Big]  \bigg| < N^{-c}. 
\end{aligned} 
\end{flalign}

\end{prop}

\begin{prop}[{\cite[Theorem 2.2]{FEUM}}]

\label{universalityperturbation2}

Adopt the notation of Proposition \ref{universalityperturbation1}, and fix a positive integer $k$. Then, there exists a sufficiently small real number $c = c_{\delta; k} > 0$ such that the following holds. For any $F \in \mathcal{C}_0^{\infty} (\mathbb{R}^k)$, we have (for sufficiently large $N$) that 
\begin{flalign}
\label{pmt}
\begin{aligned}
& \Bigg| \displaystyle\int_{\mathbb{R}^k} F (a_1, a_2, \ldots , a_k) p_{\textbf{\emph{M}}^{(s)}}^{(k)} \left( E + \displaystyle\frac{a_1}{N \rho^{(s)} (E)}, E + \displaystyle\frac{a_2}{N \rho^{(s)} (E)}, \ldots , E + \displaystyle\frac{a_k}{N \rho^{(s)} (E)} \right) \displaystyle\prod_{j = 1}^k d a_j \\
& - \displaystyle\int_{\mathbb{R}^k} F (a_1, a_2, \ldots , a_k) p_{\textbf{\emph{GOE}}_N}^{(k)} \left( E + \displaystyle\frac{a_1}{N \rho_{\semicircle} (E)}, E + \displaystyle\frac{a_2}{N \rho_{\semicircle} (E)}, \ldots , E + \displaystyle\frac{a_k}{N \rho_{\semicircle} (E)} \right) \displaystyle\prod_{j = 1}^k d a_j \Bigg| < N^{-c}. 
\end{aligned}
\end{flalign}

\end{prop}

These two propositions can be applied to deduce universality of a matrix $\textbf{H}_t$, defined from the original generalized Wigner matrix $\textbf{H} = \{ h_{ij} \}$, as follows. For each $1 \le i, j \le N$, let $B_{ij} (s)$ be a Brownian motion so that $B_{ij} (s) = B_{ji} (s)$ and the $\big\{ B_{ij} (s) \big\}$ are mutually independent (and also independent from $\textbf{H}$) for $1 \le i \le j \le N$. Denote by $h_{ij} (s)$ the solution to the Ornstein-Uhlenbeck equation 
\begin{flalign}
\label{hijt}
d h_{ij} (s) = N^{-1 / 2} d B_{ij} (s) - (2 N s_{ij})^{-1} h_{ij} (s) ds, 
\end{flalign}

\noindent and define the $N \times N$ random real symmetric matrix $\textbf{H}_t = \big\{ h_{ij} (s) \big\}$. 

Using Proposition \ref{universalityperturbation1} and Proposition \ref{universalityperturbation2}, one can deduce the following result. 

\begin{prop}

\label{universalityperturbation3}

Fix constants $\kappa > 0$; $0 < \delta < \varepsilon < 1 / 2$; $0 < c_1 < 1 < C_1$; and $C_2 > 1$. Let $\big\{ \textbf{\emph{H}} = \textbf{\emph{H}}_N \big\}_{N \ge 1}$ be a family of generalized Wigner matrices as in Definition \ref{momentassumption}. Let $\lambda_1, \lambda_2, \ldots , \lambda_N$ denote the eigenvalues of $\textbf{\emph{H}}$, and denote $t = t_N = t^{\delta - 1}$. Define $\textbf{\emph{H}}_t$ as above. 

Fix a positive integer $k$. Then, there exists a sufficiently small real number $c = c_{\delta; k} > 0$ such that the following holds. For any compactly supported smooth function $F \in \mathcal{C}_0^{\infty} (\mathbb{R}^k)$ and any positive integers $i_1 < i_2 < \cdots < i_k < N^c$, we have (for sufficiently large $N$) that
\begin{flalign}
\label{htf}
\begin{aligned}
\bigg| \mathbb{E}_{\textbf{\emph{H}}_t} \Big[ & 	F \big( N (\lambda_i - \lambda_{i + i_1}), N (\lambda_i - \lambda_{i + i_2}), \ldots , N (\lambda_i - \lambda_{i + i_n} )\big) \Big] \\
& - \mathbb{E}_{\textbf{\emph{GOE}}_N} \Big[ F \big( N (\lambda_i - \lambda_{i + i_1}), N (\lambda_{i} - \lambda_{i + 1}), \ldots , N (\lambda_{i + n - 1} - \lambda_{i + n} ) \big) \Big]  \bigg| < N^{-c}. 
\end{aligned}
\end{flalign}

Furthermore, we have (for sufficiently large $N$) that 
\begin{flalign}
\label{pht}
\begin{aligned}
\Bigg| & \displaystyle\int_{\mathbb{R}^k} F (a_1, a_2, \ldots , a_k) p_{\textbf{\emph{H}}_t}^{(k)} \left( E + \displaystyle\frac{a_1}{N \rho_{\semicircle} (E)}, E + \displaystyle\frac{a_2}{N \rho_{\semicircle} (E)}, \ldots , E + \displaystyle\frac{a_k}{N \rho_{\semicircle} (E)} \right) \displaystyle\prod_{j = 1}^k d a_j \\
& - \displaystyle\int_{\mathbb{R}^k} F (a_1, a_2, \ldots , a_k) p_{\textbf{\emph{GOE}}_N}^{(k)} \left( E + \displaystyle\frac{a_1}{N \rho_{\semicircle} (E)}, E + \displaystyle\frac{a_2}{N \rho_{\semicircle} (E)}, \ldots , E + \displaystyle\frac{a_k}{N \rho_{\semicircle} (E)} \right) \displaystyle\prod_{j = 1}^k d a_j \Bigg| < N^{-c}. 
\end{aligned}
\end{flalign}

\end{prop}

Given the local semicircle law Theorem \ref{localmoments} and the universality statements Proposition \ref{universalityperturbation1} and Proposition \ref{universalityperturbation2} above, the proof of this proposition is very similar to that of Theorem 3.1 in \cite{BUSM} and Proposition 4.9 in \cite{SSSG}; therefore, it is omitted. However, let us briefly explain the idea of the proof, referring to the references \cite{BUSM, SSSG} for the remaining details. 

First observe that $\textbf{H}_t$ is formed by $\textbf{H}$ from applying an Ornstein-Uhlenbeck process for time $t = N^{\delta - 1}$, while Proposition \ref{universalityperturbation1} and Proposition \ref{universalityperturbation2} are stated for matrices of the form $\textbf{M} + s^{1 / 2} \textbf{GOE}_N$. It happens that $\textbf{H}_t$ is also of the latter form (see, for example equation (2.17) in \cite{BUSM}), with $\textbf{M} = \textbf{H}_t^{(1)} = \{ h_{ij; t}^{(1)} \}$ defined by 
\begin{flalign*}
h_{ij; t}^{(1)} = e^{- t / 2 N s_{ij}} h_{ij} + N^{-1 / 2} B_{ij} (s) \sqrt{N s_{ij} \big( 1 - e^{t / N s_{ij}} \big) - r \left( \displaystyle\frac{1 + \textbf{1}_{i = j}}{2} \right) \big( 1 - e^{-t / r} \big) },
\end{flalign*}

\noindent where $r = N \min_{1 \le i, j \le N} s_{ij} \in (c_1, C_1)$. In particular, the law of $\textbf{H}_t$ is that of 
\begin{flalign}
\label{sh1}
\textbf{H}_t^{(1)} + s^{1 / 2} \textbf{GOE}_N, \qquad \text{where $s = \displaystyle\sqrt{r(1 - e^{-t / r})}{2}$,} 
\end{flalign} 

\noindent and $\textbf{GOE}_N$ is chosen to be independent from $\textbf{H}_t^{(1)}$; observe that $s$ is of order $t = N^{\delta - 1}$. 

It can be quickly verified that $\textbf{H}_t^{(1)}$ is also a generalized Wigner matrix in the sense of Definition \ref{momentassumption}, meaning by Theorem \ref{localmoments} that it satisfies a local semicircle law on some event $\Omega$ that has probability at least $1 - C N^{-c \log \log N}$, for some constants $c, C > 0$. One can then condition on the matrix $\textbf{H}_t^{(1)}$ and apply Proposition \ref{universalityperturbation1} and Proposition \ref{universalityperturbation2} with $\textbf{A} = \textbf{H}_t^{(1)}$ and $s$ as in \eqref{sh1}, to deduce that \eqref{fmt} and \eqref{pmt} hold for $\textbf{A}^{(s)} = \textbf{H}_t$. 

The remaining difference between \eqref{fmt} and \eqref{htf} and between \eqref{pmt} and \eqref{pht} is in the scaling. Specifically, one must approximate the factors of $\rho^{(s)} \big( \gamma_j^{(s)} \big)$ by $\rho_{\semicircle} (\gamma_i)$ in \eqref{fmt} and the factors of $\rho^{(s)} (E)$ by $\rho_{\semicircle} (E)$ in \eqref{pmt}. This approximation can be justified using the local semicircle law Theorem \ref{localmoments}; this can be done in a very similar way to what was explained in Lemma 3.3 and Lemma 3.4 of \cite{BUSM} and Lemma 4.12 of \cite{SSSG} and thus we omit further details. 

This provides an outline of the proof of Proposition \ref{universalityperturbation3} assuming Theorem \ref{localmoments}, Proposition \ref{universalityperturbation1}, and Proposition \ref{universalityperturbation2}; we again refer to Section 3 of \cite{BUSM} and Section 4.2 of \cite{SSSG} for a more comprehensive exposition.

\subsection{Continuity Estimates}

\label{HHt}

Recall that one of our goals is to establish Theorem \ref{bulkfunctions}, which states that the correlation functions of $\textbf{H}$ are universal; in view of Proposition \ref{universalityperturbation3} it suffices to show that the correlation functions of the perturbed matrix $\textbf{H}_t$ (from \eqref{hijt}) equal those of $\textbf{H}$ in the large $N$ limit. The following lemma, which first appeared in some weaker form as Lemma 6.4 in \cite{BUGM} (but was later \cite{DRM, BUSM} altered, with very little modification in the proof, to essentially as below) provides a sufficient condition for when the correlation functions of two generalized Wigner matrices asymptotically coincide.

\begin{lem}[{\cite[Theorem 15.3]{DRM}, \cite[Theorem 6.4]{BUGM}, \cite[Theorem 5.3]{BUSM}} ]

\label{hhtcorrelations}

Fix $\kappa > 0$, $\varepsilon > 0$, $0 < c_1 < 1 < C_1$, and $C_2 > 1$. Let $\{ \textbf{\emph{A}} = \textbf{\emph{A}}_N \}_{N \in \mathbb{Z}_{\ge 1}}$ and $\{ \textbf{\emph{B}} = \textbf{\emph{B}}_N \}_{N \in \mathbb{Z}_{\ge 1}}$ be two families of generalized Wigner matrices. Further fix an arbitrary $k \in \mathbb{Z}_{\ge 1}$ and an arbitrary real $\omega \in (0, 1]$ (bounded away from $0$ independently of $N$). Also fix positive integers $r_1, r_2, \ldots , r_k$, and let $\{ z_1^{(j)}, z_2^{(j)}, \ldots , z_{r_j}^{(j)} \}_{1 \le j \le k} \subset \mathbb{H}$ be families of complex numbers such that $\Re z_i^{(j)} \in [\kappa - 2, 2 - \kappa]$ and $\Im z_i^{(j)} \in [N^{-\omega - 1}, N^{-1}]$ for each $i, j$.  

For each $z \in \mathbb{H}$, denote $\textbf{\emph{G}}^{(\textbf{\emph{A}})} (z) = (\textbf{\emph{A}} - z)^{-1}$ and define $\textbf{\emph{G}}^{(\textbf{\emph{B}})} (z)$ similarly. Assume that there exist constants $\widetilde{c} = \widetilde{c}_{k; r_1, r_2, \ldots , r_k} > 0$ and $\widetilde{C} = \widetilde{C}_{k; r_1, r_2, \ldots , r_k} > 0$ such that the following holds whenever $\omega < \widetilde{c}$. For any compactly supported smooth function $\Lambda \in \mathcal{C}_0^{\infty} (\mathbb{C}^k)$ satisfying 
\begin{flalign}
\label{lambdaderivatives}
\displaystyle\max_{1 \le |\alpha| \le 4} \displaystyle\sup_{|y_j| \le N^{\beta}} \big| \partial^{\alpha} \Lambda (y_1, y_2, \ldots , y_k) \big| \le N^{\widetilde{C} \beta}; \qquad \displaystyle\max_{1 \le |\alpha| \le 4} \displaystyle\sup_{|y_j| \le N^2 } \big| \partial^{\alpha} \Lambda (y_1, y_2, \ldots , y_k) \big| \le N^{\widetilde{C}},
\end{flalign}

\noindent for each $\beta > 0$, we have (for sufficiently large $N$) that 
\begin{flalign}
\label{ablambda}
\begin{aligned}
\Bigg| \mathbb{E} & \bigg[ \Lambda \Big( N^{-r_1} \Tr \displaystyle\prod_{j = 1}^{r_1} \textbf{\emph{G}}^{(\textbf{\emph{A}})} \big( z_j^{(1)} \big), N^{-r_2} \Tr \displaystyle\prod_{j = 1}^{r_2} \textbf{\emph{G}}^{(\textbf{\emph{A}})} \big( z_j^{(2)} \big), \ldots , N^{-r_k} \Tr \displaystyle\prod_{j = 1}^{r_k} \textbf{\emph{G}}^{(\textbf{\emph{A}})} \big( z_j^{(k)} \big) \Big)  \bigg] \\
& - \mathbb{E} \bigg[ \Lambda \Big( N^{-r_1} \Tr \displaystyle\prod_{j = 1}^{r_1} \textbf{\emph{G}}^{(\textbf{\emph{B}})} \big( z_j^{(1)} \big), N^{-r_2} \Tr \displaystyle\prod_{j = 1}^{r_2} \textbf{\emph{G}}^{(\textbf{\emph{B}})} \big( z_j^{(2)} \big), \ldots , N^{-r_k} \Tr \displaystyle\prod_{j = 1}^{r_k} \textbf{\emph{G}}^{(\textbf{\emph{B}})} \big( z_j^{(k)} \big) \Big)  \bigg] \Bigg| < N^{-\widetilde{c}}. 
\end{aligned}
\end{flalign}

\noindent Then, there exists some $c = c_k > 0$ such that for any $F \in \mathcal{C}_0^{\infty} (\mathbb{R}^k)$, we have (for sufficiently large $N$) that 
\begin{flalign*}	
\Bigg| \displaystyle\int_{\mathbb{R}^n} F (a_1, a_2, \ldots , a_N) & \bigg( p_{\textbf{\emph{A}}}^{(k)} \Big( E + \displaystyle\frac{a_1}{N}, E + \displaystyle\frac{a_2}{N}, \ldots , E + \displaystyle\frac{a_N}{N} \Big) \\
& - p_{\textbf{\emph{B}}}^{(k)} \Big( E + \displaystyle\frac{a_1}{N}, E + \displaystyle\frac{a_2}{N}, \ldots , E + \displaystyle\frac{a_N}{N} \Big) \bigg) \displaystyle\prod_{i = 1}^N d a_i \Bigg| < N^{- c}.
\end{flalign*}

\end{lem}

\noindent In view of Lemma \ref{hhtcorrelations}, we would like to establish the following proposition.

\begin{prop}

\label{functionghht}

Adopt the notation of Theorem \ref{hhtcorrelations}, and let $\{ \textbf{\emph{H}} = \textbf{\emph{H}}_N \}_{N \in \mathbb{Z}_{\ge 0}}$ denote a family of generalized Wigner matrices. If $\omega$ is sufficiently small (in a way that only depends on $\widetilde{C}$ and $\varepsilon$), then there exists some $\delta > 0$ (independent of $N$) such that \eqref{ablambda} holds with $\textbf{\emph{A}} = \textbf{\emph{H}}$ and $\textbf{\emph{B}} = \textbf{\emph{H}}_t$. Here, we have set $t = N^{\delta - 1}$ and recalled the definition of $\textbf{\emph{H}}_t$ from \eqref{hijt}. 

\end{prop}

We will establish Proposition \ref{functionghht} later, in Section \ref{UniversalityMatrix}. 

In the context of less singular Wigner matrix (whose entry laws have at least three moments, for example), proofs of Proposition \ref{functionghht} have appeared in several previous works; for instance, see Lemma 5.2 of \cite{BUSM}. The recent proofs of such results are based on a certain \emph{continuity estimate} that originally appeared as Lemma A.1 of \cite{EMFLQUE}. 

Unfortunately, that lemma assumes that $\mathbb{E} \big[ |h_{ij} \sqrt{N} |^3 \big] < \infty$, which might be false in our setting. Thus we require a modification of that result, which is given below as Lemma \ref{fsmall}. 

To state this lemma, we require some additional notation. Fix some positive integer $N$. For each pair of integers $1 \le a, b \le N$, let $\textbf{X}_{ab}$ denote the $N \times N$ matrix whose entries are all equal to $0$, except for the $(a, b)$ and $(b, a)$ entries which are equal to $1$. 

Furthermore, fix some $N \times N$ real symmetric matrix $\textbf{M} = \{ m_{ij} \}$. For any $\theta \in [0, 1]$ and integers $1 \le a, b  \le N$, let $\Theta_{ab} \textbf{M} = \{ \widetilde{m_{ij}} \}$ denote the $N \times N$ symmetric matrix whose entries $\widetilde{m_{ij}} = m_{ij}$ if $(i, j) \notin \{ (a, b), (b, a) \}$ and $\widetilde{m_{ij}} = \theta m_{ij}$ otherwise. When $\theta = 0$, we denote $\Theta_{ab} \textbf{M} = \textbf{Z}_{ab} \textbf{M}$. 

Moreover, for any smooth function $F$ (from the set of $N \times N$ real symmetric matrices to $\mathbb{C}$), let $\partial_{ij} F$ denote the partial derivative of $F$ in the $X_{ij}$-coordinate. Specifically, we set $\partial_{ij} F (\textbf{M}) = \lim_{y \rightarrow 0} y^{-1} \big( F(\textbf{M} + y \textbf{X}_{ij}) - F(\textbf{M}) \big)$, if it exists. 

Now we have the following estimate.

\begin{lem}

\label{fsmall}

Let $\textbf{\emph{H}} = \{ h_{ij} \}$ be an $N \times N$ generalized Wigner matrix in the sense of Definition \ref{momentassumption}. For any $t \ge 0$, define $\textbf{\emph{H}}_t = \{ h_{ij} (t) \}$ as in \eqref{hijt}. Let $F$ be a smooth function. Then, 
\begin{flalign}
\label{fsmall1}
\Big| \mathbb{E} \big[ F (\textbf{\emph{H}}_t) - F (\textbf{\emph{H}}_0) \big] \Big| \le t N \Xi, 
\end{flalign}

\noindent where 
\begin{flalign}
\label{fsmall2}
\begin{aligned}
\Xi & = \displaystyle\max_{1 \le i, j \le N} \Bigg| s_{ij}^{-1} \mathbb{E} \bigg[ \textbf{\emph{1}}_{|h_{ij} (t)| \ge N^{-\varepsilon / 10}} \big| h_{ij} (t) \big|  	\displaystyle\sup_{0 \le \theta \le 1} \big| \partial_{ij} F (\Theta_{ij} \textbf{\emph{H}}_t) \big| \bigg] \\
& \qquad \qquad + s_{ij}^{-1} \mathbb{E} \bigg[ \textbf{\emph{1}}_{|h_{ij} (t)| \ge N^{-\varepsilon / 10}} \big| h_{ij} (t) \big|^2 \displaystyle\sup_{0 \le \theta \le 1} \big| \partial_{ij}^2 F (\Theta_{ij} \textbf{\emph{H}}_t) \big| \bigg]  \\
& \qquad \qquad + \mathbb{E} \bigg[\textbf{\emph{1}}_{|h_{ij} (t)| \ge N^{-\varepsilon / 10}} \displaystyle\sup_{0 \le \theta \le 1} \big| \partial_{ij}^2 F (\Theta_{ij} \textbf{\emph{H}}_t) \big| \bigg] 	 \\
& \qquad \qquad + s_{ij}^{-1} \mathbb{E} \bigg[ \textbf{\emph{1}}_{|h_{ij} (t)| < N^{-\varepsilon / 10}} \big| h_{ij} (t) \big|^3 \displaystyle\sup_{0 \le \theta \le 1} \big| \partial_{ij}^3 F (\Theta_{ij} \textbf{\emph{H}}_t) \big| \bigg] \\
& \qquad \qquad + \mathbb{E} \bigg[ \textbf{\emph{1}}_{|h_{ij} (t)| < N^{-\varepsilon / 10}} \big| h_{ij} (t) \big| \displaystyle\sup_{0 \le \theta \le 1} \big| \partial_{ij}^3 F (\Theta_{ij} \textbf{\emph{H}}_t) \big| \bigg] \Bigg|. 
\end{aligned}
\end{flalign}
\end{lem}

\begin{proof}

Applying It\^{o}'s Lemma in the definition \ref{hijt} of the $h_{ij} (t)$ yields 
\begin{flalign}
\label{fsmall3}
\partial_t \mathbb{E} \big[ F (\textbf{H}_t) \big]= \displaystyle\frac{1}{2N} \displaystyle\sum_{1 \le i \le j \le N} \Big( \mathbb{E} \big[ \partial_{ij}^2 F (\textbf{H}_t) \big] - s_{ij}^{-1} \mathbb{E} \big[ h_{ij} (t) \partial_{ij} F (\textbf{H}_t) \big] \Big). 
\end{flalign}

\noindent We would now like to Taylor expand both $\partial_{ij}^2 F (\textbf{H}_t)$ and $h_{ij} (t) \partial_{ij} F (\textbf{H}_t)$ in a neighborhood of $h_{ij} (t) = 0$. However, there is the issue that $h_{ij} (t)$ might not be small. 

To resolve that, we observe 
\begin{flalign}
\label{hfhderivative}
\begin{aligned}
\Big| h_{ij} (t) & \partial_{ij} F (\textbf{H}_t) - h_{ij} (t) \partial_{ij} F (\textbf{Z}_{ij} \textbf{H}_t) - h_{ij} (t)^2 \partial_{ij}^2 F (\textbf{Z}_{ij} \textbf{H}_t) \Big| \\
& \le \textbf{1}_{|h_{ij} (t)| \ge N^{-\varepsilon / 10}} \bigg( 2 \big| h_{ij} (t) \big|  \displaystyle\sup_{0 \le \theta \le 1} \big| \partial_{ij} F (\Theta_{ij} \textbf{H}_t) \big| + \big| h_{ij} (t) \big|^2 \displaystyle\sup_{0 \le \theta \le 1} \big| \partial_{ij}^2 F (\Theta_{ij} \textbf{H}_t) \big|  \bigg) \\
& \quad + \textbf{1}_{|h_{ij} (t)| < N^{-\varepsilon / 10}} \big| h_{ij} (t) \big|^3 \displaystyle\sup_{0 \le \theta \le 1} \big| \partial_{ij}^3 F (\Theta_{ij} \textbf{H}_t) \big|. 
\end{aligned}
\end{flalign}

\noindent Similarly, we find that 
\begin{flalign}
\label{2fhderivative}
\begin{aligned}
\Big| \partial_{ij}^2 F (\textbf{H}_t) - \partial_{ij}^2 F (\textbf{Z}_{ij} \textbf{H}_t) \Big| & \le 2 \textbf{1}_{|h_{ij} (t)| \ge N^{-\varepsilon / 10}} \displaystyle\sup_{0 \le \theta \le 1} \big| \partial_{ij}^2 F (\Theta_{ij} \textbf{H}_t) \big| 	 \\
& \quad + \textbf{1}_{|h_{ij} (t)| < N^{-\varepsilon / 10}} \big| h_{ij} (t) \big| \displaystyle\sup_{0 \le \theta \le 1} \big| \partial_{ij}^3 F (\Theta_{ij} \textbf{H}_t) \big|. 
\end{aligned}
\end{flalign}

\noindent Now, since the $h_{ij} (t)$ are centered, $\mathbb{E} \big[ |h_{ij} (t)|^2 \big] = \mathbb{E} \big[ |h_{ij}|^2 \big] = s_{ij}$, and $h_{ij} (t)$ is independent from $\textbf{Z}_{ij} \textbf{H}_t$, we find that 
\begin{flalign}
\label{hijfexpectation}
s_{ij}^{-1} \mathbb{E} \big[ h_{ij} (t) \partial_{ij} F (\textbf{Z}_{ij} \textbf{H}_t) + h_{ij} (t)^2 \partial_{ij}^2 F (\textbf{Z}_{ij} \textbf{H}_t) \big] = \mathbb{E} \big[ \partial_{ij}^2 F (\textbf{Z}_{ij} \textbf{H}_t) \big].
\end{flalign}
 
\noindent The claimed estimate \eqref{fsmall1} now follows from combining \eqref{fsmall3}, \eqref{hfhderivative}, \eqref{2fhderivative}, and \eqref{hijfexpectation}, and summing over all $1 \le i \le j \le N$. 
\end{proof}

To use Lemma \ref{fsmall}, we must estimate certain derivatives of $F$. The following lemma does this in the case when $F = m_{\textbf{H}} (z)$ is the Stieltjes transform of $\textbf{H}$, which will be useful for us later in Section \ref{UniversalityMatrix}. 

\begin{lem}

\label{derivativesm}

Fix constants $\kappa > 0$, $\varepsilon > 0$, $0 < c_1 < 1 < C_1$, and $C_2 > 1$. Let $\textbf{\emph{H}}$ be an $N \times N$ generalized Wigner matrix in the sense of Definition \ref{momentassumption}. Then, there exist constants $0 < c < C$ such that 
\begin{flalign}
\label{entriesthetag}
\mathbb{P} \left[ \displaystyle\sup_{z \in \mathscr{D}_{\kappa; N}} \displaystyle\max_{1 \le i, j \le N} \displaystyle\max_{1 \le a, b \le N} \displaystyle\sup_{0 \le \theta \le 1} \big| (\Theta_{ab} \textbf{\emph{H}} - z)_{ij}^{-1} \big| > C \right] \le C N^{-c \log \log N}, 
\end{flalign}

\noindent where we recall the definition of $\mathscr{D}_{\kappa; N}$ from below \eqref{dkappanr}. Furthermore, fix a real number $0 \le \omega < 1$ and set $r = r_N = r_{N; \omega} = N^{-1 - \omega}$. Then, we have that 
\begin{flalign}
\label{entriesthetagomega}
\mathbb{P} \left[ \displaystyle\sup_{z \in \mathscr{D}_{\kappa; N; r}} \displaystyle\max_{1 \le i, j \le N} \displaystyle\max_{1 \le a, b \le N} \displaystyle\sup_{0 \le \theta \le 1} \big| (\Theta_{ab} \textbf{\emph{H}} - z)_{ij}^{-1} \big| > C N^{\omega} (\log N)^{C \log \log N} \right] \le C N^{-c \log \log N}. 
\end{flalign}

\noindent Moreover, for any integer $1 \le k \le 4$, we have that 
\begin{flalign}
\label{entriesthetagomegaderivative}
\mathbb{P} \left[ \displaystyle\sup_{z \in \mathscr{D}_{\kappa; N; r}} \displaystyle\max_{\substack{1 \le i, j \le N \\ 1 \le a, b \le N}} \displaystyle\sup_{0 \le \theta \le 1} \Big| N^{-1} \partial_{ij}^{(k)} \big( \Tr (\Theta_{ab} \textbf{\emph{H}} - z )^{-1} \big) \Big| > C N^{(k + 1) \omega} (\log N)^{C \log \log N} \right] \le C N^{- c \log \log N}. 
\end{flalign}

\end{lem}

\begin{proof}

We begin with the proof of \eqref{entriesthetag}. First observe that, since $\Var h_{ij} < C_1 N^{-1}$, each $\Theta_{ab} \textbf{H}$ is a generalized Wigner matrix in the sense of Definition \ref{momentassumption} (perhaps with a slightly larger value of $C_1$). Hence, we can apply Theorem \ref{gijestimate} to deduce the existence positive constants $\widetilde{c}$ and $\widetilde{C}$ such that 
\begin{flalign*}
\mathbb{P} \left[ \displaystyle\sup_{z \in \mathscr{D}_{\kappa; N}}  \displaystyle\max_{1 \le i, j \le N} \big| (\Theta_{ab} \textbf{H} - z)_{ij}^{-1} \big| > \widetilde{C} \right] < \widetilde{C} N^{-\widetilde{c} \log \log N}, 
\end{flalign*}

\noindent for each fixed $\theta \in [0, 1]$ and fixed $1 \le a, b \le N$. Now, let $\mathbb{M}_N = \big\{ \theta \in [0, 1] : N^8 \theta \in \mathbb{Z}  \big\}$. From a union estimate, it follows that 
\begin{flalign}
\label{entriesthetag2}
\mathbb{P} \left[ \displaystyle\sup_{z \in \mathscr{D}_{\kappa; N}} \displaystyle\max_{1 \le i, j \le N}  \displaystyle\max_{1 \le a, b \le N} \displaystyle\sup_{\theta \in \mathbb{M}_N} \big| (\Theta_{ab} \textbf{H} - z)_{ij}^{-1} \big| > \widetilde{C} \right] < \widetilde{C} N^{10 - \widetilde{c} \log \log N}. 
\end{flalign}

Now the existence of constants $0 < c < C$ satisfying \eqref{entriesthetag} follows from \eqref{entriesthetag2} and the fact that $\big| ( \Theta_{ab} \textbf{H} - z )_{ij}^{-1} - ( \Theta_{ab} \textbf{H} - z' )_{ij}^{-1} \big| < 1$ if $\Im z, \Im z' > N^{-2}$ and $|z - z'| > N^{-7}$; the latter estimate holds due to the resolvent identity \eqref{resolvent} and the deterministic estimate \eqref{gijeta}. 

The estimate \eqref{entriesthetagomega} follows directly from \eqref{entriesthetag} and the fact that $\Gamma (E + \textbf{i} \eta / R) \le R \Gamma (E + \textbf{i} \eta)$, for any real number $R$, where $\Gamma (z) = \Gamma_{\textbf{M}} (z) = \max_{1 \le i, j \le N} \max \big\{ 1, \big| (\textbf{M} - z)_{ij}^{-1} \big| \big\} $ for any $N \times N$ deterministic matrix $\textbf{M}$; the latter estimate appears as Lemma 2.1 of \cite{LSLRRG}. 

To derive \eqref{entriesthetagomegaderivative}, one uses \eqref{entriesthetagomega} and the fact that 
\begin{flalign}
\label{derivativeijabhz}
\partial_{ij}^{(k)} \Tr (\Theta_{ab} \textbf{H} - z)^{-1} = (-1)^k k! \Tr \Big( \big(  (\Theta_{ab} \textbf{H} - z)^{-1}  \textbf{X}_{ij} \big)^k   (\Theta_{ab} \textbf{H} - z)^{-1} \Big). 
\end{flalign}

\noindent In particular, since $\textbf{X}_{ij}$ only has two nonzero entries (both of which are equal to $1$), this trace is a sum of at most $2^k N$ terms that are each products of at most $k + 1$ entries of $(\Theta_{ab} \textbf{H} - z)^{-1}$. Each of these entries is bounded by $C N^{\omega}	(\log N)^{C \log \log N}$ with very high probability in view of \eqref{entriesthetagomega}, from which we deduce \eqref{entriesthetagomegaderivative} (after incrementing $C$ if necessary). 
\end{proof}

\subsection{Comparing \texorpdfstring{$\textbf{H}$}{} and \texorpdfstring{$\textbf{H}_t$}{}}

\label{UniversalityMatrix}

We now use the estimates from Section \ref{HHt} to establish Proposition \ref{functionghht}. 

\begin{proof}[Proof of Proposition \ref{functionghht}]

To ease notation, we assume that $k = 1$ and $r_1 = 1$; the proof in the more general case is very similar. 

For any symmetric matrix $\textbf{M}$, denote $\widetilde{\Lambda} (\textbf{M}) = \Lambda \big( N^{-1} \Tr \textbf{G}^{(\textbf{M})} (z_1^{(1)}) \big)$. To establish the proposition, we would like to apply Lemma \ref{fsmall} with $F = \widetilde{\Lambda}$; this requires estimates on the derivatives of $\widetilde{\Lambda}$. To that end, observe that 
\begin{flalign}
\label{derivativelambdaij}
\partial_{ij}^{(k)} \widetilde{\Lambda} (\textbf{M}) = \displaystyle\sum_{j = 0}^k \binom{k}{j} \Big( \partial^{(j)} \Lambda (\textbf{M}) \Big) \partial_{ij}^{(k - j)} \big( N^{-1} \Tr (\textbf{M} - z)^{-1} \big). 
\end{flalign}

We will obtain two types of estimates on these derivatives of $\widetilde{\Lambda}$, a deterministic bound and a singificantly improved high-probability estimate. Let us begin with the deterministic estimate. To that end, observe that in view of \eqref{derivativeijabhz}, \eqref{gijeta}, the fact that $\omega < 1$, 	and the fact that $k \le 4$, we deterministically have that
\begin{flalign}
\label{derivativelambdaij1}
\Big| \partial_{ij}^{(k - m)} \big( N^{-1} \Tr (\textbf{M} - z)^{-1} \big) \Big| < 1000 N^{20}. 
\end{flalign} 

\noindent Furthermore, the second estimate in \eqref{lambdaderivatives} and the fact that $N^{-1} \Tr \textbf{G}^{(\textbf{H})} (z_1^{(1)}) < N^2$ (again due to \eqref{gijeta}) together yield that $\big| \partial^{(k)}  \Lambda \big( N^{-1} \Tr (\Theta_{ij} \textbf{M} - z_1^{(1)} )^{-1}\big)  \big| < N^{\widetilde{C}}$. Combining with \eqref{derivativelambdaij} and \eqref{derivativelambdaij1} yields the deterministic estimate 
\begin{flalign}
\label{derivativelambdaij2}
\Big| \partial_{ij}^{(k)} \widetilde{\Lambda} (\textbf{M}) \Big| < 10000 N^{\widetilde{C} + 20}. 
\end{flalign}

Now let us obtain a very high probability estimate on the right side of \eqref{derivativelambdaij} in the case when $\textbf{M} = \textbf{H}_t$ (which we recall is a generalized Wigner matrix). In view of \eqref{entriesthetagomegaderivative} and the fact that $k \le 4$, we can estimate 
\begin{flalign}
\label{derivativelambdaij3}
\begin{aligned} 
\mathbb{P} & \left[ \displaystyle\sup_{z \in \mathscr{D}_{\kappa; N; r}} \displaystyle\max_{\substack{1 \le i, j \le N \\ 1 \le a, b \le N}} \displaystyle\sup_{0 \le \theta \le 1} \Big| \partial_{ij}^{(k - j)} \big( N^{-1} \Tr (\Theta_{ab} \textbf{H}_t - z)^{-1} \big) \Big| > \overline{C} N^{5 \omega} (\log N)^{\overline{C} \log \log N} \right] \\
& \qquad \qquad \qquad \qquad \qquad \qquad \qquad \qquad \qquad \qquad \qquad \qquad \qquad < \overline{C} N^{-\overline{c} \log \log N},
\end{aligned}
\end{flalign}

\noindent for some constants $\overline{c}$ and $\overline{C}$; above, $r = r_{N; \omega} = N^{-\omega - 1}$. 

Furthermore, in view of the first estimate in \eqref{lambdaderivatives}, the fact that $\omega$ is bounded away from $0$, and the estimate \eqref{entriesthetagomega}, we deduce that 
\begin{flalign}
\label{derivativelambdaij4}
\begin{aligned} 
\mathbb{P} & \left[ \displaystyle\sup_{z \in \mathscr{D}_{\kappa; N; r}} \displaystyle\max_{\substack{1 \le i, j \le N \\ 1 \le a, b \le N}} \displaystyle\sup_{0 \le \theta \le 1} \Big| \partial^{(j)} \Lambda \big( N^{-1} \Tr (\Theta_{ij} \textbf{H}_t - z_1^{(1)} )^{-1} \big) \Big| > \overline{C} N^{2 \widetilde{C} \omega} \right] <\overline{C} N^{-\overline{c} \log \log N},
\end{aligned}
\end{flalign} 

\noindent for sufficiently large $N$. Therefore, it follows from \eqref{derivativelambdaij}, \eqref{derivativelambdaij3}, and \eqref{derivativelambdaij4} that there exist positive constants $\widehat{c}, \widehat{C}$ such that if we denote the event 
\begin{flalign}
\label{evente}
\mathcal{E} = \mathcal{E}_{\widehat{C}} \left\{ \displaystyle\sup_{z \in \mathscr{D}_{\kappa; N; r}} \displaystyle\sup_{0 \le \theta \le 1} \displaystyle\max_{\substack{1 \le i, j \le N \\ 1 \le a, b \le N}} \big| \partial_{ij}^{(k)} \widetilde{\Lambda} (\Theta_{ab} \textbf{H}) \big| > \widehat{C} N^{\widehat{C} \omega} \right\},  
\end{flalign}

\noindent then we have that 
\begin{flalign}
\label{derivativelambdaij5}
\mathbb{P} [\mathcal{E}] < \widehat{C} N^{-\widehat{c} \log \log N}. 
\end{flalign}

Using \eqref{derivativelambdaij2} and \eqref{derivativelambdaij5}, we can estimate the value of $\Xi$ \eqref{fsmall2} from Lemma \ref{fsmall}. In particular, we claim that there exists a constant $C$ such that $\Xi \le C N^{C \omega - \varepsilon / 20}$. 

To establish this, we recall from \eqref{fsmall2} that $\Xi$ is the sum of five terms; each one will be bounded by some quantity of the form $C N^{C \omega - \varepsilon / 20}$. We will only explicitly verify this for the second and fourth term; the remaining three terms can be addressed similarly. 

We begin with the second term. For fixed $i, j$, it is equal to 
\begin{flalign}
\label{derivativelambdaij6}
\begin{aligned}
& s_{ij}^{-1} \mathbb{E} \bigg[ \textbf{1}_{|h_{ij} (t)| \ge N^{-\varepsilon / 10}} \big| h_{ij} (t) \big|^2 \displaystyle\sup_{0 \le \theta \le 1} \big| \partial_{ij}^2 F (\Theta_{ij} \textbf{H}_t) \big| \bigg] \\
& \qquad = s_{ij}^{-1} \mathbb{E} \bigg[ \textbf{1}_{\mathcal{E}} \textbf{1}_{|h_{ij} (t)| \ge N^{-\varepsilon / 10}} \big| h_{ij} (t) \big|^2 \displaystyle\sup_{0 \le \theta \le 1} \big| \partial_{ij}^2 F (\Theta_{ij} \textbf{H}_t) \big| \bigg] \\
& \qquad \qquad + s_{ij}^{-1} \mathbb{E} \bigg[ \textbf{1}_{\overline{\mathcal{E}}} \textbf{1}_{|h_{ij} (t)| \ge N^{-\varepsilon / 10}} \big| h_{ij} (t) \big|^2 \displaystyle\sup_{0 \le \theta \le 1} \big| \partial_{ij}^2 F (\Theta_{ij} \textbf{H}_t) \big| \bigg].
\end{aligned}
\end{flalign}

\noindent where $\overline{\mathcal{E}}$ denotes the complement of the event $\mathcal{E}$. The first summand on the right side of \eqref{derivativelambdaij6} can be bounded by $N^{C - \widehat{c} \log \log N}$ (for some constant $C > 0$), due to the deterministic estimate \eqref{derivativelambdaij2}, the probability estimate \eqref{derivativelambdaij5}, the fact that $|s_{ij}| > c_1 N^{-1}$, and the boundedness of the second moment of $|h_{ij} \sqrt{N}|$. 

The second summand can be bounded as 
\begin{flalign*}
s_{ij}^{-1} \mathbb{E} \bigg[ \textbf{1}_{\overline{\mathcal{E}}} \textbf{1}_{|h_{ij} (t)| \ge N^{-\varepsilon / 10}} \big| h_{ij} (t) \big|^2 \displaystyle\sup_{0 \le \theta \le 1} \big| \partial_{ij}^2 F (\Theta_{ij} \textbf{H}_t) \big| \bigg] & \le c_1^{-1} \widehat{C} N^{1 + \widehat{C} \omega} \mathbb{E} \big[ \textbf{1}_{|h_{ij} (t)| \ge N^{-\varepsilon / 10}} \big| h_{ij} (t) \big|^2 \big] \\
& \le c_1^{-1} C_2 \widehat{C} N^{\widehat{C} \omega - \varepsilon / 20}. 
\end{flalign*}

\noindent Here, we used the fact that $s_{ij} \ge c_1 N^{-1}$ and the definition \eqref{evente} of the event $\mathcal{E}$ in the first estimate; in the second estimate, we used assumption \ref{moments} and the fact that 
\begin{flalign*}
\mathbb{E} \big[ \textbf{1}_{|h_{ij} (t)| \ge N^{-\varepsilon / 10}} \big| h_{ij} (t) \big|^2 \big] \le N^{\varepsilon^2 / 10} \mathbb{E} \big[ |h_{ij} (t)|^{2 + \varepsilon} \big] \le C_2 N^{- 1 - \varepsilon / 20}.
\end{flalign*}

\noindent Hence, it follows that the left side of \eqref{derivativelambdaij6} is bounded by $N^{C - \widehat{c} \log \log N} + c_1^{-1} C_2 \widehat{C} N^{\widehat{C} \omega - \varepsilon / 20} = \mathcal{O} (N^{\widehat{C} \omega - \varepsilon / 20})$. 

The fourth term in the definition \eqref{fsmall2} of $\Xi$ can be estimated similarly. Specifically, 
\begin{flalign}
\label{derivativelambdaij8}
\begin{aligned}
& s_{ij}^{-1} \mathbb{E} \bigg[ \textbf{1}_{|h_{ij} (t)| < N^{-\varepsilon / 10}} \big| h_{ij} (t) \big|^3 \displaystyle\sup_{0 \le \theta \le 1} \big| \partial_{ij}^3 F (\Theta_{ij} \textbf{H}_t) \big| \bigg]  \\
& \qquad =  s_{ij}^{-1} \mathbb{E} \bigg[ \textbf{1}_{\mathcal{E}} \textbf{1}_{|h_{ij} (t)| < N^{-\varepsilon / 10}} \big| h_{ij} (t) \big|^3 \displaystyle\sup_{0 \le \theta \le 1} \big| \partial_{ij}^3 F (\Theta_{ij} \textbf{H}_t) \big| \bigg]  \\
& \qquad \qquad + s_{ij}^{-1} \mathbb{E} \bigg[ \textbf{1}_{\mathcal{E}} \textbf{1}_{|h_{ij} (t)| < N^{-\varepsilon / 10}} \big| h_{ij} (t) \big|^3 \displaystyle\sup_{0 \le \theta \le 1} \big| \partial_{ij}^3 F (\Theta_{ij} \textbf{H}_t) \big| \bigg]. 
\end{aligned}
\end{flalign}

Again, the first term on the right side of \eqref{derivativelambdaij8} can be bounded by $N^{C - \widehat{c} \log \log N}$ due to \eqref{derivativelambdaij2} and \eqref{derivativelambdaij5}. To estimate the second term, we observe that
\begin{flalign}
\label{derivativelambdaij9} 
\begin{aligned}
s_{ij}^{-1} \mathbb{E} \bigg[ \textbf{1}_{\overline{\mathcal{E}}} \textbf{1}_{|h_{ij} (t)| \ge N^{-\varepsilon / 10}} \big| h_{ij} (t) \big|^3 \displaystyle\sup_{0 \le \theta \le 1} \big| \partial_{ij}^3 F (\Theta_{ij} \textbf{H}_t) \big| \bigg] & \le c_1^{-1} \widehat{C} N^{1 + \widehat{C} \omega} \mathbb{E} \big[ \textbf{1}_{|h_{ij} (t)| < N^{-\varepsilon / 10}} \big| h_{ij} (t) \big|^3 \big] \\
& \le c_1^{-1} C_2 \widehat{C} N^{\widehat{C} \omega - \varepsilon / 20},
\end{aligned}
\end{flalign}

\noindent where in the second estimate we used the fact that 
\begin{flalign*}
\mathbb{E} \big[ \textbf{1}_{|h_{ij} (t)| < N^{-\varepsilon / 10}} \big| h_{ij} (t) \big|^3 \big] < N^{(\varepsilon / 10)(\varepsilon - 1)} \mathbb{E} \big[ |h_{ij} (t)|^{2 + \varepsilon} \big] < N^{- 1 - \varepsilon / 20}.
\end{flalign*}

Thus, we can bound the right side of \eqref{derivativelambdaij9} by $N^{C - \widehat{c} \log \log N} + c_1^{-1} C_2 \widehat{C} N^{\widehat{C} \omega - \varepsilon / 20} = \mathcal{O} (N^{\widehat{C} \omega - \varepsilon / 20})$; this estimates the fourth term in \eqref{fsmall2}.

As mentioned previously, the other three terms in the definition \eqref{fsmall2} of $\Xi$ can be bounded similarly. It follows that there exists a constant $C > 0$ such that $\Xi < C N^{C \omega - \varepsilon / 20}$. Recalling that $t = N^{\delta - 1}$ and inserting the result into \eqref{fsmall1} yields that the left side of \eqref{ablambda} is bounded by $C N^{C \omega - \varepsilon / 20 + \delta}$. Setting $\omega$ and $\delta$ sufficiently small so that $C \omega + \delta < \varepsilon / 40$ then yields \eqref{ablambda} with $\widetilde{c} = \varepsilon / 40$; this confirms the proposition. 
\end{proof}

\begin{proof}[Proof of Theorem \ref{bulkfunctions}]
This follows from Proposition \ref{universalityperturbation3}, Lemma \ref{hhtcorrelations}, and Proposition \ref{functionghht}. 
\end{proof}

This establishes Theorem \ref{bulkfunctions}. We can also prove Theorem \ref{gapsfunctions} but, given what we have already done, this is very similar to what was already explained in several previous works; see, for example, Section 4 and Lemma 5.1 of \cite{BUSM} or Section 4.1.2 of \cite{SSSG}. 

The main difference between what should be done in our setting and what was done in their setting is that we must use the less restrictive continuity estimate Lemma \ref{fsmall} as opposed to Lemma 4.5 of \cite{SSSG} or Lemma 4.3 of \cite{BUSM} (which both require that $\mathbb{E} \big[ |h_{ij} \sqrt{N}|^3 \big]$ is bounded). Usage of Lemma \ref{fsmall} has already been explained in the proof of Proposition \ref{functionghht} above and, since all other parts of the proof of Theorem \ref{gapsfunctions} are essentially the same as those in \cite{SSSG, BUSM}, we omit further details.

\appendix

\section{Large Deviation Estimates}

\label{Deviations}

In this section we state the large deviation results that were useful to us in Section \ref{LawLarge}, Section \ref{LawSmallNotDeviant}, and Section \ref{LawSmallDeviant}. The following proposition appears as the first part of Lemma 3.8 of \cite{SSG}. 

\begin{prop}[{\cite[Lemma 3.8]{SSG}}]

\label{largeprobability1}

Let $N \in \mathbb{Z}_{> 1}$ and $q = q_N > 1$. Let $X_1, X_2, \ldots , X_N$; $Y_1, Y_2, \ldots , Y_N$ be centered random variables such that there exists a $C > 0$ satisfying
\begin{flalign}
\label{smallamoment}
\mathbb{E} \big[ |X_i|^p \big] \le \displaystyle\frac{q^2}{N} \left( \displaystyle\frac{C}{q} \right)^p; \qquad \mathbb{E} \big[ |Y_i|^p \big] \le \displaystyle\frac{q^2}{N} \left( \displaystyle\frac{C}{q} \right)^p,
\end{flalign}

\noindent for all $2 \le p \le (\log N)^{\log \log N}$.	Denote $s_i = \mathbb{E} [|X_i|^2]$ for each $1 \le i \le N$. 

Let $\{ R_i \}_{1 \le i \le N}$ and $\{ R_{ij} \}_{1 \le i, j \le N}$ be sequences of real numbers. Then, there exists a constant $\nu = \nu (C) > 0$, only dependent on the constant $C$ in \eqref{smallamoment}, such that the four estimates
\begin{flalign}
\label{suma}
\mathbb{P} \left[ \left| \displaystyle\sum_{j = 1}^N R_j X_j \right| \ge (\log N)^{\xi} \left( q^{-1} \displaystyle\max_{1 \le i \le N} |R_i| + \left( \displaystyle\frac{1}{N} \displaystyle\sum_{j = 1}^N |R_j|^2 \right)^{1 / 2} \right)  \right] \le \exp \big( -\nu (\log N)^{\xi} \big); 
\end{flalign}

\begin{flalign}
\label{bi}
\begin{aligned}
& \mathbb{P} \left[ \left| \displaystyle\sum_{j = 1}^N \big( |X_j|^2 - s_j \big) R_j \right| \ge (\log N)^{\xi} q^{-1} \displaystyle\max_{1 \le i \le N} |R_i|  \right] \le \exp \big( - \nu (\log N)^{\xi} \big);	
\end{aligned}
\end{flalign}

\begin{flalign}
\label{bij}
\begin{aligned}
& \mathbb{P} \left[ \left| \displaystyle\sum_{1 \le i \ne j \le N } X_i R_{ij} X_j \right| \ge (\log N)^{2 \xi} \left( q^{-1} \displaystyle\max_{1 \le i \ne j \le N} |R_{ij}| + \left( \displaystyle\frac{1}{N^2} \displaystyle\sum_{1 \le i \ne j \le N} |R_{ij}|^2 \right)^{1 / 2} \right)  \right]  \\
& \qquad \qquad \qquad \qquad \qquad \qquad \qquad \qquad \qquad \qquad \qquad \qquad \qquad \qquad \le \exp \big( - \nu (\log N)^{\xi} \big); 
\end{aligned}
\end{flalign}

\begin{flalign}
\label{abij}
\begin{aligned}
& \mathbb{P} \left[ \left| \displaystyle\sum_{1 \le i, j \le N } X_i R_{ij} Y_j \right| \ge (\log N)^{2 \xi} \left( 2 q^{-1} \displaystyle\max_{1 \le i, j \le N} |R_{ij}| + \left( \displaystyle\frac{1}{N^2} \displaystyle\sum_{1 \le i, j \le N} |R_{ij}|^2 \right)^{1 / 2} \right)  \right]  \\
& \qquad \qquad \qquad \qquad \qquad \qquad \qquad \qquad \qquad \qquad \qquad \qquad \qquad \qquad \le \exp \big( - \nu (\log N)^{\xi} \big); 
\end{aligned}
\end{flalign}

\noindent all hold for any $2 \le \xi \le \log \log N$. 
\end{prop}

In our setting, the random variables $X_i$ are not centered but instead ``almost centered.'' The following corollary adapts the previous proposition to apply in this slightly more setting.

\begin{cor}

\label{largeprobability2} 

Let $N \in \mathbb{Z}_{> 1}$ and $q = q_N \in (1, \sqrt{N})$. Let $X_1, X_2, \ldots , X_N; Y_1, Y_2, \ldots , Y_N$ be random variables such that there exist $\delta \in (0, 1)$ and $C, C' > 1$ satisfying 
\begin{flalign}
\label{smallmomentsadelta}
\big| \mathbb{E} [ X_i ] \big| \le C' N^{-1 - \delta }; \quad \mathbb{E} \big[ |X_i|^p \big] \le \displaystyle\frac{q^2}{N} \left( \displaystyle\frac{C}{q} \right)^p; \quad \big| \mathbb{E} [ Y_i ] \big| \le C' N^{-1 - \delta }; \quad \mathbb{E} \big[ |Y_i|^p \big] \le \displaystyle\frac{q^2}{N} \left( \displaystyle\frac{C}{q} \right)^p,
\end{flalign}

\noindent for each $2 \le p \le (\log N)^{\log \log N}$. Denote $s_i = \mathbb{E} [|X_i|^2]$ for each $1 \le i \le N$. 

Let $\{ R_i \}_{1 \le i \le N}$ and $\{ R_{ij} \}_{1 \le i, j \le N}$ be sequences of real numbers. There exists a constant $\nu = \nu ( C, C' ) > 0$ (dependent on only $C$ and $C'$) such that 
\begin{flalign}
\label{suma2}
& \mathbb{P} \left[ \left| \displaystyle\sum_{j = 1}^N R_j X_j \right| \ge (\log N)^{\xi} \left( (C' N^{- \delta} + q^{-1}) \displaystyle\max_{1 \le i \le N} |R_i| + \left( \displaystyle\frac{1}{N} \displaystyle\sum_{j = 1}^N |R_j|^2 \right)^{1 / 2} \right)  \right] \nonumber \\
& \qquad \qquad \qquad \qquad \qquad \qquad \qquad \qquad \qquad \qquad \qquad \qquad \qquad \qquad \le \exp \big( - \nu (\log N)^{\xi} \big); 
\end{flalign}

\begin{flalign}
\label{bi2}
& \mathbb{P} \left[ \left| \displaystyle\sum_{j = 1}^N \big( |X_j|^2 - s_j \big) R_j \right| \ge (\log N)^{\xi} (5 C'^2 N^{- \delta } + q^{-1})  \displaystyle\max_{1 \le i \le N} |R_i|  \right] \le \exp \big( - \nu (\log N)^{\xi} \big);	
\end{flalign}

\begin{flalign}
\label{bij2}
& \mathbb{P} \left[ \left| \displaystyle\sum_{1 \le i \ne j \le N } X_i R_{ij} X_j \right| \ge (\log N)^{2 \xi} \left( (5 C'^2 N^{- \delta} + q^{-1}) \displaystyle\max_{1 \le i \ne j \le N} |R_{ij}| + \left( \displaystyle\frac{1}{N^2} \displaystyle\sum_{1 \le i \ne j \le N} |R_{ij}|^2 \right)^{1 / 2} \right)  \right] \nonumber \\
& \qquad \qquad \qquad \qquad \qquad \qquad \qquad \qquad \qquad \qquad \qquad \qquad \qquad \qquad \le \exp \big( - \nu (\log N)^{\xi} \big); 
\end{flalign}

\begin{flalign}
\label{abij2}
& \mathbb{P} \left[ \left| \displaystyle\sum_{1 \le i, j \le N } X_i R_{ij} Y_j \right| \ge (\log N)^{2 \xi} \left( (5 C'^2 N^{- \delta} + 2 q^{-1}) \displaystyle\max_{1 \le i, j \le N} |R_{ij}| + \left( \displaystyle\frac{1}{N^2} \displaystyle\sum_{1 \le i, j \le N} |R_{ij}|^2 \right)^{1 / 2} \right)  \right] \nonumber \\
& \qquad \qquad \qquad \qquad \qquad \qquad \qquad \qquad \qquad \qquad \qquad \qquad \qquad \qquad \le \exp \big( - \nu (\log N)^{\xi} \big); 
\end{flalign}

\noindent for all $2 \le \xi \le \log \log N$.

\end{cor}

\begin{proof}

The proof will follow from centering the $X_j$ and then applying Proposition \ref{largeprobability1}. To that end, let $m_j = \mathbb{E} [X_j]$, for each $j \in [1, N]$; by \eqref{smallmomentsadelta}, we have that $|m_j| \le C' N^{-1 - \delta}$. 

Denote $\widetilde{X}_i = X_i - m_j$. Then, the $\widetilde{X}_i$ are centered and satisfy 
\begin{flalign*}
\mathbb{E} \big[ |\widetilde{X}_i|^p \big] \le 2^p \Big( \mathbb{E} \big[ |X_i|^p \big] + |m_i|^p \Big) \le \displaystyle\frac{q^2}{N} \left( \displaystyle\frac{2 C}{q} \right)^p + \left( \displaystyle\frac{2 C'}{N^{1 + \delta}} \right)^p \le \displaystyle\frac{q^2}{N} \left( \displaystyle\frac{2 (C + C')}{q} \right)^p, 
\end{flalign*}

\noindent for any $2 \le p \le (\log N)^{\log \log N}$; in the first estimate, we used \eqref{smallmomentsadelta}, and in the last estimate we used the fact that $q, N > 1$. 

Thus, Proposition \ref{largeprobability1} applies to the centered random variables $\widetilde{X}_i$; the estimates \eqref{suma2}, \eqref{bi2}, \eqref{bij2}, and \eqref{abij2} will follow from \eqref{suma}, \eqref{bi}, \eqref{bij}, and \eqref{abij} respectively. 

As an example, we only establish \eqref{bij2}; the proofs of the other estimates are very similar and thus omitted. To that end, we apply \eqref{suma} and \eqref{bij} to obtain that 
\begin{flalign}
\label{sumbi2} \mathbb{P} \left[ \bigg| \displaystyle\sum_{j = 1}^N R_{ij} \widetilde{X}_j \bigg| \ge (\log N)^{\xi} (q^{-1} + 1)  \displaystyle\max_{1 \le i, j \le N} |R_{ij}|  \right] \le \exp \big( - \widetilde{\nu} (\log N)^{\xi} \big), 
\end{flalign}

\noindent for each $i \in [1, N]$, and 
\begin{flalign}
\label{sumbijtilde} 
\begin{aligned}
 \mathbb{P} \Bigg[ \bigg| \displaystyle\sum_{1 \le i \ne j \le N} \widetilde{X}_i R_{ij} \widetilde{X}_j \bigg| \ge (\log N)^{\xi} \bigg( q^{-1} \displaystyle\max_{1 \le j \le N} |R_{ij}|  & + \Big( \displaystyle\frac{1}{N^2} \displaystyle\sum_{1 \le i \ne j \le N}^N |R_{ij}|^2 \Big)^{1 / 2} \bigg)  \Bigg] \\
 & \le \exp \big( - \widetilde{\nu} (\log N)^{\xi} \big), 
\end{aligned}
\end{flalign}
 
\noindent where we set $\widetilde{\nu} = \widetilde{\nu} (C, C') = \widetilde{\nu} \big( 2 (C + C') \big)$. 

Furthermore, observe that 
\begin{flalign}
\label{atildea}
\begin{aligned} 
& \left| \displaystyle\sum_{1 \le i \ne j \le N } X_i R_{ij} X_j  - \displaystyle\sum_{1 \le i \ne j \le N} \widetilde{X}_i R_{ij} \widetilde{X}_j \right| \\
& \quad \qquad \le \displaystyle\sum_{j \ne i} m_j \left| \displaystyle\sum_{i = 1}^N \widetilde{X}_i R_{ij} \right| + \displaystyle\sum_{i \ne j} m_i \left| \displaystyle\sum_{j = 1}^N \widetilde{X}_j R_{ij} \right| + \left| \displaystyle\sum_{1 \le i \ne j \le N} m_i m_j R_{ij} \right| \\
& \quad \qquad \le C' N^{-\delta} \displaystyle\max_{1 \le i \le N} \left| \displaystyle\sum_{i = 1}^N \widetilde{X}_i R_{ij} \right| + C' N^{-\delta} \displaystyle\max_{1 \le j \le N} \left| \displaystyle\sum_{i = 1}^N \widetilde{X}_j R_{ij} \right| + C'^2 N^{-2 \delta} \displaystyle\max_{1 \le i, j \le N} |R_{ij}|, 
\end{aligned} 
\end{flalign}

\noindent where we have used the fact that $|m_i| < C' N^{-1 - \delta}$ for each $i \in [1, N]$. Thus, applying \eqref{sumbi2} for all $i \in [1, N]$, \eqref{sumbijtilde}, \eqref{atildea}, and a union estimate yields 
\begin{flalign}
\label{sumbijtilde2} 
\begin{aligned}
 \mathbb{P} \Bigg[ \bigg| \displaystyle\sum_{1 \le i \ne j \le N} \widetilde{X}_i R_{ij} \widetilde{X}_j \bigg| \ge (\log N)^{2 \xi} \bigg( \gamma \displaystyle\max_{1 \le j \le N} |R_{ij}|  & + \Big( \displaystyle\frac{1}{N^2} \displaystyle\sum_{1 \le i \ne j \le N}^N |R_{ij}|^2 \Big)^{1 / 2} \bigg)  \Bigg] \\
 & \le (2 N + 1) \exp \big( - \widetilde{\nu} (\log N)^{\xi} \big), 
\end{aligned}
\end{flalign} 

\noindent where 
\begin{flalign*}
\gamma = q^{-1} + 2 C' N^{-\delta} (q^{-1} + 1) + C'^2 N^{-2 \delta} \le q^{-1} + 5 C'^2 N^{-\delta}. 
\end{flalign*}

\noindent In the second estimate above, we used the facts that $q, C', N \ge 1$. Now select $\nu = \nu (C, C') = \widetilde{\nu} / 2$, so that $(2 N + 1) \exp \big( - \widetilde{\nu} (\log N)^{\xi} \big) \le \exp \big( - \nu (\log N)^{\xi} \big)$. Then, \eqref{bij2} follows from \eqref{sumbijtilde}. 
\end{proof}

\section{Continuity of \texorpdfstring{$G_{ij} (z)$}{}}

\label{etasmaller}

In this section we establish continuity estimates on the entries of the resolvent, which allowed us to proceed with the multiscale argument in Section \ref{LawSmallNotDeviant} and Section \ref{LawSmallDeviant}.

\begin{lem}

\label{etadecrease}

Fix $E \in \mathbb{R}$; $\eta, \eta' \in \mathbb{R}_{> 0}$; $N \in \mathbb{Z}_{> 0}$; and an $N \times N$ matrix $\textbf{\emph{H}}$. Denote $z = E + \textbf{\emph{i}} \eta$, $z' = E + \textbf{\emph{i}} (\eta + \eta')$, $G (z) = (\textbf{\emph{H}} - z)^{-1} = \{ G_{jk} \}$, and $G' (z) = ( \textbf{\emph{H}} - z')^{-1} = \{ G_{jk}' \}$. 

Fixing $j, k \in [1, N]$, we have that 
\begin{flalign}
\label{gjkzgjk}
\big| G_{jk}' - G_{jk}  \big| \le \displaystyle\frac{\eta'}{2 \eta} \Big( \big| \Im G_{jj}' \big| + \big| \Im G_{kk} 	\big| \Big). 
\end{flalign}
\end{lem}

\begin{proof}
Applying the resolvent identity \eqref{resolvent} with $A = H - z' \Id$ and $B = H - z \Id$, and comparing $(j, k)$-entries, we find that 
\begin{flalign*}
G_{jk}' - G_{jk}  = - \textbf{i} \eta' \displaystyle\sum_{i = 1}^N G_{ji}' G_{ik}. 
\end{flalign*}

\noindent Thus, 
\begin{flalign*}
\big| G_{jk}' - G_{jk}  \big| = \eta' \left| \displaystyle\sum_{i = 1}^N G_{ji}' G_{ik} \right| & \le \eta' \left( \displaystyle\sum_{i = 1}^N \big| G_{ji}' \big|^2 \right)^{1 / 2} \left( \displaystyle\sum_{i = 1}^N \big| G_{ik} \big|^2 \right)^{1 / 2} \\
&= \displaystyle\frac{\eta' \big( \Im G_{jj}' \Im G_{jj} \big)^{1 / 2}}{\big( \eta ( \eta + \eta') \big)^{1 / 2}} \le \displaystyle\frac{\eta' }{2 \eta} \Big( \big| \Im G_{jj}' \big| + \big| \Im G_{jj} \big| \Big),
\end{flalign*} 

\noindent where the third identity was deduced from Ward's identity \eqref{sumgij}; this implies the lemma. 
\end{proof}

\begin{cor}

\label{gminimummaximum}

Adopt the notation of Lemma \ref{etadecrease}. Then, 
\begin{flalign}
\label{gjjzgjj}
\displaystyle\frac{\min \big\{ |G_{jj}'|, |G_{jj}| \big\} }{\max \big\{ |G_{jj}'|, G_{jj}| \big\}} > 1 - \displaystyle\frac{\eta'}{\eta}. 
\end{flalign}
\end{cor}

\begin{proof}
Let $a = |G_{jj} (e + \textbf{i} \eta + \textbf{i} \eta')|$ and $b = |G_{jj} (e + \textbf{i} \eta)|$, and assume that $a \ge b$; the case $b > a$ is entirely analogous. Lemma \ref{etadecrease} applied with $j = k$ yields $a - b < (a + b) \eta' / 2 \eta$. Therefore,
\begin{flalign*}
\displaystyle\frac{b}{a} & >  \displaystyle\frac{2 \eta - \eta'}{2 \eta + \eta'} = 1 - \displaystyle\frac{2 \eta'}{2 \eta + \eta'} > 1 - \displaystyle\frac{\eta'}{\eta}, 
\end{flalign*}

\noindent from which we deduce the corollary. 
\end{proof}

\end{document}